\documentclass{article}

\usepackage[utf8]{inputenc} 
\usepackage{amsfonts}
\usepackage{amsmath}
\usepackage{amsthm} % not for journal version
\usepackage{amssymb}
\usepackage{color}
\usepackage{graphicx}
\usepackage{setspace}
\usepackage{enumerate}
\usepackage{booktabs}
\usepackage{pstricks}
\usepackage{url}
\usepackage{hyperref}
\usepackage{mathtools}
\usepackage{cancel}
\usepackage[ruled, boxed,  linesnumbered, norelsize]{algorithm2e}
\usepackage[normalem]{ulem}
\usepackage{multicol}
\hypersetup{
	colorlinks,
	linkcolor={red!50!black},
	citecolor={blue!50!black},
	urlcolor={blue!80!black}
}

\newcommand{\reInt}{\mathrm{ri}\,}

\newcommand{\lineality}{\mathrm{lin}\,}

\newcommand{\closure}{\mathrm{cl}\,}
\newcommand{\dirCone}{\mathrm{dir}\,}

\newcommand{\spanVec}{\mathrm{span}\,}
\newcommand{\norm}[1]{\lVert{#1}\rVert}
\newcommand{\inProd}[2]{\langle #1 , #2 \rangle }

\newcommand{\intOracle}{\hyperref[oracle:int]{\mathcal{O}_{\text{int}}}}
\newcommand{\feasP}{\mathcal{F}_{\text{P }}}
\newcommand{\feasD}{\mathcal{F}_{\text{D }}}
\newcommand{\feasS}{\mathcal{F}_{\text{D }}^S}
\newcommand{\PSDcone}[1]{{\mathcal{S}^{#1}_+}}
\newcommand{\PDcone}[1]{{\mathcal{S}^{#1}_{++}}}		
\newcommand{\tanSpace}[2]{ {T_{#1}#2}}			 
\newcommand{\tanCone}[2]{ \mathrm{tan}\,({#1},{#2})}	 
\newcommand{\minFaceD}{ {\mathcal{F}_{\min}^{\text{D}}}}
\newcommand{\minFaceP}{ {\mathcal{F}_{\min}^{\text{P}}}}
\newcommand{\stdMap}{ {\mathcal{A}}}

\newcommand{\stdCone}{ {\mathcal{K}}}
\newcommand{\stdSpace}{ \mathcal{L}}
\newcommand{\stdFace}{ \mathcal{F}}

\newcommand{\stdInt}{ {e}}
\newcommand{\pOpt}{ {\theta _{\rm P}}}
\newcommand{\dOpt}{ {\theta _{\rm D}}}
\newcommand{\matRank}{{\mathrm{ rank } \,}}	
\newcommand{\matRange}{{\mathrm{ range } \,}}	

\newcommand{\T}{*}
\newcommand{\Tr}{\top\hspace{-1pt}}   
\newcommand{\opt}[1]{ {\theta _{#1}}}
\newcommand{\minFacePoint}[2]{ {\mathcal{F}(#1,#2)}}
\newcommand{\dist}{ {\mathrm{dist}\,}}	
\renewcommand{\Re}{\mathbb{R}}    
\newcommand{\ambSpace}{ \mathcal{E} }
\renewcommand{\S}{\mathcal{S}}        
\newcommand{\minFace}[1]{ {\mathcal{F}_{\min}^{#1}}}            
    
\newcommand{\zM}[1]{0_{#1}}  
\newcommand{\relmiddle}[1]{\mathrel{}\middle#1\mathrel{}}
\title{Solving SDP Completely with an Interior Point Oracle}

\author{Bruno F. Louren\c{c}o
\thanks{
Department of Statistical Inference and Mathematics, Institute of Statistical Mathematics, 10-3 Midori-cho, Tachikawa, Tokyo 190-8562, Japan.
(\texttt{bruno@ism.ac.jp})}
        \and
        Masakazu Muramatsu\thanks{
                     Department of Computer and Network Engineering, The University of Electro-Communications 1-5-1 Chofugaoka, Chofu-shi, Tokyo, 182-8585 Japan. (E-mail: \texttt{MasakazuMuramatsu@uec.ac.jp})
                  }
         \and       
                Takashi Tsuchiya
\thanks{
National Graduate Institute for Policy Studies 7-22-1 Roppongi, Minato-ku, Tokyo 106-8677, Japan. (E-mail: \texttt{tsuchiya@grips.ac.jp}) \newline 
B.~F.~Louren\c{c}o is partially supported by the
JSPS Grant-in-Aid for Young Scientists 19K20217.
M.~Muramatsu is   partially supported by the
JSPS Grant-in-Aid for Scientific Research (B)26280005, (C)17K00031 and (B)20H04145. 
T.~Tsuchiya is   partially supported by the
JSPS Grant-in-Aid for Scientific Research (B)15H02968. 
M. Muramatsu and T. Tsuchiya are also supported in part by the JSPS Grant-in-Aid for Scientific Research (B)24310112 and (C) 26330025.
B.~F.~Louren\c{c}o and T.~Tsuchiya are  also partially supported by the JSPS Grant-in-Aid for Scientific Research (B)18H03206.
                  }
        }

\date{July 2015 (Revised: November 2020)}

\newtheorem{definition}{Definition}
\newtheorem{lemma}[definition]{Lemma}
\newtheorem{proposition}[definition]{Proposition}
\newtheorem*{proposition*}{Proposition}

\newtheorem{theorem}[definition]{Theorem}
\newtheorem*{remark}{Remark}

\begin{document}
\renewcommand{\topfraction}{0.9}
\maketitle

\begin{abstract}
We suppose the existence of an oracle which  solves
any semidefinite programming  (SDP) problem  satisfying strong feasibility (i.e., Slater's condition) simultaneously at its primal and dual sides. 
We note that such an oracle might not be able to
directly solve general SDPs
even after certain regularization schemes are applied.
In this work we fill this gap and show 
how to use such an oracle to {``completely solve''} an arbitrary SDP.
Completely solving entails, for example, 
distinguishing between weak/strong feasibility/infeasibility and 
detecting when the optimal value is attained or not.
We will employ several tools, including 
a variant of facial reduction where all auxiliary problems are ensured 
to satisfy strong feasibility at all sides. Our 
main technical innovation, however, is an analysis of \emph{double facial reduction}, which is the process of applying facial 
reduction twice: first to the original problem and then once more to  
the dual of the regularized problem obtained during the first run.
Although our discussion is focused on semidefinite programming, the majority 
of the results are proved for general convex cones.

\noindent \textbf{Keywords:} double facial reduction, facial reduction,	semidefinite programming, feasibility problem.
%Here, we use the term \textit{to completely solve} an SDP to mean a
%scheme which works in the following way;
%given an SDP, the scheme checks whether it is feasible or not, and
%whenever feasible,
%computes its optimal value, and if the optimal value is attained, obtains an optimal solution. If the optimal value is 
%not attained, computes a feasible solution whose objective value is arbitrarily close to the optimal value. 
%Moreover, if the original SDP is infeasible, it distinguishes
%between strong and weak infeasibility, and in the case of weak infeasibility,
%can compute a point in the corresponding affine space whose distance
%to the positive semidefinite cone is less than an arbitrary small
%given positive number.
\end{abstract}

\section{Introduction}
Consider the following pair of primal and dual linear semidefinite programs (SDPs).
\begin{multicols}{2}
	\noindent\begin{align}
\underset{x}{\inf} & \quad \inProd{c}{x} \label{eq:primal}\tag{SDP-P}\\ 
\mbox{subject to} & \quad \stdMap x = b \nonumber \\ 
&\quad x \in \PSDcone{n} \nonumber
	\end{align}
	\noindent
	\begin{align}
\underset{y}{\sup} & \quad \inProd{b}{y} \label{eq:dual} \tag{SDP-D} \\ 
\mbox{subject to} & \quad c - \stdMap ^\T y \in \PSDcone{n} , \nonumber
	\end{align}	
\end{multicols}
\vspace*{-1\baselineskip} 
\noindent where $\PSDcone{n}$ denotes the cone of $n\times n$ symmetric positive semidefinite matrices, which is contained in $\S^n$ (the space of $n \times n$ real symmetric matrices). Here $\stdMap: \S^n \to \Re^m$ is a linear map, $b \in \Re^m$, $c \in \S^n$ and $\stdMap^\T$ denotes the adjoint map of $\stdMap$. 
In addition, we assume that both $\Re^m$ and $\S^n$ are equipped with the usual Euclidean product and the trace inner product, respectively. We will use the same symbol $\inProd{\cdot}{\cdot}$ to express both inner products. 
%We will denote the pair \eqref{eq:primal}, \eqref{eq:dual} by $(\PSDcone{n},\stdMap, b,c)$.

We start with the following observation.
\begin{quote}
\emph{To the best of our knowledge, all (or almost all) 
	methods for solving SDP require some kind of assumption on the problems 
\eqref{eq:primal} and \eqref{eq:dual} in order for its convergence theory to work. In addition, there seems to be no method that can solve arbitrary SDP instances and distinguish between all kinds of ill-behaviour that can happen in semidefinite programming.} 
\end{quote}

%The previous observation contains both surprising and non-surprising aspects. 
On one hand, it might seem almost obvious that some condition must be imposed on the pair \eqref{eq:primal} and \eqref{eq:dual} in order to get meaningful convergence results, which is the pattern established in classical nonlinear programming since the early days, where many convergence results required some constraint qualification to hold.
On the other hand, for linear programming we have the simplex method, which, at least in theory, is able to solve any linear program and detect all possible outcomes including infeasibility and unboundedness.
Given that semidefinite programming is one of the most natural extensions of linear programming, it is somewhat disappointing that, as of this writing, we still cannot claim to be able to solve SDPs in the same thorough way.

However, there are indeed classes of SDPs that we can reasonably claim that are solvable by current methods. One of those classes consists of 
the SDPs for which strong feasibility (also called \emph{Slater's condition}) holds 
at both \eqref{eq:primal} and \eqref{eq:dual}.
They are solvable, for instance, by using interior point methods  \cite{NN94,ali-95}. 
In this paper, we aim to show the following result.
\begin{quote}
Suppose we have access to an oracle that can solve any SDP instance, provided that the instance is both primal and dual strongly feasible. 
Then, we can ``completely solve'' \emph{any} SDP instance  with polynomially (in $n$) many calls	to this oracle.
\end{quote}
Later in Section~\ref{sec:sum} we state our results more precisely including a suitable definition of ``completely solving'' but, for now, we give some background to our research and results.

\subsection{Background and previous works}
%Interior point methods (IPMs) \cite{NN94,ali-95} are  one of the standard methods for solving
%semidefinite programs \eqref{eq:primal} and \eqref{eq:dual}. However, in order to function properly, they require 
%that the problem at hand satisfy certain regularity conditions which may 
%fail to be satisfied and this leads to numerical difficulties. The usual 
%requirement is that there is a primal feasible solution $x$ such that 
%$x$ is positive definite \emph{and} that there is a dual feasible solution $y$ 
%such that the corresponding slack $s = c -\stdMap^{\T}y$ is positive definite. In other words, 
%Slater's condition must hold for both \eqref{eq:primal} and \eqref{eq:dual}.
We now discuss briefly how strong feasibility is connected with some research trends in 
continuous optimization.

%In this article, we suppose the existence of an idealized machine that is 
%able to solve any SDP having primal and dual interior feasible points and discuss whether 
%it could be used to \emph{completely solve} other SDPs which might not necessarily 
%satisfy Slater's condition.
%In this paper, we define \emph{complete solvability} as follows.
%\begin{definition}[Completely solving an optimization problem]\label{def:comp_solve}
%	An algorithm, procedure or a scheme is said to 
%	\emph{completely solve} an SDP instance, 
%	if it achieves the following goals.
%	\begin{enumerate}
%		\item It decides whether the SDP is feasible or not.
%		\item When the SDP is feasible, it computes the optimal value.
%		If the optimal value is attained, it computes an optimal solution of maximal rank.
%		If the optimal value is not attained, given arbitrary small $\epsilon>0$,
%		it can compute an $\epsilon$-optimal solution.
%		\item When the SDP is infeasible, it distinguishes between strong infeasibility and weak infeasibility. If the instance is strongly infeasible, it computes a certificate of infeasibility.
%		If the SDP is weakly infeasible, then it find a matrix that is arbitrarily close to feasibility (this will be made precise later). \end{enumerate}
%\end{definition}

\begin{itemize}
	\item \emph{Interior point algorithms and software.} Most modern IPM softwares \cite{Sturm1999, SDPAReport,sdpt3} including the commercial solver Mosek do not require explicit knowledge 
	of  an interior feasible point beforehand. SeDuMi \cite{Sturm1999}, for instance, 
 transforms a standard form problem into the so-called homogeneous self-dual formulation, which 
 has a trivial starting point. SDPA \cite{SDPAReport} and SDPT3 \cite{sdpt3} use an infeasible interior point method. The fact 
 that these methods can work without explicit knowledge of an interior feasible point, does 
 not  mean that they do \emph{not require the existence of an interior feasible point}. Quite the opposite, 
 the absence of interior feasible points may introduce theoretical and numerical difficulties in recovering 
 a solution for the original problem. Also, detection of infeasibility is a complicated task. 
 Some interior point methods, such as the one discussed in \cite{nesterov_infeasible} by Nesterov, Todd and Ye,   
 are able to obtain a certificate of infeasibility if the problem is dual or primal strongly infeasible, but the 
 situation is less clear in the presence of the so-called \emph{weak infeasibility} \cite{lourenco_muramatsu_tsuchiya}. These issues are also discussed  by Karimi and Tun\c{c}el in~\cite{KT19_2}, in the context of their software \emph{DDS (Domain-Driven Solver)}~\cite{KT19,KT20}. 
	\item \emph{Ramana's extended dual.} %In general, there could be a nonzero duality gap between 
	%\eqref{eq:primal} and \eqref{eq:dual}, i.e., the optimal values of \eqref{eq:primal} and \eqref{eq:dual} might differ. To address this,	
	Ramana \cite{Ramana95anexact,ramana_strong_1997} developed an alternative duality 
	theory 	for \eqref{eq:dual} beyond the usual Lagrangian duality. Remarkable features 
	of Ramana's dual include the fact that whenever the optimal value of an SDP is finite, 
	its Ramana's dual attains the same optimal value without any further assumptions.
	However, Ramana's dual 
	is not necessarily suitable to be solved by IPMs due to the fact that it does not ensure the existence 
	of interior feasible points at both sides.
	
	\item \emph{Facial reduction.} 	
	Denote by $\feasS$, the set of 
	feasible slacks of \eqref{eq:dual}, i.e.,
	\[
	\feasS = \{s \in \PSDcone{n} \mid \exists y, s = c-\stdMap^\T y\}.
	\]
	Let $\minFaceD$ be the minimal face of $\PSDcone{n}$ which 
	contains $\feasS$. If we replace $\PSDcone{n}$ by 
	$\minFaceD$ in \eqref{eq:dual}, then the new \eqref{eq:dual} will be strongly feasible, because $\minFaceD$ is characterized as the unique face for which $\feasS$ intersects the relative interior of 
	$\minFaceD$.
	The process of finding $\minFaceD$ is called \emph{facial reduction} \cite{article_waki_muramatsu,pataki_strong_2013} and 
	was developed originally by Borwein and Wolkowicz \cite{Borwein1981495,borwein_facial_1981} for convex optimization with conic constraints.
	Descriptions for 
	the conic linear programming case have appeared, for instance, in Pataki \cite{pataki_strong_2013} and 	in Waki and Muramatsu~\cite{article_waki_muramatsu}.
	 We will overview facial reduction in more detail in Section~\ref{sec:fr}.
	
	However, one important point 
	is that facial reduction only guarantees that  strong feasibility is satisfied at one side of the problem. So, again, even this regularized problem might fail to have 	interior solutions at both primal and dual sides.
	
%	One of the problems with facial reduction is that  
%	even if we find some face for which the reduced system is 
%	strongly feasible, there is no guarantee that the corresponding primal will also 
%	be strongly feasible. That is, facial reduction only restores strong feasibility 
	%at one of the sides of the problem. So the unfortunate fact remains that even after performing  facial %reduction one may 
	%end up with a problem that is still not suitable to be solved by primal-dual interior 
	%point methods.
	
	We remark that strong feasibility at only one of the sides of the problem can also be a source 
	of numerical difficulties. In Section 2 of \cite{WNM12}, Waki, Nakata and Muramatsu shows an instance satisfying strong feasibility at the primal side, but not at the dual side. 
	Its optimal value is zero but 
	both SDPA \cite{SDPAReport} and SeDuMi \cite{Sturm1999} output $1$ instead.
	
	\item \emph{Algebraic approaches}. Henrion, Naldi and Din described an algebraic approach to the problem of obtaining a feasible 
	solution to \eqref{eq:dual}, see \cite{HND16, HND19}. Interesting features of their algorithm include, among others, the fact that their algorithm is implementable in exact arithmetic (as opposed to floating point arithmetic) and that, as long as \eqref{eq:dual} satisfies certain genericity assumptions,  the algorithm can find solutions even in degenerate cases when strong feasibility is not satisfied. In addition, when a solution is found, a so-called rational parametrization is provided for it.
	A description of their package \emph{Spectra} is given in \cite{HND19}. Drawbacks, however, include that in most cases, only small instances can be solved, see Section~1 of~\cite{HND19}. Furthermore, 	optimization problems cannot be solved directly. 

\end{itemize}
There is a growing body of research 
aimed at understanding SDPs and conic linear programs having pathological behaviours such as nonzero 
duality gaps and weak infeasibility. Here we will mention a few of them.
A problem is called \emph{weakly infeasible} if there is no 
feasible solution but the distance between the underlying affine space and the cone under consideration 
is zero. Weak infeasibility is known to be very hard to detect numerically, see for instance 
P\'{o}lik and Terlaky \cite{polik_new_2009}.
In \cite{waki_how_2012}, Waki showed that weakly infeasible problems sometimes 
arise from polynomial optimization. 
There is also a discussion on weak infeasibility semidefinite programming and second order cone programming in \cite{lourenco_muramatsu_tsuchiya} and \cite{LMT16}, respectively. 
Some of the results in \cite{lourenco_muramatsu_tsuchiya} were
generalized to arbitrary closed convex cones by Liu and Pataki, see \cite{LP15} for 
more details. See also \cite{LMT18}, where some results of \cite{LP15} on weakly infeasible problems are sharpened when the polyhedral faces of the underlying cone are taken into account.

It is hard to obtain finite certificates of infeasibility 
for SDPs, because there is no straightforward extension of Farkas' Lemma for non-polyhedral cones.
Another issue is that, as shown by Porkolab and Khachiyan \cite{PK97}, even a reasonably sized SDP may only have exponentially small feasible solutions, which makes it hard to detect feasibility/infeasibility numerically.

Nevertheless, the first finite infeasibility certificate was obtained by Ramana in \cite{Ramana95anexact} using his extended duality theory.
Since 
then, Sturm mentioned the possibility of obtaining a finite certificate for 
infeasibility by using the directions produced in his regularization procedure, see 
page 1243 of \cite{sturm_error_2000}. More recently,  
Liu and Pataki have also obtained finite certificates through elementary reformulations \cite{LP15_2}. 
Interestingly, Klep and Schweighofer \cite{klep_exact_2013} also obtained 
certificates through a completely different approach using 
tools from real algebraic geometry.
As we move from SDPs to conic linear programs over arbitrary cones, facial reduction seems to one of the few approaches that can provide finite certificates of infeasibility see, for example, 
\cite{LP15}.

In \cite{WNM12}, Waki, Nakata and Muramatsu discussed
 SDP instances for which known solvers failed to obtain the correct answer and in one case, 
this happened even though  the problem had an interior feasible point at the primal side.
In \cite{pataki_bad_sdps,P17}, Pataki gave a definition of ``bad behaviour'' and 
showed that  all SDPs in that class can be put in the same form, after performing an elementary 
reformulation. A discussion on duality gaps and many interesting examples of pathological 
SDPs are given by Tun\c{c}el  and Wolkowicz in \cite{TW12}.
Pataki has recently provided an extensive study of duality gaps in semidefinite programming in \cite{P18}. He showed, for instance, that all SDPs with positive duality gap and $m = 2$ (i.e., the dual problem has two variables) have a common reformulation, see Theorem~1 therein. 

\subsection{Summary and contributions of this work}\label{sec:sum}
%In general, SDPs can be  plagued by all sorts of ill-behavior and this stands 
%in sharp contrast to the relatively nice nature of SDPs where both the primal and dual are 
%strongly feasible. In order to capture this distinction, 
We consider the following oracle, which we will denote by $\intOracle$.

\begin{algorithm}[H]\NoCaptionOfAlgo
	%\caption{The interior point oracle $\intOracle$ for SDPs }\SetAlgoRefName{Oracle}
	\SetAlgoRefName{$O_{test}$}
	\caption{The interior point oracle $\intOracle$ for SDPs }\label{oracle:int}
	\DontPrintSemicolon
	\SetKwInOut{Input}{Input}\SetKwInOut{Output}{Output}
	\Input{The problem data: $\stdMap, b,c$. Both \eqref{eq:primal} and \eqref{eq:dual} must be strongly feasible.
%		\begin{enumerate}
%			\item The problem data: $\PSDcone{n}, \stdMap, b,c$ 
%			\item A dual relative interior solution: $y$ such that $c - \stdMap^\T y \in \reInt \PSDcone{n}$,
%			\item A primal relative interior solution: $x$ such that $\stdMap x = b$, $x \in \reInt \PSDcone{n}$.
%		\end{enumerate}
	}
	\Output{A zero duality gap optimal solution pair $x^*$, $y^*$. That is,
		$x^*$ and $y^*$ satisfy
		\begin{align*}
		\inProd{c}{x^*} & = \inProd{b}{y^*}\\
		c - \stdMap^\T y^* & \in \PSDcone{n} \\ 
		\stdMap x^* & = b\\
		x^* & \in \PSDcone{n}.
		\end{align*}
	}
\end{algorithm}%\setcounter{algocf}{0}%Reset the algorithm environment counter so that the other Algorithms start at 1. 
%We will refer to the interior point oracle by $\intOracle$.
We can regard $\intOracle$ as a machine 
running an idealized version of either the homogeneous self-dual embedding method \cite{PS98, KTR00,LSZ00},
an infeasible interior point method \cite{nesterov_infeasible}, the ellipsoid method or even 
an augmented Lagrangian method. 
An important point is that no assumption is made on the inner workings of the oracle. 
Now we are ready to define the meaning of \emph{completely solving an SDP}.

\begin{definition}[Completely solving \eqref{eq:dual}]\label{def:comp_solve}
	An algorithm, procedure or a scheme is said to 
	\emph{completely solve} \eqref{eq:dual}, 
	if it receives as input $\stdMap,b,c$ and $\epsilon > 0$ and achieves the following goals.
	\begin{enumerate}[$(a)$]
		\item It decides whether the \eqref{eq:dual} is feasible or not.
		\item When \eqref{eq:dual} is feasible, it computes the optimal value.
		If the optimal value is attained, it computes an optimal solution. %of maximal rank.
		If the optimal value is finite but not attained,	it computes an $\epsilon$-optimal solution. If \eqref{eq:dual} is unbounded (i.e., $\dOpt = +\infty$) this must be detected.
		\item When \eqref{eq:dual} is infeasible, it correctly distinguishes between strong infeasibility and weak infeasibility. 
%		If the instance is strongly infeasible, it computes a certificate of infeasibility.
%		
		If  \eqref{eq:dual} is weakly infeasible, then it finds a matrix that is arbitrarily close to feasibility (this will be made precise later). \end{enumerate}
\end{definition}

Although we focus on semidefinite programming, the majority 
of our results will be proved for general conic linear programs (CLPs).
Keeping this remark in mind, we now state our contributions in this paper.
\begin{enumerate}
	\item We present an algorithm for completely solving 
	general CLPs, provided that we can solve certain auxiliary problems that are strongly feasible, see Section~\ref{sec:dfr} and Algorithm~\ref{alg:comp}. In particular, 
	we will show that an \emph{arbitrary} SDP can be {completely solved} by $O(n)$ calls to $\intOracle$. 
	This implies that even 
	though an {arbitrary} SDP may have unfavourable properties, we can always completely solve it in the sense of Definition~\ref{def:comp_solve} if we assume that we are capable of solving instances that are  both primal and dual strongly feasible. An important feature of our approach is that it is 
	\emph{method agnostic} and does not rely in any way on the inner working of $\intOracle$. 
	See Appendix~\ref{app:ex} for an example of applying Algorithm~\ref{alg:comp} to a particularly ill-behaved instance.

	\item  We present a detailed discussion of \emph{double facial reduction} for general conic linear programs, which is the process of applying facial reduction twice: first to an CLP and then, to the dual of the regularized CLP obtained at the first step. 
	
	Through double facial reduction,  whenever the optimal value of \eqref{eq:dual} is finite, we are ensured to obtain a new pair of primal and dual strongly feasible problems and whose 
	common optimal value coincides with the optimal value of \eqref{eq:dual}.

	Although we cannot always recover optimal solutions for \eqref{eq:dual} from this new  pair of problems (after all, \eqref{eq:dual} might not even have optimal solutions in the first place), we will show how it is possible to obtain feasible solutions that are arbitrarily close to optimality, for any desired accuracy, by using the directions that appear when applying facial reduction. See Section~\ref{sec:almost} and Algorithm~\ref{alg:eps} for more details. The discussion on 
	obtaining almost optimal solution leads naturally to an approach 
	for obtaining almost feasible solution for weakly infeasible problems and this is discussed in Section~\ref{sec:inf}.

	\item We present several technical results about facial reduction that we believe might be of independent interest.
	For example, we show how to perform facial reduction by solving 
	auxiliary problems that are ensured to be  both primal and dual strongly feasible, see Lemma~\ref{lemma:red_sdp} and Algorithm~\ref{alg:fra}.

	We also provide a technical result on how the feasibility properties of a problem might change when facial reduction is applied to its dual, see Proposition~\ref{prop:feas_changes} and Theorem~\ref{theo:min_changes}.
\end{enumerate}
We remark that this paper is a thorough  extension and reformulation of an earlier technical report \cite{lourenco_muramatsu_tsuchiya3}, where the results were only proved for semidefinite programming by different techniques.

\subsection{Limitations of this work}\label{sec:lim}
%Although we stand by our theoretical results including the ones on facial reduction and double facial reduction, we must admit 
A limitation of this work is that the algorithm for completely solving conic linear programs (Algorithm~\ref{alg:comp}) is somewhat hypothetical.
This is because, except in very special cases \cite{VY95,W96,M05}, we cannot solve exactly an SDP even if it is primal and dual strongly feasible. Usually, the best we can do is to compute solutions that are approximately feasible and approximately optimal to some specified tolerance $\epsilon > 0$ or, under special circumstances, provide a rational parametrization to the solution set as in \cite{HND16,HND19}. So, strictly speaking, only an approximate version of the oracle $\intOracle$  might be practically implementable.
%Although we hope to provide some theoretical insight, in all likelihood, our approach does not directly lead to a practical algorithm for general problems. 

Missing from our analysis is how to deal with the case where there is 
some imprecision in the answer returned by $\intOracle$.
This is a very complex issue because since regularity conditions might fail, small perturbations in the input data might lead to problems whose optimal values are vastly different. Furthermore, impreciseness  whilst doing facial reduction might lead to a wrong face being computed and  feasible solutions could be inadvertently removed. 

 We believe however, that the analysis of the exact case is an important stepping stone and  we see a similar pattern in many subareas of optimization. For example, for augmented Lagrangian methods, understanding the behavior of the algorithm when subproblems are solved exactly seems to be quite important for getting the larger picture of the algorithm and its convergence analysis, even though, in practice, the subproblems are only approximately solved.
 
We remark that related approaches  by de Klerk, Roos and Terlaky \cite{KTR00} and Permenter, Friberg and Andersen \cite{PFA17} also assume that exact solutions are obtainable. However, numerical experiments are provided in \cite{PFA17} to check how their approach fare under inexactness. 
We provide a detailed comparison between \cite{KTR00, PFA17} and our approach in 
Section~\ref{sec:further}.
\subsection{Structure of this paper}
This paper is organized as follows. Section~\ref{sec:review} discusses the notation used throughout the paper and contains a review of the necessary notions from convex analysis.
Some technical aspects related to the faces of $\PSDcone{n}$ and interior point oracle $\intOracle$ are discussed in Section~\ref{sec:oracle}. Section~\ref{sec:fr} presents 
a facial reduction algorithm that is suitable to be used in conjunction with $\intOracle$. Section~\ref{sec:dfr} discusses double facial reduction and how it can be used to obtain almost optimal solutions and analyze weak infeasibility. 
Section~\ref{sec:comp} contains the description of an algorithm for completely solving a general conic linear program which can be adapted to use $\intOracle$ when the underlying 
cone is $\PSDcone{n}$. Section~\ref{sec:further} contains a discussion on related approaches. Section~\ref{sec:conc} concludes this work.
%
%This work is organized as follows. In Section \ref{sec:review}, we review a 
%few basic notions and also facial reduction. In Section \ref{sec:oracle} we define 
%the interior point oracle and discuss a version of facial reduction for SDPs. In
%Section \ref{sec:hyper}, we recall some properties of hyper feasible partitions and 
%show how they can be constructed with the interior point oracle. Section \ref{sec:opt_value} shows 
%how to compute the optimal value and in Section \ref{sec:recovery}, we discuss how to check 
%if the value is attained or not. Section \ref{sec:glueing} shows a complete list of 
%steps for completely solving \eqref{eq:dual} and in Section \ref{sec:further} we compare our 
%approach with previous works. Section \ref{sec:conc} wraps up this paper.

\section{Preliminary discussion and review of relevant notions}\label{sec:review}
Let $C\subseteq \ambSpace$ be a closed convex set contained in a real finite dimensional space $\ambSpace$. Its relative interior, closure, linear span and 
dimension are denoted  by $\reInt{C}$, $\closure C$, $\spanVec C$ and $\dim C$, respectively. We assume 
that $\ambSpace$ is equipped with some inner product $\inProd{\cdot}{\cdot}$ and we will denote by $C^\perp$ the subspace of $\ambSpace$ which contains the elements orthogonal to $C$ with respect 
to $\inProd{\cdot}{\cdot}$. 
We will denote by $\norm{\cdot}$ the norm induced by $\inProd{\cdot}{\cdot}$. For a pair of sets $C,D\subseteq \ambSpace$ we define the distance between $C$ and $D$ as
\[
\dist(C,D) \coloneqq \inf \{\norm{x-s} \mid x \in C, s \in D \}.
\]
If $x \in \ambSpace$, we will use $\dist(x, C)$ as a shorthand for 
$\dist(\{x\}, C)$.

If $\stdMap$ is a linear map, we will denote its image, kernel and adjoint  by $\matRange \stdMap, \ker \stdMap$ and $\stdMap^\T$, respectively.

For $\stdCone \subseteq \ambSpace$ a closed convex cone, we denote by $\lineality \stdCone$ the \emph{lineality space} of $\stdCone$, i.e., 
\[
\lineality \stdCone \coloneqq \stdCone \cap - \stdCone.
\]
We denote by $\stdCone^*$ the dual cone of $\stdCone$:
\[
\stdCone^* \coloneqq \{x \in \ambSpace \mid \inProd{s}{x} \geq 0, \forall s \in \stdCone \}.
\]
A closed convex subset  $\stdFace$ contained in $\stdCone$ is said to be a 
\emph{face} of $\stdCone$ if 
\[
s,\hat s \in \stdCone, \frac{s+\hat s}{2} \in \stdFace \Rightarrow s, \hat s \in \stdFace.
\]
The \emph{conjugate face} of $\stdFace$ is defined as
\[
\stdFace^{\Delta} \coloneqq \stdCone^*\cap \stdFace^\perp.
\]
Given $x \in \stdCone$, we write $\minFacePoint{x}{\stdCone}$ for the  intersection of all faces of $\stdCone$ containing $x$. $\minFacePoint{x}{\stdCone}$ is the minimal (with respect to inclusion) face of $\stdCone$ containing $x$.

For a given $x \in \stdCone$, we write $\dirCone (x,\stdCone)$ for the \emph{cone of 
	feasible directions} of $\stdCone$ at $x$. This is the set 
\[
\dirCone (x,\stdCone) \coloneqq \{z \in \ambSpace \mid \exists t > 0, x +tz \in \stdCone \}.
\]
The closure of $\dirCone (x,\stdCone)$ is the  \emph{tangent cone} of $\stdCone$ at $x$ and is denoted by $\tanCone{x}{\stdCone}$.
The \emph{tangent space} of $\stdCone$ at $x$ is the lineality space of $\tanCone{x}{\stdCone}$ and is denote by $\tanSpace{x}{\stdCone}$. In summary, we have
\begin{align*}
\tanCone{x}{\stdCone} & \coloneqq \closure \dirCone (x,\stdCone), \\
\tanSpace{x}{\stdCone} & \coloneqq \tanCone{x}{\stdCone}\cap - \tanCone{x}{\stdCone}.
\end{align*}
Some of the relationships between the sets defined so far will be summarized at Lemma~\ref{lemma:aux}.

Although our focus is on semidefinite programming, most of the results will be proved for the following primal and dual pair of general conic linear programs: 
\begin{multicols}{2}
	\noindent\begin{align}
	\underset{x}{\inf} & \quad \inProd{c}{x} \label{eq:c_primal}\tag{Conic-P}\\ 
	\mbox{subject to} & \quad \stdMap x = b \nonumber \\ 
	&\quad x \in \stdCone^* \nonumber
	\end{align}
	\noindent
	\begin{align}
	\underset{y}{\sup} & \quad \inProd{b}{y} \label{eq:c_dual} \tag{Conic-D} \\ 
	\mbox{subject to} & \quad c - \stdMap ^\T y \in \stdCone, \nonumber
	\end{align}	
\end{multicols}
\noindent where $\stdMap:\ambSpace \to \Re^m$ is a linear map, $b \in \Re^m$, $c \in \ambSpace$. Semidefinite programming corresponds to the specific case where $\ambSpace = \S^n$ and $\stdCone= \PSDcone{n}$.

We
will denote by $\pOpt$ and $\dOpt$, the  optimal values of 
 \eqref{eq:c_primal} and \eqref{eq:c_dual} respectively. 
It is understood that $\pOpt = +\infty$ if \eqref{eq:c_primal} is infeasible and 
$\dOpt = -\infty$ if \eqref{eq:c_dual} is infeasible. 
%The pair \eqref{eq:c_primal} and \eqref{eq:c_dual} is said to have \emph{zero duality gap} if $\pOpt = \dOpt$. 
The 
primal and dual feasible regions are defined as follows:
\begin{align*}
\feasP & \coloneqq  \{x \in \stdCone^* \mid \stdMap x = b\},\\
\feasD & \coloneqq \{y \in \Re^m \mid c - \stdMap ^\T y \in \stdCone \},\\
\feasS & \coloneqq \{s \in \stdCone \mid \exists y \in \Re^m, s = c-\stdMap^\T y\} = (c+ \matRange \stdMap^\T) \cap \stdCone.
\end{align*}
If $s \in \ambSpace$ can be written as $s = c - \stdMap^* y$ for some $y$, then $s$ is said to 
be a \emph{dual slack}. Furthermore, if $s \in \feasS$ then $s$ is called a \emph{dual feasible slack}.
The dual 
optimal value $\dOpt$ is said to be \emph{attained} if there is  $y \in \feasD$ such 
that $\inProd{b}{y} = \dOpt$. The  notion of 
primal attainment is analogous. 
We recall the following basic constraint qualification.
\begin{proposition}[Slater]\label{prop:slater}
Consider the pair \eqref{eq:c_primal} and \eqref{eq:c_dual}.
\begin{enumerate}[(i)]
	\item If there exists $x \in (\reInt \stdCone^* ) \cap \feasP$ then $\pOpt = \dOpt$. If, in addition, $\pOpt$ is finite then $\dOpt$ is attained.
	\item If there exists $s \in  (\reInt \stdCone) \cap \feasS$ then $\pOpt = \dOpt$. If, in addition, $\dOpt$ is finite then $\pOpt$ is attained.
\end{enumerate}
\end{proposition}

For the reader's convenience, before we proceed we recall a few  basic facts from convex analysis. We provide references for the items and/or short proofs.
\begin{lemma}\label{lemma:aux}%Todo: remove a few items
	Let $\stdCone\subset \ambSpace$ be a closed convex cone, $\stdInt \in \reInt \stdCone$, 
	$x \in \stdCone$ and $z \in \stdCone^*$. 
	\begin{enumerate}[$(i)$]
%		\item $\stdCone ^{**} = \stdCone$. \label{laux:sd}
%		\item $\lineality \stdCone = \stdCone^{*\perp}$, where $\stdCone^{*\perp}$ is a short-hand 		for $(\stdCone^*)^\perp$. \label{laux:lin}
%		
		\item $\stdCone^\perp = \lineality (\stdCone^*)$. \label{laux:perp}
		\item $x + \stdInt \in \reInt \stdCone$. \label{laux:ri}
		\item There exists $\alpha > 1$ such that $\alpha \stdInt + (1-\alpha)x\in \stdCone$. \label{laux:ri2}
		\item $z \in \stdCone^\perp$ if and only if $\inProd {\stdInt} {z} =  0$. \label{laux:ri3}
%		\item $x \in \reInt \minFacePoint{x}{\stdCone} $. \label{laux:fri}
		\item $\minFacePoint{x}{\stdCone}^{\Delta} = \stdCone^* \cap \{x\}^\perp$. \label{laux:conj}
		\item $(\tanCone{x}{\stdCone})^* = \minFacePoint{x}{\stdCone}^{\Delta}$. \label{laux:tan}
		\item $\tanSpace{x}{\stdCone} = \minFacePoint{x}{\stdCone}^{\Delta \perp}$. \label{laux:ts}
%		\item If $w \in \tanCone{x}{\stdCone}$ then $\lim _{t\to + \infty} \dist(tx + w, \stdCone) = 0$. \label{laux:dist}
	\end{enumerate}
\end{lemma}
\begin{proof}
	\begin{enumerate}[$(i)$]
%		\item This is the  bipolar theorem, see Theorem 14.1 of \cite{rockafellar}.
%		\item If $z \in \lineality \stdCone$, then $\inProd{z}{y} \geq 0$ and $\inProd{-z}{y} \geq 0$, for every 
%		$y \in \stdCone^*$. It follows that $z \in \stdCone^{*\perp}$. Reciprocally, if 
%		$z \in \stdCone^{*\perp}$, then $z \in \stdCone^{**} = \stdCone$, by item $(\ref{laux:sd})$. 
%		Since $\stdCone^{*\perp}$ is a subspace, we have $\stdCone^{*\perp} \subseteq \lineality \stdCone$.
		\item See item (a) of Proposition~2.1 in \cite{Tam85}.
		%First, suppose that $z \in \lineality (\stdCone^*)$.
%		Since 
%		$z, -z \in \stdCone^*$, we conclude that $z \in \stdCone^\perp$.
%		Conversely, suppose that $z \in \stdCone^\perp$.
%		Then, $-z$ belongs to 
%		$\stdCone^\perp$. Since, $\stdCone^\perp$ is contained in 
%		$\stdCone^*$, we conclude that $z,-z \in \stdCone^*$. That is, 
%		$z \in \lineality (\stdCone^*)$.
		\item Since $\stdInt \in \reInt \stdCone$, for any $z \in \stdCone$ we have 
		that all points in the relative interior of the line segment connecting 
		$z$ and $\stdInt$ also belong to the relative interior of $\stdCone$, see Theorem 6.1 of \cite{rockafellar}.
		Since \[x + \stdInt = \frac{1}{2}\stdInt + \frac{1}{2}(2x + \stdInt),\]we have 
		$x + \stdInt \in \reInt \stdCone$.
		\item See Theorem 6.4 in \cite{rockafellar}.
		\item If $z \in \stdCone^\perp$, then $\inProd {\stdInt} {z}$ is zero.
		Conversely, suppose that  $\inProd {\stdInt} {z}$ is zero. By item~$(\ref{laux:ri2})$, 
		there is $\alpha > 1$ such that\[u \coloneqq \alpha \stdInt + (1-\alpha)x\in \stdCone.\]
		On one hand, since $z \in \stdCone^*$, we have $\inProd{u}{z} \geq 0$. On the 
		other, $\inProd{u}{z} =  (1-\alpha)\inProd{x}{z} \leq 0$. So, we must have 
		$\inProd{x}{z} = 0$. As $x$ is an arbitrary element, it holds that $z \in \stdCone^\perp$.
%		\item Suppose $x \not \in \reInt \minFacePoint{x}{\stdCone}$. 
%		Then $x$ and 
%		$\minFacePoint{x}{\stdCone}$ can be properly separated by an hyperplane $H$ (see Theorem~11.3 in \cite{rockafellar}), where we recall that \emph{proper separation} means that $x$ and $\minFacePoint{x}{\stdCone}$ belong to opposite closed half-spaces defined by $H$ and at least one of them is not entirely contained in $H$. Since $x \in \minFacePoint{x}{\stdCone}$, we must have  $x \in H$ and, therefore, $\minFacePoint{x}{\stdCone}$ is not entirely contained in $H$. 
%		
%		The fact that $\minFacePoint{x}{\stdCone}$ is contained in one of the closed half-spaces defined by $H$ together with $x \in H$ imply that 
%		$\minFacePoint{x}{\stdCone} \cap H$ is a face of $\minFacePoint{x}{\stdCone}$ (and, therefore, a face of $\stdCone$).
%		We also have
%		\[
%		\minFacePoint{x}{\stdCone} \cap H \subsetneq \minFacePoint{x}{\stdCone},
%		\]
%		because $\minFacePoint{x}{\stdCone}$ is not entirely contained in $H$. This contradicts the minimality of $\minFacePoint{x}{\stdCone}$.
		
		\item[($v$) and ($vi$)] First, we observe that $\dirCone(x, \stdCone)$ coincides 
		$\{\alpha(w-x) \mid w \in \stdCone, \alpha \geq 0 \}$. The latter  is called the \emph{cone of $\stdCone$ at $x$} in the terminology of \cite{Tam85} and its dual is given by $\stdCone^* \cap\{x\}^\perp$. With this in mind, both items follow from Proposition~3.1 and Corollary~3.2 in \cite{Tam85}.
		%By definition, $\minFacePoint{x}{\stdCone} ^\Delta = \stdCone^* \cap \minFacePoint{x}{\stdCone}^\perp$.
%		Since $x \in \minFacePoint{x}{\stdCone}$, we have 
%		\[
%		\minFacePoint{x}{\stdCone} ^\Delta= \stdCone^* \cap \minFacePoint{x}{\stdCone}^\perp \subseteq 
%		\stdCone^* \cap \{x\}^\perp.
%		\]
%		For the converse, suppose that $s \in \stdCone^* \cap \{x\}^\perp$.
%		We have $s \in \minFacePoint{x}{\stdCone}^*$, since $\minFacePoint{x}{\stdCone} \subseteq \stdCone$.
%		By item $(\ref{laux:fri})$, $x \in \reInt\minFacePoint{x}{\stdCone}$. From item $(\ref{laux:ri3})$ we conclude that $s \in \minFacePoint{x}{\stdCone}^\perp$ and, therefore, $s \in \minFacePoint{x}{\stdCone} ^\Delta$.
%		\item[$(vii)$] First we show that $\minFacePoint{x}{\stdCone}^\Delta \subseteq (\tanCone{x}{\stdCone})^*$. 		
%		If $s \in \minFacePoint{x}{\stdCone}^\Delta$ and $z \in \dirCone (x,\stdCone)$, then, for some $t > 0$ we have \[\inProd{s}{x+tz} \geq 0.\] 
%		Because $\inProd{s}{x} = 0$, we must have $\inProd{s}{z} \geq 0$, which 
%		shows that $s \in \dirCone (x,C)^* = (\tanCone{x}{\stdCone})^* $, since a set and its closure have the same dual.
%		
%		Conversely, suppose that $s  \in (\tanCone{x}{\stdCone})^*$. Because $\stdCone \subseteq \tanCone{x}{\stdCone}$, we 
%		have that $s \in \stdCone^*$. In addition, since both $x$ and $-x$ belong to $\tanCone{x}{\stdCone}$, we 
%		have $\inProd{s}{x} = 0$. This implies that $s \in \stdCone^* \cap \{x\}^\perp$, which coincides with $\minFacePoint{x}{\stdCone}^\Delta$, by item $(\ref{laux:conj})$.
		\item[$(vii)$] Follows from  $(\ref{laux:perp})$ and $(\ref{laux:tan})$.
%		\item Since $\stdCone$ is a closed convex cone, we have 
%		$\dist(a+b,\stdCone) \leq \dist(a,\stdCone) + \dist(b,\stdCone)$, for all 
%		$a,b \in \ambSpace$. Now, 
%		for every $\epsilon > 0$, there exists $w_\epsilon \in \dirCone(x,\stdCone)$ such 
%		that $\dist(w,w_\epsilon) < \epsilon$. Moreover, there exists $t_\epsilon$ such 
%		that $t_\epsilon x + w_\epsilon \in \stdCone$. It follows that 
%		\begin{align*}
%		\dist(tx + w, \stdCone) & \leq \dist(tx + w_\epsilon, \stdCone) + \dist(w - w_\epsilon , \stdCone) \\
%		&\leq \dist(tx + w_\epsilon, \stdCone) + \epsilon,
%		\end{align*} 
%		where the last inequality follows from the fact that $0 \in \stdCone$, so 
%		$\dist(w - w_\epsilon , \stdCone) \leq \dist(w - w_\epsilon , 0)$. However, 
%		since $t_\epsilon x + w_\epsilon \in \stdCone$, we must have 
%		$\lim _{t\to + \infty}\dist(tx + w_\epsilon, \stdCone) = 0$, since for $t$ 
%		sufficiently large we have $tx + w_\epsilon \in \stdCone$. It follows 
%		that $\lim _{t\to + \infty} \dist(tx + w, \stdCone) \leq  \epsilon$. Since 
%		$\epsilon$ is arbitrary, we conclude that $(\ref{laux:dist})$ must hold.
	\end{enumerate}
\end{proof}

\subsection{Types of feasibility, almost optimality, almost feasibility}\label{sec:type}
Here, we review the fact that a conic linear program can be 
in four different mutually exclusive feasibility statuses. 
%Sometimes, we will be interested solely in the conic feasibility problem, which 
%we will denote by $(\stdCone , \stdSpace , c)$. This is 
%the problem of seeking a point  in the intersection $(\stdSpace + c)\cap %\stdCone$, where 
% $\stdSpace \subseteq \Re^n$ is a subspace and $c \in \Re^n$. There  
%four mutually exclusive categories that $(\stdCone , \stdSpace , c)$ can fall in:
We say that \eqref{eq:c_dual} is 
\begin{enumerate}[$(i)$]
	\item strongly feasible if $(\reInt \stdCone)  \cap (c+ \matRange\stdMap^\T ) \neq \emptyset$ (i.e., Slater's condition hold),
	\item weakly feasible if it is feasible but not strongly feasible,
	\item weakly infeasible if it is infeasible 
	but  $\text{dist}(c+\matRange\stdMap^\T , \stdCone) = 0$,
	\item strongly infeasible if $\text{dist}(c+\matRange\stdMap^\T , \stdCone) > 0$.
\end{enumerate}
Strong/weak feasibility/infeasibility of \eqref{eq:c_primal} is defined analogously by replacing  $(c+\matRange \stdMap^\T)$ by $\mathcal{V} \coloneqq \{x \mid \stdMap x = b\}$. As a matter of convention, if $\mathcal{V} = \emptyset$, we will say that \eqref{eq:c_primal} is strongly infeasible. If a problem is either weakly infeasible or weakly feasible we will say that it is in \emph{weak status}.
In view of these definitions, the usual assumption underlying interior point methods amounts to requiring both primal and dual strong feasibility.
%A primal-dual pair is called regular if both \eqref{eq:c-primal} and \eqref{eq:c-dual} are strongly feasible.

We have the following characterization of strong infeasibility, see Lemma 5 in \cite{Luo97dualityresults}. 

\begin{proposition}[Characterization of strong infeasibility]
\label{prop:infeas}	The following hold.
\begin{enumerate}[$(i)$]
	\item \eqref{eq:c_primal} is strongly infeasible if 
	and only if there exists $y$ such that
	\begin{equation}\label{eq:prop:si}
	\inProd{b}{y} = 1 \quad \text{and}\quad -\stdMap^\T y \in \stdCone.
	\end{equation}
	\item \eqref{eq:c_dual} is strongly infeasible if and only 
	if there exists $x$ such that
	\begin{equation}\label{eq:prop:si:d}
	\inProd{c}{x} = -1 \quad \text{and}\quad x \in   \stdCone^*\cap \ker \stdMap
	\end{equation}
\end{enumerate}	
\end{proposition}
Moving on, let $y \in \ambSpace$ and $\epsilon > 0$.
We say that $y$ is an \emph{$\epsilon$-feasible solution to \eqref{eq:c_dual}} if $\dist(c - \stdMap^\T y, \stdCone) \leq \epsilon$. 
In addition, we say that $y$ is an \emph{$\epsilon$-optimal solution to \eqref{eq:c_dual}} if 
$y$ is feasible for \eqref{eq:dual}  and $\inProd{b}{y} \geq \dOpt - \epsilon$. These notions will be used in Sections~\ref{sec:almost} and \ref{sec:inf}.

In general, 
even if $s = c-\stdMap^\T y$ is such that $\dist(s, \stdCone)$ is small, there is no guarantee that $\dist(s, \feasS )$ will also be small.
In this case, the quantities $\dist(s, \stdCone )$ and 
$\dist(s,\feasS)$ are sometimes called the \emph{backward error} and 
\emph{forward error}, respectively.
The problem of bounding the forward error by the backward error   is intrinsically connected with the notion of \emph{error bounds}. 
See, for example, the fundamental work by Sturm \cite{sturm_error_2000} on error bounds for linear matrix inequalities, where he showed the importance of facial reduction in analyzing these questions. See also \cite{L17,LLP20} for some generalizations of Sturm's results to the so-called
amenable cones and beyond.

%We also remark the following immediate consequence of Proposition \ref{prop:infeas}.
%\begin{corollary}\label{col:unbound}
%\begin{enumerate}[$i.$]
%	\item If the primal \eqref{eq:c_primal} is strongly infeasible and the dual  \eqref{eq:c_dual}
%	is feasible, then $\dOpt= +\infty$.
%	\item If the dual is strongly infeasible and the primal is feasible then $\pOpt = -\infty$. 
%\end{enumerate}
%\end{corollary}
%\begin{proof}
%To prove item $(i)$, note that by Proposition \ref{prop:infeas}, strong infeasibility at the primal side implies the 
%existence of $y$ such that $\inProd{b}{y} = 1$ and $-\stdMap ^\T  y \in \stdCone ^*$. So, if 
%$\hat y$ is any dual feasible solution then $\hat y+\alpha y$ is also feasible for every $\alpha \geq 0$ and we can make the objective 
%function value as large as we want. The proof of item $ii.$ is analogous.
%\end{proof}

\subsection{Facial structure of $\PSDcone{n}$ and a few remarks on $\intOracle$}\label{sec:oracle}
The cone of positive semidefinite symmetric matrices has a very special
structure and every face of $\PSDcone{n}$ is linearly isomorphic 
to some $\PSDcone{r}$ for $r \leq n$. This is a well-known fact 
which we state as a proposition for future reference.
For a proof, see \cite{pataki_handbook}. See also Section~6 of 
\cite{BC75}. Since it will be clear from the context, in what follows we use the convention that $0$ always denotes a zero matrix of appropriate size.
\begin{proposition}\label{prop:psdface}
Let $\stdFace$ be a nonempty face of $\PSDcone{n}$. There exists $r\leq n$ and an orthogonal $n\times n$ matrix $Q$ such that 
	\begin{equation}\label{eq:face}
	Q^\Tr \stdFace Q = \left\{ \begin{pmatrix}U & 0 \\ 0 & 0 \end{pmatrix}  \in \S^n \relmiddle{\vert} U \in \PSDcone{r} \right\}
	\end{equation}
\end{proposition}
Let $\stdFace$ be as in Proposition~\ref{prop:psdface}, then $\stdFace^*$ satisfies
\begin{equation}\label{eq:face_dual}
Q \stdFace ^* Q^\Tr = (Q^\Tr \stdFace Q)^* = \left\{ \begin{pmatrix}U & V \\ {V}^\Tr & W \end{pmatrix} \in \S^n \relmiddle{\vert} U \in \PSDcone{r} \right\}.
\end{equation}
Given an arbitrary nonempty face $\hat \stdFace$ of  $\stdFace^*$,
there is an  orthogonal matrix $\hat Q$ such that
\begin{equation}\label{eq:face_face_dual}
\hat Q \hat \stdFace\hat Q^\Tr =  \left\{ \begin{pmatrix}\left(\begin{array}{cc}U&0 \\ 0&0 \end{array}\right) & V \\ {V}^\Tr & W \end{pmatrix} \in \S^n \relmiddle{\vert} U \in \PSDcone{q} \right\},
\end{equation}
where $q \leq r$.  Then, $\hat \stdFace^*$ satisfies
\begin{equation}\label{eq:face_face_primal}
\hat Q^\Tr{\hat \stdFace}^* \hat Q  =  \left\{ \begin{pmatrix}\left(\begin{array}{cc}U&V \\ V^T& W \end{array}\right) & 0 \\ 0 & 0 \end{pmatrix} \in \S^n \relmiddle{\vert} U \in \PSDcone{q} \right\}.
\end{equation}

 %The machine does not need to receive 
%interior points, but it is  only guaranteed to work properly if  both \eqref{eq:primal} and 
%\eqref{eq:dual} have interior feasible points. 

%Our discussion could be framed precisely in the real computation model of Blum, Shub and Smale \cite{blum_complexity_1997}, 
%but for simplicity, apart from  $\intOracle$, we will only assume that  all the elementary real arithmetic operations 
%can be carried out exactly. Under this setting, we can use  $\intOracle$to compute 
%the square root of any positive real number $\beta$ by solving the SDP $\dOpt =  \sup \{ x \mid \bigl(\begin{smallmatrix} 1  & x \\ x & \beta \end{smallmatrix} \bigr) \in \PSDcone{2}  \}$.
%Since it has both primal and dual interior feasible solutions, it is licit to call  $\intOracle$ to solve it.

In the definition of $\intOracle$,
the affine space is contained in the space of $n\times n$ symmetric matrices 
and the optimization is carried over $\PSDcone{n}$. Note that $n$ is 
the same for both $\PSDcone{n}$ and $\S^n$. However, for fixed $\stdMap,b,c $ we might 
be interested in solving problems over a \emph{face} of $\PSDcone{n}$, the dual of a face of $\PSDcone{n}$ or even over a face of the dual of a face as in \eqref{eq:face_face_dual}.

%For this subsection, suppose that $\stdCone$ is a face of $\PSDcone{n}$.
In those cases, even if \eqref{eq:c_dual} and \eqref{eq:c_primal} are
both primal and dual strongly feasible, it is not immediately clear how to use  $\intOracle$ to solve \eqref{eq:c_dual} and \eqref{eq:c_primal}, since they are not exactly standard form SDPs. One possibility would be to consider a, \emph{a priori}, stronger 
oracle that is also able to solve strongly feasible problems over faces of $\PSDcone{n}$.

We will show that this is not necessary and, after some linear algebra, we can still solve \eqref{eq:c_primal} and \eqref{eq:c_dual} using $\intOracle$.
We register this fact as a proposition.
Let $\PSDcone{r,n}$ denote the face of $\PSDcone{n}$ corresponding to the right-hand side 
of \eqref{eq:face}.
Let $\stdCone$ be a cone as in \eqref{eq:face}, \eqref{eq:face_dual}, \eqref{eq:face_face_dual} or \eqref{eq:face_face_primal}. We see that in all those cases, we have 
\begin{equation}\label{eq:cone_ref}
R\stdCone R^\Tr = \PSDcone{r,n} \oplus \stdSpace,
\end{equation}
for some orthogonal matrix $R$, some $r \leq n$ and some linear subspace $\stdSpace \subseteq \S^n$ 
such that   $ \stdSpace \subseteq (\PSDcone{r,n})^\perp$. Here $\oplus$ denotes the Minkowski sum.
\begin{proposition}\label{prop:refor}
Let $\stdCone$ be as in \eqref{eq:cone_ref} with some orthogonal matrix $R$, some $r \leq n$, and
some linear subspace $\stdSpace \subseteq \S^n$ such that $\stdSpace\subseteq (\PSDcone{r,n})^\perp$.
Suppose that  \eqref{eq:c_primal} and \eqref{eq:c_dual} are strongly feasible. 	Then, \eqref{eq:c_primal} and \eqref{eq:c_dual} are solvable 
	with a single call to $\intOracle$.
\end{proposition}
The proof is elementary but quite cumbersome, so it is deferred to Appendix~\ref{app:refor}.

\section{Facial reduction with $\intOracle$}\label{sec:fr}
A major obstacle for solving \eqref{eq:dual}  with 
the oracle $\intOracle$ is that, in general, \eqref{eq:dual} is not strongly feasible, i.e., Slater's condition might not hold. 
%Therefore, a critical step towards solving \eqref{eq:dual} is to reformulate \eqref{eq:dual} as an SDP instance that is primal and dual strongly feasible.
By using facial reduction, we are able to either detect infeasibility or to reformulate \eqref{eq:dual} as an SDP instance that is strongly feasible at \emph{one side} of the problem.
This will be an important step towards completely solving \eqref{eq:dual}.

In this section, we discuss facial reduction for general conic linear programs and how it can be carried out by solving auxiliary problems that are ensured to be strongly feasible at both primal and dual sides. 
In particular, when the underlying cone $\stdCone$ is 
$\PSDcone{n}$, this will mean that facial reduction can be implemented through calls to $\intOracle$.
Although we will focus on problems formulated in the dual form \eqref{eq:c_dual}, any analysis carried out for \eqref{eq:c_dual} can be translated back to \eqref{eq:c_primal}. Here, we will follow the approach described in \cite{article_waki_muramatsu}, which relies on the following key result.

\begin{lemma}[The facial reduction lemma: Lemma 3.2 in \cite{article_waki_muramatsu}]\label{lem:fr}
The following hold.
\begin{enumerate}[$(i)$]
\item 
	\eqref{eq:c_dual} is not strongly feasible (i.e., Slater's condition fails) if and only if there is $d \in \stdCone^* \cap \ker \stdMap$ such that:
	\begin{enumerate}[$(i)$]
		\item  $\inProd{c}{d} = 0$ and $d \not \in \stdCone^\perp$, or
		\item $\inProd{c}{d} < 0$.
	\end{enumerate}
\item 
\eqref{eq:c_primal} is not strongly feasible (i.e., Slater's condition fails) if and only if there are
$y \in \mathbb{R}^m, f \in \stdCone$ such that $f =- \stdMap^\T y$ and
	\begin{enumerate}[$(i)$]
		\item  $\inProd{b}{y} = 0$ and $f \not \in \stdCone^{*\perp}=\lineality \stdCone$  (item $\eqref{laux:perp}$ of Lemma~\ref{lemma:aux}), or
		\item $\inProd{b}{y} > 0$.
	\end{enumerate}
\end{enumerate}
\end{lemma}
%\begin{remark}
%The primal counter part of Lemma~\ref{lem:fr} is obtained when $\ker \stdMap, \stdCone^*$ are substituted by $\matRange \stdMap^\T, \stdCone$ respectively and the condition on $c$ is replaced by a condition on $b$. To wit,
%\eqref{eq:c_primal} is not strongly feasible if and only if 
%there exists $(f,y)$ such that $f = -\stdMap^\T y \in \stdCone$ and 
%either $(i)$ $\inProd{b}{y} = 0$ and $f \not \in \stdCone^{*\perp} = \lineality \stdCone$ (item $\eqref{laux:perp}$ of Lemma~\ref{lemma:aux})
%holds or $(ii)$ $\inProd{b}{y} > 0$ holds.
%\end{remark}

%TODO: write the comments bellow as a proposition or add to Lemma 7
Therefore,  whenever \eqref{eq:c_dual} lacks a relative interior solution (i.e., $ \stdCone \cap (c+\matRange \stdMap^\T) = \emptyset$), it is 
either because \eqref{eq:c_dual} is infeasible (alternative $(ii)$ together with Proposition~\ref{prop:infeas}) or because 
the set of dual feasible slacks $\feasS$ is contained in $\stdCone \cap \{d\}^\perp$ (alternative $(i)$).\footnote{One must be careful that even if \eqref{eq:c_dual} is infeasible it might be the case that alternative $(ii)$ is not satisfied at this stage. This happens, for instance, if \eqref{eq:c_dual} is weakly infeasible.}
If alternative $(i)$ holds, since  $d \not \in \stdCone ^{\perp} $, 
we have
\begin{equation}
\stdCone \cap \{d\}^\perp \subsetneq \stdCone, \label{eq:fac_des}
\end{equation}
that is, the face $\stdFace_2 \coloneqq  \stdCone \cap \{d\}^\perp$  is properly contained in $\stdCone $. We then substitute 
$\stdCone$ for   $\stdFace _2$ and repeat.
As long as $(\reInt \stdFace _i)\cap (c+\matRange \stdMap^\T) = \emptyset $, we can find a new direction 
$d$.

We recall that if $\stdFace$ is a face of $\stdCone$, 
then $\stdFace \subsetneq \stdCone$ holds if and only if $\dim \stdFace < \dim \stdCone$. Therefore, \eqref{eq:fac_des} implies that after a 
finite number of iterations,  we will either find some face $\stdFace_\ell$ such that $ (\reInt \mathcal{F}_\ell )\cap (c + \matRange \stdMap) \neq \emptyset$ or 
we will eventually find out that the problem is infeasible. 

It turns out that $\mathcal{F}_\ell$ must be the smallest face $\minFaceD$ of $\stdCone$ which contains $\feasS$. 
This process is called 
\emph{facial reduction} and it aims at finding $\minFaceD$. 
%The direction $d$ will be 
%henceforth called a \emph{reducing direction}. 
If $\feasS = \emptyset$, we 
have $\minFaceD = \emptyset$ by convention. 
For the sake of preciseness, we will state the following definition.

\begin{definition}[Reducing directions]\label{def:red}
A \emph{reducing direction for \eqref{eq:c_dual}} is an element 
$d \in \stdCone^* \cap \ker \stdMap$ such that 
$\inProd{c}{d} \leq 0$. A \emph{reducing direction for \eqref{eq:c_primal}} is a pair $(f,y)$ such that $f \in \stdCone $, $f = -\stdMap^\T {y}$ (i.e., $f \in \matRange \stdMap^\T$) and $\inProd{b}{y} \geq 0$.

Next, $\{d_1,\ldots, d_{\ell}\}$ is said to be a \emph{sequence of reducing directions for \eqref{eq:c_dual}} if
\begin{align}
d_i & \in (\stdCone \cap \{d_1\}^\perp \cap \cdots \cap \{d_{i-1}\}^\perp)^* \cap \ker \stdMap \cap \{c\}^\perp, \text{ for } i = 1,\ldots, {\ell-1} \label{eq:frd_d:1}\\
d_{\ell} & \in (\stdCone \cap \{d_1\}^\perp \cap \cdots \cap \{d_{\ell-1}\}^\perp)^* \cap \ker \stdMap, \qquad \inProd{c}{d_{\ell}} \leq 0.\label{eq:frd_d:2}
\end{align}

Analogously, $\{(f_1,y_1),\ldots, (f_{\ell},y_{\ell})\}$ is said to be a \emph{sequence of reducing directions for \eqref{eq:c_primal}} if 
\begin{align}
f_i = -\stdMap^\T y_i,\quad y_i \in \{b\}^\perp,\quad & f_i  \in (\stdCone^* \cap \{f_1\}^\perp \cap \cdots \cap \{f_{i-1}\}^\perp)^*, \text{ for } i = 1,\ldots, {\ell-1} \label{eq:frd_p:1}\\
f_\ell = -\stdMap^\T y_\ell,\quad \inProd{b}{y_{\ell}} \geq 0,\quad & f_\ell  \in (\stdCone^* \cap \{f_1\}^\perp \cap \cdots \cap \{f_{\ell-1}\}^\perp)^*.\label{eq:frd_p:2}
\end{align}
\end{definition}
\begin{remark}
	Liu and Pataki introduced in \cite{LP15} the notion of \emph{facial reduction cone}, see Definition~2 therein. The $k$-th facial reduction cone 	of $\stdCone$ is given by 
	\[
	\mathrm{FR}_k(\stdCone) = \{(d_1,\ldots, d_{k})\mid d_1 \in \stdCone^*,d_i  \in (\stdCone \cap \{d_1\}^\perp \cap \cdots \cap \{d_{i-1}\}^\perp)^*, i=2,\ldots, k  \}.
	\]
With that, \eqref{eq:frd_d:1}, \eqref{eq:frd_d:2} and \eqref{eq:frd_p:1}, \eqref{eq:frd_p:2} imply that 
\[ 
(d_1,\ldots,d_{\ell}) \in\mathrm{FR}_\ell(\stdCone), \quad (f_1,\ldots,f_{\ell}) \in  \mathrm{FR}_\ell(\stdCone^*).
\]  	
\end{remark}

The minimal face ${\cal F}_{\min}^D$ containing the feasible region of \eqref{eq:c_dual}
also has the following well-known characterization, see 
for instance, Proposition 3.2.2 in \cite{pataki_handbook}.

\begin{proposition}[Characterizations of the minimal face]\label{prop:min_face}
	Let  $\stdFace$ be a face of $\stdCone $ containing $\feasS$.
	Suppose $\stdFace$ and $\feasS$ are both non-empty. Then the conditions 
	below are equivalent.
	\begin{enumerate}[$(i)$]
		\item $\feasS\cap \reInt \stdFace \neq \emptyset$.
		\item $\reInt \feasS \subseteq \reInt \stdFace$.
		\item $\stdFace = \minFaceD$.
	\end{enumerate}
\end{proposition}

%Facial reduction is a very powerful procedure and it 
%can be used to solve feasibility problems over arbitrary closed convex cones. 
The computationally challenging part of facial reduction is computing $d$ which requires, in general, solving another CLP. At first glance, it seems that we are again stuck solving an CLP that might also not be strongly feasible. However, \emph{even if the original CLP is not strongly feasible, 
searching for $d$ can always be done by solving problems that
are primal and dual strongly feasible}.
In particular, when $\stdCone = \PSDcone{n}$, finding reducing directions can be done with
$\intOracle$. 

\begin{lemma}[Finding a reducing direction through strongly feasible auxiliary problems]\label{lemma:red_sdp}
Let $\stdInt \in \reInt \stdCone$, $\stdInt^* \in \reInt \stdCone^*$
and consider the following pair of primal and dual problems.
	\begin{align}
	\underset{x,t,w}{\inf} & \quad t & \tag{$P_\stdCone$}\label{eq:red}  \\ 
	\mbox{subject to} & \quad -\inProd{c}{x  -t\stdInt^*} + t - w &= 0 \label{eq:red:1}\\
	&\quad \inProd{\stdInt}{x} + w &= 1 \label{eq:red:2}\\
	& \quad\stdMap x  -t\stdMap \stdInt^* & = 0 \label{eq:red:3}\\
	&\quad (x,t,w) \in \stdCone ^* \times \Re_+ \times \Re_+ \nonumber
	\end{align}	
	\begin{align}
	\underset{y_1,y_2,y_3}{\sup} & \quad  y_2 & \tag{$D_\stdCone $}\label{eq:red_dual}  \\ 
	\mbox{subject to} & \quad cy_1 -\stdInt y_2 -\stdMap^\T y_3 \in \stdCone \label{eq:red_dual:1} \\
	&\quad 1-y_1(1+\inProd{c}{\stdInt^*}) + \inProd{\stdInt^* }{\stdMap ^\T y_3}\geq 0 \label{eq:red_dual:2} \\
	& \quad y_1-y_2 \geq 0 \label{eq:red_dual:3}
	\end{align}
	
	The following properties hold.
	\begin{enumerate}[$(i)$]
		\item Both \eqref{eq:red} and \eqref{eq:red_dual} are strongly feasible.
	\end{enumerate}
	Let $(x^*,t^*,w^*)$ be an optimal solution to \eqref{eq:red} and 
	$(y_1^*,y_2^*,y_3^*)$ be  an optimal solution to \eqref{eq:red_dual}.

	\begin{enumerate}[$(i)$]
		\setcounter{enumi}{1} 
		\item The primal optimal value $\opt{\text{\ref{eq:red}}}$ is zero if and only if
		$\minFaceD \subsetneq \stdCone$. In this case, one of the two alternatives below must hold:
		\begin{enumerate}
			\item $\inProd{c}{x^*} < 0 $ and $(c+\matRange \stdMap^\T)\cap \stdCone = \emptyset$ (i.e., \eqref{eq:c_dual} is infeasible), or
			\item $\inProd{c}{x^*} = 0 $ and $(c+\matRange \stdMap^\T)\cap \stdCone \subseteq  \stdCone \cap \{x^*\}^\perp \subsetneq \stdCone$. 
		\end{enumerate}
		
		\item The primal optimal value $\opt{\text{\ref{eq:red}}}$ is positive 
		if and only if \eqref{eq:c_dual} is strongly feasible, i.e., $\minFaceD = \stdCone$. In this case, we have
		\[
		c  -\stdMap^\T \frac{y_3^*}{y_1^*} \in \reInt \stdCone. 
		\]
	\end{enumerate}
	
\end{lemma}
\begin{proof}
	\begin{enumerate}[$(i)$]
		\item Let 
		\[
		t \coloneqq \frac{1}{\inProd{\stdInt}{\stdInt^*}+1},\quad w \coloneqq \frac{1}{\inProd{\stdInt}{\stdInt^*}+1},\quad x \coloneqq \frac{\stdInt^*}{\inProd{\stdInt}{\stdInt^*}+1}.
		\]
		Then $(x,t,w) $ is a strongly feasible solution to \eqref{eq:red}, i.e., 
		\[
		(x,t,w) \in \reInt (\stdCone \times \Re_+ \times \Re_+) = \reInt\stdCone \times \reInt \Re_+ \times \reInt \Re_+.
		\]
		
		Next, we observe that $(y_1,y_2,y_3) \coloneqq (0,-1,0)$ is a feasible solution to \eqref{eq:red_dual} such that \eqref{eq:red_dual:2}, \eqref{eq:red_dual:3} are satisfied strictly and 
		\[
		 cy_1 -\stdInt y_2 -\stdMap^\T y_3 = e \in \reInt \stdCone.
		\]
		We have thus shown that both \eqref{eq:c_primal} and \eqref{eq:c_dual} are strongly feasible.
		\item First, let $(x^*,t^*,w^*)$ be an optimal solution to \eqref{eq:c_primal} and 
		suppose that $\opt{\text{\ref{eq:red}}}$  is zero.
	We have 
		$t^* = 0$. Then, 	\eqref{eq:red:1} and \eqref{eq:red:3} together with $x^* \in \stdCone^*$ and $w^* \geq 0$ imply that 
		\begin{equation}
		x^* \in \ker \stdMap \cap \stdCone^*, \quad \inProd{c}{x^*} \leq 0. \label{eq:aux1}
		\end{equation}
		Then, we have two possibilities.
		\begin{enumerate}
			\item Suppose $\inProd{c}{x^*} < 0 $. We will show that \eqref{eq:c_dual} must be infeasible. Let $s \in (c+\matRange \stdMap^\T)$, then \eqref{eq:aux1} implies $\inProd{s}{x^*} < 0$. Since $x^* \in \stdCone^*$, we conclude that 
			$s$ cannot belong to $\stdCone$, because otherwise we would have $\inProd{x^*}{s} \geq 0$.
			
			Therefore,  \eqref{eq:c_dual} must be infeasible and $(c+\matRange \stdMap^\T)\cap \stdCone = \emptyset$. In this case, we have $\minFaceD = \emptyset$ and, indeed, $\minFaceD \subsetneq \stdCone$.
			
			\item Suppose $\inProd{c}{x^*} = 0 $.  This, together with \eqref{eq:aux1} implies that
			\[
			(c+\matRange \stdMap^\T)\cap \stdCone = \feasS \subseteq  \stdCone \cap \{x^*\}^\perp. 
			\]
			Next, we will check that the inclusion $\stdCone \cap \{x^*\}^\perp \subsetneq \stdCone$ is indeed proper.  We observe that since $t^* = 0$ and 
			$\inProd{c}{x^*} = 0 $, \eqref{eq:red:1} implies that $w^* = 0$ too. 
			Therefore, \eqref{eq:red:2} implies that $\inProd{\stdInt}{x^*} = 1$.
			In particular, $x^*$ does not belong to $\stdCone^\perp$. In other words, 
			\[
			\stdCone \cap \{x^*\}^\perp \subsetneq \stdCone.
			\]
			Since $\minFaceD \subseteq \stdCone \cap \{x^*\}^\perp $, we also have 
			$\minFaceD \subsetneq \stdCone$.
		\end{enumerate}
		
		Now, we will prove the converse. That is, we will suppose that 
		$\minFaceD \subsetneq \stdCone$ and we will show that  $\opt{\text{\ref{eq:red}}} = 0$. We start by observing that since the objective function of \eqref{eq:red} is ``$t$'' and 
		$t$ is constrained to be nonnegative, if we exhibit a feasible solution 
		for \eqref{eq:red} having $t = 0$ this would be enough to show that 
		$\opt{\text{\ref{eq:red}}} = 0$.

		Since $\minFaceD \subsetneq \stdCone$, \eqref{eq:c_dual} is not strongly feasible.
		By Lemma~\ref{lem:fr}, there exists some $x \in \stdCone^* \cap \ker \stdMap$ such that
		either $(a)$ $\inProd{c}{x} = 0$ and $x \not \in \stdCone^\perp$ or $(b)$ $\inProd{c}{x} < 0$. 		
		 Let us check each case.
		\begin{enumerate}
			\item Suppose 			$\inProd{c}{x} = 0$ and $x \not \in \stdCone^\perp$.
			Then the condition $x \not\in \stdCone^\perp$ implies that $\inProd{\stdInt}{x} > 0$, by item $(\ref{laux:ri2})$ of Lemma~\ref{lemma:aux}.
			Then, 
			\[
			\left(\frac{x}{\inProd{\stdInt}{x}},0,0\right)
			\]
			is a feasible  solution for \eqref{eq:red}, which shows 
			that $\opt{\text{\ref{eq:red}}} = 0$.

		\item Suppose that $\inProd{c}{x} < 0$. We define
		\[\alpha \coloneqq \frac{1}{\inProd{\stdInt}{x} - \inProd{c}{x}}
		\]  and this is well-defined 
		because $- \inProd{c}{x} > 0$ and $\inProd{\stdInt}{x} \geq 0$. Then $(x\alpha,0,-\alpha \inProd{c}{x})$ is a feasible solution to \eqref{eq:red}, which also shows 
		that  $\opt{\text{\ref{eq:red}}} = 0$. 
		
	\end{enumerate}
		\item Since \eqref{eq:red} and \eqref{eq:red_dual} are both strongly feasible, we have, in particular, that $\opt{\text{\ref{eq:red}}} = \opt{\text{\ref{eq:red_dual}}}$ and there is an optimal solution to \eqref{eq:red_dual} $(y_1^*,y_2^*,y_3^*)$ satisfying 
		$y_2^* = \opt{\text{\ref{eq:red}}}$. 
		By item $(ii)$, we have that $\opt{\text{\ref{eq:red}}} = 0$ if and only if $\minFaceD \subsetneq \stdCone$, which happens if and only if \eqref{eq:c_dual} is \emph{not} strongly feasible, by Proposition~\ref{prop:min_face}.
		As $\opt{\text{\ref{eq:red}}} $ is always nonnegative (because $t$ is constrained to be nonnegative), we conclude that \eqref{eq:c_dual} is strongly feasible if and only if $\opt{\text{\ref{eq:red}}} $ is positive.

		Next, suppose that $\opt{\text{\ref{eq:red}}}$ is indeed positive. 
		In this case we have that  $y_2^* = \opt{\text{\ref{eq:red}}}$ is positive and that 
	    \[
	    \stdInt y_2^* \in \reInt \stdCone,
	    \] 
	    since $\stdInt \in \reInt \stdCone$.
		This fact, together with \eqref{eq:red_dual:1} and item $(\ref{laux:ri})$ of Lemma~\ref{lemma:aux}, implies that 
		\[ cy_1^*  -\stdMap^\T  y_3^* \in \reInt \stdCone. \]		
		To conclude, we observe that \eqref{eq:red_dual:3} implies that $y_1^* \geq y_2^* > 0$. Therefore, 
		\[c  -\stdMap^\T  \frac{y_3^*}{y_1^*} \in \reInt \stdCone\]
	   Using Proposition \ref{prop:min_face}, we conclude that indeed $\minFaceD = \stdCone$.
		
	\end{enumerate}
\end{proof}
\begin{remark}
	Lemma~\ref{lemma:red_sdp} holds for any pair of $\stdInt, \stdInt^*$ satisfying $\stdInt \in \reInt \stdCone, \stdInt^* \in \reInt \stdCone^*$.
	When $\stdCone = \PSDcone{n}$, we may take $\stdInt$ and $\stdInt^*$ to be, for example, both equal to the $n\times n$ identity matrix. If $\stdCone$ is some face of $\PSDcone{n}$,
	we can use Proposition \ref{prop:psdface} together  with \eqref{eq:face} and \eqref{eq:face_dual} to find  $\stdInt$ and $\stdInt^*$ as follows. We take $\stdInt = \stdInt^*$ and let 
	$\stdInt$ be such that $Q^\Tr \stdInt Q = \bigl(\begin{smallmatrix} I_r & 0 \\ 0 & 0 \end{smallmatrix} \bigr)$, where 
	$I_r$ is the $r\times r$ identity matrix. 	

	For SDPs, we note that Cheung, Schurr and Wolkowicz  also 
discuss an auxiliary problem that is primal and dual strongly feasible, see the problem $(AP)$ in \cite{csw13}. A key difference is that 
$(AP)$ requires an additional second order cone constraint, whereas \eqref{eq:red} and \eqref{eq:red_dual} only use linear equalities/inequalities and cone constraints involving the original cone $\stdCone$ and its dual.
	
\end{remark}

With the aid of Lemma~\ref{lemma:red_sdp}, we now are able to state a 
facial reduction algorithm that can be easily adapted to use the oracle $\intOracle$, when $\stdCone = \PSDcone{n}$, see Algorithm~\ref{alg:fra}.

\begin{algorithm}[h]
	\caption{Facial reduction with strongly feasible auxiliary problems}\label{alg:fra}
	\DontPrintSemicolon
	\SetKwInOut{Input}{Input}\SetKwInOut{Output}{Output}
	\Input{$\stdCone, \stdMap, c$ 
	}
	\Output{Reducing directions $d_1,\ldots, d_{\ell}$ for \eqref{eq:c_dual} (Definition~\ref{def:red}) together with \texttt{Feasible} or \texttt{Infeasible}. If \texttt{Feasible}, a pair $(s,y)$ is also returned so that
	\begin{align*}
	s = c-\stdMap^\T y & \in \reInt (\stdCone \cap \{d_1\}^\perp \cap \cdots \cap \{d_\ell\}^\perp ).
	\end{align*}
	}
$\stdFace_1 \leftarrow \stdCone$, $i \leftarrow 1$.

Replace $\stdCone,\stdCone^*$ by $\stdFace_i, \stdFace_i^*$ in \eqref{eq:red_dual} and \eqref{eq:red}, respectively, and 
solve the resulting pair of problems (associated with the cones $\stdFace_i, \stdFace_i^*$). 
Denote the obtained   optimal solutions by $(x^*,t^*,w^*)$ and $(y_1^*,y_2^*,y_3^*)$.\label{ln:or}%TODO: 

\eIf{$t^* = 0$}{ 
	$d_{i} \leftarrow x^*$\tcc*{Found a reducing direction}
	
	\eIf{$\inProd{c}{x^*} < 0$}{ \label{alg:fra:red1}
		$\minFaceD \leftarrow \emptyset$ \label{alg:fra:inf} \tcc*{$\inProd{c}{x^*} < 0$ attests that \eqref{eq:c_dual} is infeasible}
		
		\Return{} \texttt{Infeasible}, $\minFaceD,d_1,\ldots, d_i$
	}
	{
		$\stdFace _{i+1} \leftarrow \stdFace_i \cap \{d_i\}^\perp$ \label{alg:fra:red2} \tcc*{In this case we have $\inProd{c}{x^*} = 0$ }
		
		 $i \leftarrow i+1$	
		 
		\textbf{go to } line \ref{ln:or}}
}
{
	 $\minFaceD \leftarrow \stdFace_i$, \label{alg:fra:min}\tcc*{Found the minimal face }
	 
	 $s \leftarrow  c  -\stdMap^\T  \frac{y_3^*}{y_1^*}$ \tcc*{$s \in \reInt \stdFace_i$ } \label{alg:fra:min2}
	 
	 \Return{} \texttt{Feasible}, $\minFaceD,d_1,\ldots, d_i,\left(s,\frac{y_3^*}{y_1^*}\right)$ 
}
\end{algorithm} 
\noindent
\begin{proposition}[Algorithm~\ref{alg:fra} is correct]\label{prop:fra_cor}
Algorithm~\ref{alg:fra} correctly detects whether \eqref{eq:c_dual} is feasible or not.
If \eqref{eq:c_dual} is feasible, Algorithm~\ref{alg:fra} correctly identifies the minimal face $\minFaceD$ and the pair $(s,y)$ returned by Algorithm~\ref{alg:fra} does indeed satisfy
\[
s = c- \stdMap^\T y \in \reInt \minFaceD.
\]
\end{proposition}
\begin{proof}
The correctness of Algorithm~\ref{alg:fra} is a consequence of Lemma~\ref{lemma:red_sdp} and we will now explain some of the details. We have several claims.

\noindent\fbox{\textbf{Claim 1} For all $i$, $\stdFace_i$ contains $(c+\matRange \stdMap^\T)\cap \stdCone$ and $\stdFace_{i+1}$ is strictly contained in $\stdFace_i$}

\noindent This claim holds by induction.
When Algorithm~\ref{alg:fra} starts, we have $\stdFace _{1} = \stdCone$.
Now, suppose that for some $i$ we have that $\stdFace _i$ contains $(c+\matRange \stdMap^\T)\cap \stdCone$. Given 
$\stdFace _i$, we have that $\stdFace _{i+1}$ is constructed by the relation
\[
\stdFace _{i+1} = \stdFace _{i} \cap \{d_i\}^\perp.
\]
However, $\stdFace _{i+1}$ is only computed if the optimal value of \eqref{eq:red} is $0$ and $\inProd{c}{x^*} = 0$, see Lines~\ref{alg:fra:red1} and \ref{alg:fra:red2}.
In this case, item $(ii)(b)$ of Lemma~\ref{lemma:red_sdp} ensures
\begin{equation}
(c+\matRange \stdMap^\T)\cap \stdFace _i \subseteq \stdFace _{i+1} \subsetneq \stdFace _i.\label{eq:fra_cor}
\end{equation}
Since $\stdFace _i$ is a face (and, therefore, a subset) of $\stdCone$, 
the hypothesis that  $\stdFace _i$ contains $(c+\matRange \stdMap^\T)\cap \stdCone$
implies that, in fact,
\[(c+\matRange \stdMap^\T)\cap \stdFace _i = (c+\matRange \stdMap^\T)\cap \stdCone,
\]
which, combined with \eqref{eq:fra_cor}, implies that $\stdFace _{i+1}$ must 
also contain $(c+\matRange \stdMap^\T)\cap \stdCone$. 
This concludes the proof of \textbf{Claim 1}.

\noindent\fbox{\textbf{Claim 2} The minimal face of $\stdFace _i$ containing $(c+\matRange \stdMap^\T)\cap \stdCone$ coincides with $\minFaceD$ }

\textbf{Claim 2} follows from \textbf{Claim 1} and the fact that if $\stdFace$ is a face of $\stdCone$ and 
$\hat \stdFace$ is a face of $\stdFace$, then $\hat \stdFace$ is a face of $\stdCone$.

\noindent\fbox{\textbf{Claim 3} For all $i$, $(c+\matRange \stdMap^\T)\cap \stdFace _i = \emptyset$ holds if and only if \eqref{eq:c_dual} is infeasible  }
\textbf{Claim 3} is a consequence of \textbf{Claim 1}.

Now, \textbf{Claim 1} implies that whenever a new face $\stdFace_{i+1}$ is computed, it must be strictly smaller than $\stdFace _{i}$ and, therefore, the dimension must also be strictly smaller\footnote{The fact that $\stdFace _i$ and $\stdFace_{{i+1}}$ are faces is important, because, in general, $C_1 \subsetneq C_2$ does not imply $\dim C_1 < \dim C_2$.}. Since we cannot have an infinite strictly descending of faces, at some point, the optimal value of \eqref{eq:red} must become positive or a certificate that $(c+\matRange \stdMap^\T)\cap \stdFace_i = \emptyset$ will be found (see Lines~\ref{alg:fra:red1} and \ref{alg:fra:inf}). In the first case, \textbf{Claim 2} together with item $(iii)$ of Lemma~\ref{lemma:red_sdp} (applied to $\stdFace_i$) implies that $\minFaceD = \stdFace _i$ and that 
\[s = c- \stdMap^\T y \in \reInt \minFaceD,\]
where $y = y_3^*/y_1^*$.
 In the second case, \textbf{Claim 3} and item $(ii)(a)$ of Lemma~\ref{lemma:red_sdp} ensures that, indeed, \eqref{eq:c_dual} must be infeasible. 
\end{proof}

Next, we examine the computational cost of Algorithm~\ref{alg:fra}, following an analysis 
similar to other facial reduction approaches (e.g, \cite{pataki_strong_2013,article_waki_muramatsu}).
When Algorithm~\ref{alg:fra} is invoked, a chain of faces of $\stdCone$ is constructed as follows
\[
\stdCone =\stdFace _1 \supsetneq \cdots \supsetneq \stdFace _{\ell}.
\]
We recall that if $\stdFace, \hat \stdFace$ are faces of $\stdCone$ such that 
$\stdFace \subseteq \hat \stdFace$, then $\stdFace \neq \hat \stdFace$ if and only if 
$\dim \stdFace < \dim \hat \stdFace$. As $\stdCone$ is finite dimensional, we conclude 
that at most $\dim \stdCone + 1$ faces will be found when Algorithm~\ref{alg:fra} is invoked. This estimate can be sharpened in several different ways. 
For example, let $\ell _{\stdCone}$ denote the \emph{longest chain of strictly decreasing non-empty faces of $\stdCone$}. Then, the number of non-empty faces that will be found when Algorithm~\ref{alg:fra} is invoked is bounded above by $\ell _{\stdCone}$.
In particular, when $\stdCone= \PSDcone{n}$, we have
\[
\dim \stdCone = \frac{n(n+1)}{2}, \quad \ell _{\PSDcone{n}} = n+1.
\]
This shows that, in some cases, $\ell _{\stdCone}$ can be a much better bound 
than $\dim \stdCone$. For a proof that  $\ell _{\PSDcone{n}} = n+1$ see, for 
example, Theorem~14 in \cite{IL17} where it is shown that that  whenever $\stdCone$ is a symmetric cone (homogeneous self-dual cone), we have $\ell _{\stdCone} = \matRank \stdCone + 1$, where $\matRank\stdCone$ is the Jordan algebraic rank of $\stdCone$.
We summarize this discussion in the next proposition.
\begin{proposition}[Computational cost of Algorithm~\ref{alg:fra}]\label{prop:fra_cost}
The number of times that Algorithm~\ref{alg:fra}	solves the pair \eqref{eq:red} and 
\eqref{eq:red_dual} is bounded above by $\ell _\stdCone$.
In particular, when $\stdCone = \PSDcone{n}$, Algorithm~\ref{alg:fra} can be implemented by invoking $\intOracle$ at most $n+1$ times.

%\[
%\begin{cases}
%\min (\ell _\stdCone-1, \dim (\ker \stdMap \cap \{c\}^{\perp})+1) & 
%\text{ if \eqref{eq:c_dual} is feasible}\\
%\min (\ell _\stdCone, \dim (\ker \stdMap \cap \{c\}^{\perp})+1) & 
%\text{ if \eqref{eq:c_dual} is infeasible}
%\end{cases}
%\]	
\end{proposition}
\begin{proof}
In the proof of Proposition~\ref{prop:fra_cor}, we have shown that Algorithm~\ref{alg:fra} constructs a strictly nondecreasing chain of faces as follows
\begin{equation}\label{eq:chain}
\stdCone =\stdFace _1 \supsetneq \cdots \supsetneq \stdFace _{\ell}.
\end{equation}
We divide the proof in two cases.
Suppose first that \eqref{eq:c_dual} is feasible. Then, 
$\stdFace _{\ell} = \minFaceD$ by Proposition~\ref{prop:fra_cor} and $\minFaceD$ is not empty. Finding a new face $\stdFace _i$ in Algorithm~\ref{alg:fra} corresponds to solving the pair \eqref{eq:red} and \eqref{eq:red_dual} once. So, after solving the pair \eqref{eq:red} and \eqref{eq:red_dual} at most $\ell _{\stdCone}-1$ times, Algorithm~\ref{alg:fra} will set $\stdFace _{\ell}$ to $\minFaceD$. 
Then, \eqref{eq:red} and \eqref{eq:red_dual} will be solved one 
extra time in order to check that $\stdFace _{\ell}$ is indeed the minimal face and to obtain $s \in \reInt \stdCone$, as in Lines~\ref{alg:fra:min} and \ref{alg:fra:min2}. In total, \eqref{eq:red} and \eqref{eq:red_dual} is solved at most $\ell _{\stdCone}$ times.

% Next, we observe that 
%each face $\stdFace _{i+1}$ for $i > 0$ is constructed by considering an 
%intersection of the form 
%\[
%\stdFace _{i+1} = \stdFace _{i}\cap \{d_{i}\}^\perp = \stdCone \cap \{d_1\}^\perp \cap \dots \{d_{i-1}\}^\perp\cap \{d_i\}^\perp.
%\]
%If $d_i$ belongs to the linear span of $\{d_1, \ldots, d_{i-1}\}$, we would have that $\stdFace _{i+1} = \stdFace _i$, which is not allowed because all the containments in \eqref{eq:chain} are strict.
%Since all the $d_i$ belong to $\ker \stdMap \cap \{c\}^{\perp}$ and the number of directions is $\ell-1$, we also have
%\[
%\ell-1 \leq \dim (\ker \stdMap \cap \{c\}^{\perp}).
%\]

Next, suppose that \eqref{eq:c_dual} is infeasible.
In this case, the last face 
$\stdFace_{\ell}$ will be empty (see Line~\ref{alg:fra:inf}), but all faces up to $\ell-1$ will be 
nonempty. Therefore, $\ell-1 \leq \ell_\stdCone$. As in the previous case, each face in the chain \eqref{eq:chain} corresponds to solving the pair  \eqref{eq:red} and \eqref{eq:red_dual} once. 
In summary, after solving $\eqref{eq:red}$ and $\eqref{eq:red_dual}$ at most $\ell_\stdCone-1$ times, Algorithm~\ref{alg:fra} will find the last nonempty face $\stdFace _{\ell-1}$ and, then, $\eqref{eq:red}$ and $\eqref{eq:red_dual}$ will be solved once more in order to set 
$\stdFace _{\ell}$ to ``$\emptyset$''.

To conclude, we suppose that $\stdCone = \PSDcone{n}$. 
Then, Algorithm~\ref{alg:fra} successively solves the problem \eqref{eq:red} and \eqref{eq:red_dual} over $\PSDcone{n}$ and its faces at most $\ell_{\PSDcone{n}}=n+1$ times. 
By Lemma~\ref{lemma:red_sdp} and Proposition~\ref{prop:refor}, these are strongly feasible problems 
that can be solved by invoking $\intOracle$ a single time.
\end{proof}

We mention in passing that the minimal number of facial reduction steps needed to find the minimal face of \eqref{eq:c_dual} is often called the \emph{singularity degree} of \eqref{eq:c_dual}. The singularity degree is also bounded by  $\ell _{\stdCone}$, but sharper estimates can be obtained by considering facial reduction strategies that take into account the existence of polyhedral faces of $\stdCone$ as in the case of the FRA-Poly algorithm in \cite{LMT18}. 

To conclude this section, we quickly review some  variants of facial reduction.  
The search for efficient ways of doing facial reduction and computing the reducing directions is an 
area of active research.
Permenter, Friberg and Andersen have recently shown that reducing directions can be obtained naturally if we have access to relative interior solutions to a certain self-dual homogeneous model of \eqref{eq:c_primal} and \eqref{eq:c_dual}, see Theorem~3.2 and Section~4 of \cite{PFA17}.

It is also possible to relax the search criteria in order 
to make the problem of finding $d$ more tractable by considering, for example, polyhedral approximations as in the Partial Facial Reduction approach of Permenter and Parrilo \cite{PP14}
or relaxing the definition of reducing direction as in the approach by Friberg \cite{Fr16} by removing the conic constraints.
See also the work of Zhu, Pataki and Tran-Dinh for a heuristic facial reduction algorithm for SDPs in primal standard format \cite{ZPT17}.
In the case of \cite{PP14} and \cite{ZPT17}, the drawback is that facial reduction might end with a face other than $\minFaceD$, although their experiments show that many interesting instances become easier to solve nonetheless. For the approach in \cite{Fr16}, there are some representability issues affecting the cones obtained by intersecting $\stdCone$ with the hyperplanes defined by the reducing directions, see Sections~4 and 6 therein.
In \cite{LMT18}, we proposed ``FRA-Poly'', a two-phase facial reduction algorithm  that takes into consideration the presence of polyhedral faces in the face lattice of $\stdCone$. Instead of performing facial reduction until Slater's condition is satisfied, Phase~1 of the algorithm in \cite{LMT18} regularizes the problem until the so-called \emph{partial polyhedral Slater's condition} is satisfied. Then, in Phase~2, the algorithm jumps directly to the minimal face. An extension of Lemma~\ref{lemma:red_sdp} appropriate
	for FRA-Poly is proved in Lemma~3 of \cite{LMT18}.

%To conclude this section, we emphasize that what sets Algorithm~\ref{alg:fra} apart 
%from other facial reduction approaches in the literature is 
%that all auxiliary problems that must be solved 
%are always primal and dual strongly feasible 
%and that only linear constraints are added, 
%in addition to the conic constraints originating from the original problem. 
%In particular,  we never need to, for example, add a quadratic constraint in order to build the auxiliary problems  \eqref{eq:red} and \eqref{eq:red_dual}.

\section{Double facial reduction}\label{sec:dfr}
In Section~\ref{sec:fr}, we saw how to perform facial reduction by solving auxiliary problems that are always primal and dual strongly feasible. However, as we remarked previously, 
facial reduction only guarantees that one side of the problem will be strongly feasible, after reformulating the problem over the minimal face. In order to finally obtain a problem where both the primal and dual are strongly feasible, we only need to do 
facial reduction \emph{twice}, which is mildly surprising. 
We call this \emph{double facial reduction}.

In this section, we discuss technical aspects related to double facial reduction and how it can be used to 
 compute the optimal value of \eqref{eq:c_dual}. 
 Double facial reduction will also enable us to compute almost optimal solutions when the optimal value of \eqref{eq:c_dual} is not attained, as we will see in Section~\ref{sec:almost}.
We will also show how to obtain almost feasible solutions when \eqref{eq:c_dual} is weakly infeasible, see Section~\ref{sec:inf}.

\subsection{Computing the optimal value of \eqref{eq:c_dual}}\label{sec:dfr_opt}
The first step towards computing the optimal value $\dOpt$ of \eqref{eq:c_dual} is to apply facial reduction to \eqref{eq:c_dual}.  Then, if \eqref{eq:c_dual} is feasible,
% If \eqref{eq:c_dual} turns out to be infeasible, we are done and all that remains is to distinguish between weak and strong infeasibility of the problem, as we shall discuss in Section~\ref{sec:inf}.  
%So, for this subsection, we suppose that \eqref{eq:c_dual} is feasible and  after applying facial reduction to \eqref{eq:c_dual} 
we obtain the following pair of CLPs:

 \noindent{[\bf Primal-dual pair obtained after applying facial reduction to \eqref{eq:c_dual}]}%
\begin{multicols}{2}
	\noindent\begin{align}
	\underset{x}{\inf} & \quad \inProd{c}{x} \label{eq:fr_primal}\tag{\^{P}}\\ 
	\mbox{subject to} & \quad \stdMap x = b \nonumber \\ 
	&\quad x \in (\minFaceD)^* \nonumber
	\end{align}
	\noindent
	\begin{align}
	\underset{y}{\sup} & \quad \inProd{b}{y} \label{eq:fr_dual} \tag{\^{D}} \\ 
	\mbox{subject to} & \quad c - \stdMap ^\T y \in \minFaceD. \nonumber
	\end{align}	
\end{multicols}
Here, \eqref{eq:fr_dual} is strongly feasible, but it could 
still be the case that \eqref{eq:fr_primal} is not strongly feasible. Therefore, when $\stdCone = \PSDcone{n}$, the pair 
\eqref{eq:fr_primal} and \eqref{eq:fr_dual} might still not be 
solvable with $\intOracle$.
To remedy this issue, if $\minFaceD \neq \emptyset$, it is reasonable to consider applying facial reduction 
to \eqref{eq:fr_primal}, which leads to the following pair of problems.

\noindent{\bf [Primal-dual pair obtained after applying facial reduction to \eqref{eq:fr_primal}]}\noindent
\begin{multicols}{2}
	\noindent\begin{align}
	\underset{x}{\inf} & \quad \inProd{c}{x} \label{eq:dfr_primal}\tag{{P}*}\\ 
	\mbox{subject to} & \quad \stdMap x = b \nonumber \\ 
	&\quad x \in \minFace{\text{\ref{eq:fr_primal}}} \nonumber
	\end{align}%
	\noindent
	\begin{align}%
	\underset{y}{\sup} & \quad \inProd{b}{y} \label{eq:dfr_dual} \tag{D*}\\ 
	\mbox{subject to} & \quad c - \stdMap ^\T y \in (\minFace{\text{\ref{eq:fr_primal}}})^*.\nonumber
	\end{align}	
\end{multicols}
Here, $\minFace{\text{\ref{eq:fr_primal}}}$ is the minimal face of $(\minFaceD)^*$ which contains 
the feasible region of \eqref{eq:fr_primal}. Now, 
if $\minFaceD$ and  $\minFace{\text{\ref{eq:fr_primal}}}$ are non-empty, both \eqref{eq:dfr_primal} and \eqref{eq:fr_dual} are  ensured to be strongly feasible. 
However, it is not obvious at all whether \eqref{eq:dfr_dual} still satisfies strong feasibility, since $\minFaceD \subseteq (\minFace{\text{\ref{eq:fr_primal}}})^*$.
After all, $C_1 \subseteq C_2$ does not imply $\reInt C_1 \subseteq \reInt C_2$ in general.

Nevertheless, we will show in this section that, in fact, if \eqref{eq:c_dual} and 
\eqref{eq:c_primal} are both feasible, then \eqref{eq:dfr_dual} will still be strongly feasible.
In addition, if $\minFace{\text{\ref{eq:fr_primal}}}$ is empty, then it is because 
$\dOpt = +\infty$.

In essence, the question boils down to understanding the possible ways 
that the feasibility properties of \eqref{eq:c_dual} might change when 
a single facial reduction step is performed at \eqref{eq:c_primal} and 
$\stdCone$ is replaced by $(\stdCone^*\cap \{f\}^\perp)^*$ in \eqref{eq:c_dual}, for some $f \in \stdCone \cap \matRange \stdMap^\T$. First, we need a few auxiliary results.

\begin{lemma}\label{lemma:reint}
	Let $u \in \reInt \stdCone$, $d \in \stdCone$ and $v \in \tanSpace{d}{\stdCone}$, where $\tanSpace{d}{\stdCone}$ is the tangent space of $\stdCone$ at $d$.
	Then, there is $t > 0$ such that
	 \[u + v + td \in \reInt \stdCone.\]
\end{lemma}  
The intuition for Lemma~\ref{lemma:reint} is as follows. If $v+td$ were a point in $\stdCone$, then 
it would be clear that $u + v +td \in \reInt \stdCone$, by item $(\ref{laux:ri})$ of Lemma~\ref{lemma:aux}. Unfortunately, 
this does not happen in general. However,  as $t$ increases, $v+ td $ gets closer and closer to 
$\stdCone$, so adding $u$ will eventually drag everything to the relative 
interior.
\begin{proof}%todo: more gentle
	Let \[
	C = \{u+v + td \mid t\geq 0 \}.
	\]
	To prove the lemma, it is enough to show that $\reInt C \cap \reInt \stdCone \neq \emptyset$. 
	Suppose, for the sake of obtaining a contradiction, that  
	$\reInt C \cap \reInt \stdCone = \emptyset$. This implies that there is a separating hyperplane 
	\[
	H = \{w \in \ambSpace \mid \inProd{z}{w} = \theta \},
	\]
	such that $H$ properly separates $C$ and $\stdCone$, see Theorem~11.3 in \cite{rockafellar}. We recall that \emph{proper separation} means  that $C$ and $\stdCone$ lie in opposite closed half-spaces defined by $H$ and 
	\emph{$H$ does not contain both sets at the same time}. 
	Without loss of generality, we may assume that $C$ and $\stdCone$ lie in the ``lower'' and ``upper'' closed halfspaces defined by $H$, respectively.
	Therefore, we have
	\begin{equation}\label{eq:hyper_halves}
	\inProd{u}{z} + \inProd{v}{z} + \inProd{td}{z} \leq \theta \leq \inProd{w}{z}, \quad \forall t \geq 0,\,\, \forall w \in \stdCone.
	\end{equation}
	For \eqref{eq:hyper_halves} to hold, we must have $z \in \stdCone ^*$ and 
	$\theta \leq 0$ (since $0 \in \stdCone$). 
	Furthermore, because $d \in \stdCone$ (by assumption) and $z \in \stdCone ^*$,
	we have $\inProd{d}{z} \geq 0$. However, in view of \eqref{eq:hyper_halves},
	it must be the case that 
	\begin{equation}\label{eq:lem_reint_1}
	\inProd{d}{z} = 0,
	\end{equation} since $t$ can be taken to be any nonnegative number. 
	By item~($\ref{laux:conj}$) of Lemma~\ref{lemma:aux}, we conclude that $z \in \minFacePoint{d}{\stdCone}^\Delta$.
	
	From item ($\ref{laux:ts}$) of Lemma~\ref{lemma:aux}, we have $\tanSpace{d}{\stdCone} = \minFacePoint{d}{\stdCone}^{\Delta\perp}$.
	Therefore,
	\begin{equation}\label{eq:lem_reint_2}
	\inProd{v}{z} = 0.
	\end{equation}
	From \eqref{eq:hyper_halves}, \eqref{eq:lem_reint_1}, \eqref{eq:lem_reint_2} and 
	recalling that $\theta \leq 0$, we obtain $\inProd{u}{z} = 0$. 
	This implies that $C \subseteq H$ and $\theta = 0$. 
	By item $(\ref{laux:ri3})$ of Lemma~\ref{lemma:aux}, we have $\stdCone \subseteq H$ as well, since $u \in \reInt \stdCone$ by assumption. This contradicts 
	the properness of the separation.
\end{proof} 

We are now ready to state a result on the conservation of feasibility after one facial reduction step. For what follows, we recall 
that \emph{\eqref{eq:c_dual} is in weak status} if it is weakly feasible or weakly infeasible. We also recall the following basic facts. A face 
$\stdFace$ of $\stdCone$ always satisfies $\stdFace = \stdCone \cap \spanVec \stdFace$, therefore we have
\begin{equation}
\stdFace^* = \closure (\stdCone^* + \stdFace^\perp). \label{eq:facedual}
\end{equation}
Also, if $C_1$ and $C_2$ are two convex sets we have 
$\reInt (C_1+C_2) = \reInt C_1 + \reInt C_2$, $\reInt (\closure C_1) = \reInt C_1$.
%todo: write about the role of f
\begin{proposition}[Conservation of feasibility]\label{prop:feas_changes}
	Let $f\in \stdCone \cap \matRange \stdMap^\T$ and 
	let $\stdFace \coloneqq \stdCone^* \cap \{f\}^\perp = \minFacePoint{f}{\stdCone}^\Delta$ (see item~$(\ref{laux:conj})$ of Lemma~\ref{lemma:aux}). Let \eqref{eq:d_p} be the problem obtained by replacing 
$\stdCone$ by $\stdFace^*$ in \eqref{eq:c_dual}, i.e., 
\begin{align}
\underset{y}{\sup} & \quad \inProd{b}{y} \label{eq:d_p} \tag{{D}'} \\ 
\mbox{subject to} & \quad c - \stdMap ^\T y \in \stdFace^*. \nonumber
\end{align}
 We have the following  
relations:
\begin{enumerate}[$(i)$]
	\item  \eqref{eq:c_dual} is strongly feasible if and only if \eqref{eq:d_p} is;
	\item  \eqref{eq:c_dual} is strongly infeasible if and only if \eqref{eq:d_p} is;
	\item  \eqref{eq:c_dual} is in weak status if and only if \eqref{eq:d_p} is.
\end{enumerate}
\end{proposition}
\begin{proof}
	\begin{enumerate}[$(i)$]
	\item First, since  $	\stdFace^* = \closure (\stdCone + \stdFace ^\perp)$ and $\reInt \stdFace^\perp = \stdFace^\perp$, we have
	\begin{equation}
	\reInt \stdFace^* = \reInt (\closure (\stdCone + \stdFace ^\perp)) = (\reInt \stdCone) + \stdFace ^\perp \label{eq:reint_f}.
	\end{equation}
	
	Now, suppose that  \eqref{eq:c_dual} is strongly feasible.
	Since $\reInt \stdCone \subseteq \reInt \stdCone +\stdFace^\perp$, we conclude that \eqref{eq:d_p} must be strongly feasible as well.
	
	Conversely, suppose that \eqref{eq:d_p} is strongly feasible and let $s = c- \stdMap^\T y$ be such that $s \in \reInt \stdFace^*$. 
	By \eqref{eq:reint_f}, 
	we have 
	\[
	s = u+v,	
	\]	
	where $u \in \reInt \stdCone$ and $v \in \stdFace ^\perp $.
	By items ($\ref{laux:conj}$) and ($\ref{laux:ts}$) of Lemma~\ref{lemma:aux}, we have
	\[\stdFace ^\perp = \minFacePoint{f}{\stdCone}^{\Delta\perp} =  \tanSpace{f}{\stdCone}.\] 
	Invoking Lemma \ref{lemma:reint} we conclude that there exists $t > 0$ such that \[u + v + tf \in \reInt \stdCone .\]
	Since $f \in \matRange \stdMap^\T$, there exists $\hat y$ such that 
	$f = -\stdMap^\T \hat y$. We conclude that 
	\[
	s+ tf = u + v + tf = c - \stdMap^\T(y+\hat y) \in \reInt \stdCone,
	\]
	thus showing that \eqref{eq:c_dual} is strongly feasible.
	\item  Since $\stdCone \subseteq \stdFace^*$, if 
	\eqref{eq:d_p} is strongly infeasible, \eqref{eq:c_dual} must be strongly infeasible as well.
	
	Conversely, suppose that \eqref{eq:c_dual} is strongly infeasible.
	Then, Proposition~\ref{prop:infeas} implies the existence of $x$
	satisfying
	\[
\inProd{c}{x} = -1, \quad x \in \stdCone^* \cap \ker \stdMap.	
	\]
	Since $x \in \ker \stdMap$ and $f \in \matRange \stdMap^\T$, 
	we have $\inProd{x}{f} = 0$. So, in fact, $x \in \stdFace$. 
	Therefore, by Proposition~\ref{prop:infeas}, the same 
	$x$ attests that \eqref{eq:d_p} is strongly infeasible.

	\item First, we recall that the four feasibility statuses described in Section~\ref{sec:type} are mutually exclusive. Next, suppose that \eqref{eq:c_dual} is in weak status, i.e., it is either weakly feasible or 
	weakly infeasible. By items $(i)$ and $(ii)$ proved so far, \eqref{eq:d_p} cannot be strongly infeasible nor strongly feasible because that would imply that \eqref{eq:c_dual} has that same feasibility status. Therefore, \eqref{eq:d_p} must be in weak status as well.
	
	Conversely, suppose that \eqref{eq:d_p} is in weak status. Again, items $(i)$ and $(ii)$ imply that the only two possibilities for \eqref{eq:c_dual} are weak infeasibility or weak feasibility.
\end{enumerate}
\end{proof}
%Typically, we will invoke Proposition~\ref{prop:feas_changes} taking $f$ to be a reducing direction for \eqref{eq:c_dual} (Definition~\ref{def:red}). Nevertheless, $f$ need not to be a reducing direction in order for Proposition~\ref{prop:feas_changes} to hold, since 
We can now state and prove our main result on the preservation of feasibility status after facial reduction is performed on \eqref{eq:c_primal}. Intuitively, Theorem~\ref{theo:min_changes} means the following. Whenever facial reduction is applied to, say, 
\eqref{eq:c_primal}, we obtain a new problem which is ensured to be strongly feasible, if \eqref{eq:c_primal} is feasible. This new problem will also have a dual problem whose feasibility properties might be different than the original dual problem \eqref{eq:c_dual}.
However, Theorem~\ref{theo:min_changes} says that no drastic changes are allowed, i.e., if \eqref{eq:c_dual} was strongly feasible to begin with, it will stay strongly feasible. The only possible room for change is that a weakly feasible/infeasible problem might become weakly infeasible/feasible. Theorem~\ref{theo:min_changes} also contains the relatively surprising fact that strong feasibility of the new dual implies  strong feasibility of \eqref{eq:c_dual}.
\begin{theorem}[Preservation of feasibility under facial reduction]\label{theo:min_changes}
	Let $\minFaceP$ denote the minimal face of $\stdCone^*$ that 
	contains the feasible region of \eqref{eq:c_primal} and suppose that $\minFaceP \neq \emptyset$.
	Consider the problem obtained by replacing $\stdCone$ by $(\minFaceP)^*$ in \eqref{eq:c_dual}, i.e.,
	\begin{align}
	\underset{y}{\sup} & \quad \inProd{b}{y} \label{eq:c_dual_p} \tag{Conic-D-FP} \\ 
	\mbox{subject to} & \quad c - \stdMap ^\T y \in (\minFaceP)^*, \nonumber
	\end{align}	
	The following hold.
	\begin{enumerate}[$(i)$]
		\item \eqref{eq:c_dual} is strongly feasible if and only if \eqref{eq:c_dual_p} is strongly feasible;
		\item \eqref{eq:c_dual} is strongly infeasible if and only if \eqref{eq:c_dual_p} is strongly infeasible;
		\item \eqref{eq:c_dual} is in weak status if and only if \eqref{eq:c_dual_p} is in weak status. %\textbf{weakly feasible}.
	\end{enumerate}
\end{theorem}
\begin{proof}
	Applying facial reduction to \eqref{eq:c_primal} (e.g., Algorithm~\ref{alg:fra}), we see that $\minFaceP$ can be written 
	as
	\[
	\minFaceP =  \stdCone^* \cap \{f_1\}^\perp \cap \cdots \cap \{f_\ell \}^\perp,
	\]
	where each $f_i$ satisfies  
	\[
	f_i \in \left( \stdCone^* \cap \{f_1\}^\perp \cap \cdots \{f_{i-1}\}^\perp \right)^*\cap \matRange \stdMap ^\T.
	\] 
	Now, denote by $(D_i)$ the problem obtained by replacing 
	$\stdCone$ by $\left( \stdCone^* \cap \{f_1\}^\perp \cap \cdots \{f_{i-1}\}^\perp \right)^*$ in \eqref{eq:c_dual}. We observe 
	the following:
	\begin{enumerate}
		\item $(D_1)$ and $(D_{\ell+1})$ are precisely \eqref{eq:c_dual} and \eqref{eq:c_dual_p}, respectively.
		\item $f_i$ is a reducing direction for $(D_i)$, so Proposition~\ref{prop:feas_changes} applies to $f_i$, $(D_i)$ and $(D_{i+1})$, for $i = 1,\ldots, \ell$.
	\end{enumerate}
	By induction, we conclude that items $(i)$, $(ii)$ and $(iii)$ hold.
%	Furthermore, we have that \eqref{eq:c_dual} is in weak status if and only if \eqref{eq:c_dual_p} is in \textbf{weak status}.
%	All that remains is to show that \eqref{eq:c_dual_p}  cannot be weakly infeasible. 
%	
%	Consider the primal counterpart of \eqref{eq:c_dual_p}, i.e., 
%	the problem of minimizing $\inProd{c}{x}$ subject to 
%	$\stdMap x = b$ and $x \in \minFaceP$. Denote its optimal value by $\hat \theta$. Since we are assuming that $\minFaceP \neq \emptyset$, if  $\hat \theta$ turns out to be finite, \eqref{eq:c_dual_p} must have an optimal solutions affording the same value. In particular,
%	must be feasible.
%	
%	Therefore, under the assumption that $\minFaceP \neq \emptyset$, the only way that \eqref{eq:c_dual_p} could possibly be infeasible is if 
%	$\hat \theta = -\infty$. So, suppose that $\hat \theta = -\infty$. 
%	Then, there exists a sequence $\{x^k\}$   	
\end{proof}

We are now in position to state our main result on double facial reduction.

\begin{theorem}[Double facial reduction]\label{theo:double}
	Suppose $\minFaceD \neq \emptyset$ and consider the problems  \eqref{eq:fr_primal} and \eqref{eq:fr_dual} above. 
%	be the problem 
%	obtained by replacing $\stdCone^*,\stdCone$ by $(\minFaceD)^*$, $\minFaceD$ in \eqref{eq:c_primal} and \eqref{eq:c_dual}, respectively.
	
	Let $\minFace{\text{\ref{eq:fr_primal}}}$ be the minimal face of $(\minFaceD)^*$ that contains the feasible region of \eqref{eq:fr_primal}.
	Consider the pair of problems \eqref{eq:p_min} and \eqref{eq:p_min_d}, which we repeat below for convenience.%
	\begin{multicols}{2}\noindent
		\begin{align}
		\underset{x}{\inf} & \quad \inProd{c}{x} \label{eq:p_min}\tag{P*}\\ 
		\mbox{subject to} & \quad \stdMap x = b \nonumber \\ 
		&\quad x \in \minFace{\hat P} \nonumber 
		\end{align}
		\begin{align}
		\underset{y}{\sup} & \quad \inProd{b}{y} \label{eq:p_min_d} \tag{D*} \\ 
		\mbox{subject to} & \quad c - \stdMap ^\T y \in (\minFace{\hat P})^* \nonumber.
		\end{align}	
	\end{multicols}\noindent	
	The following hold.
	\begin{enumerate}[$(i)$]
		\item The optimal value of \eqref{eq:c_dual} ($\dOpt$) is finite if and only if $\minFace{\text{\ref{eq:fr_primal}}} \neq \emptyset$.
		In this case,  	\eqref{eq:p_min} and \eqref{eq:p_min_d} are both strongly feasible and 
		\[
		\dOpt = \opt{\text{\ref{eq:p_min}}} = \opt{\text{\ref{eq:p_min_d}}}.
		\]
		\item $\dOpt = + \infty$ if and only if $\minFace{\text{\ref{eq:fr_primal}}}  = \emptyset$.
	\end{enumerate}
	
\end{theorem} 
\begin{proof}
	\begin{enumerate}[$(i)$]
		\item 
		Suppose that the optimal value of \eqref{eq:c_dual} is finite.	
		Then, by Proposition~\ref{prop:slater}, the optimal value of \eqref{eq:fr_primal}  must be 	equal to $\dOpt$, since \eqref{eq:fr_dual} is strongly feasible. In particular, 
		\eqref{eq:fr_primal} must be feasible and, therefore, 
		$\minFace{\text{\ref{eq:fr_primal}}} \neq \emptyset$.
		Since $\minFace{\text{\ref{eq:fr_primal}}}$ is the minimal face of $(\minFaceD)^*$ that contains the feasible region of \eqref{eq:fr_primal}, \eqref{eq:p_min} is strongly feasible and its optimal value must coincide 
		with the optimal value of \eqref{eq:fr_primal}, which is $\dOpt$.
		
		Next, since  \eqref{eq:p_min} is strongly feasible and has finite optimal value, \eqref{eq:p_min_d} must have the same optimal value.  Therefore, as stated, we have
		\[
			\dOpt = \opt{\text{\ref{eq:p_min}}} = \opt{\text{\ref{eq:p_min_d}}}.
		\]
		
		By item $(i)$ of Theorem \ref{theo:min_changes}, 
		substituting $\minFaceD$ by $(\minFace{\text{\ref{eq:fr_primal}}})^*$ preserves strong feasibility, 		so \eqref{eq:p_min} and \eqref{eq:p_min_d} are both strongly feasible.

		Conversely, suppose that $\minFace{\text{\ref{eq:fr_primal}}} \neq \emptyset$.
		This means that \eqref{eq:p_min} is feasible.
		So \eqref{eq:fr_primal} must be feasible as well, because any feasible solution to \eqref{eq:p_min} must be a feasible solution to \eqref{eq:fr_primal}. Since 
		we are assuming that $\minFaceD \neq \emptyset $, \eqref{eq:fr_dual} must be feasible as well. Therefore, \eqref{eq:fr_primal} and \eqref{eq:fr_dual} are feasible primal and dual problems sharing the same optimal value, which must be finite. Since \eqref{eq:fr_dual} shares the same optimal value with \eqref{eq:c_dual}, we conclude that $\dOpt$ is indeed finite.
		
		\item It follows from item $(i)$. 
	\end{enumerate}
\end{proof}

The conclusion is that, when $\dOpt$ is finite, 
the pair of problems 
\eqref{eq:p_min} and \eqref{eq:p_min_d} are both strongly feasible.
When $\stdCone = \PSDcone{n}$, they can indeed be solved by $\intOracle$ in order to obtain $\dOpt$ by Proposition~\ref{prop:refor}. At this stage, even though $\dOpt$ might have been unattained for \eqref{eq:c_dual}, 
\eqref{eq:p_min_d} is never hindered by unattainment. 

The problem, however, is that a feasible 
solution to \eqref{eq:p_min_d} might not be feasible 
to \eqref{eq:c_dual}. And, indeed, if $\dOpt$ is finite 
but not attained, even though \eqref{eq:p_min_d} has an optimal solution, \eqref{eq:c_dual} will not have optimal solutions.
When $\dOpt$ is finite but not attained, the best we can do 
is to compute some solution $y_{\epsilon}$ satisfying $\inProd{b}{y_{\epsilon}} \geq \dOpt -\epsilon$, for some arbitrary $\epsilon > 0$. We will discuss this issue 
in the next subsection.

Before we move on, we give an intuitive explanation of why 
unattainment disappears when doing facial reduction.
%In fact, there are a few ways of explaning intuitively why unattainment disappears. 
The ``dual'' explanation is that strong feasibility is satisfied 
at \eqref{eq:p_min} (which is the Lagrangian dual of \eqref{eq:p_min_d}), so, of course, \eqref{eq:p_min_d} must be attained when $\dOpt$ value is finite.

We now give an ``primal'' explanation of why unattainment disappears. 
Suppose that the optimal value $\dOpt$ of \eqref{eq:c_dual} is finite 
but not attained. Then, there exists a sequence $\{y^k\}$ such that $\inProd{b}{y_k} \to \dOpt$ and \[s^k \coloneqq c-\stdMap^\T y^k \in \stdCone,\quad \forall k.\] 
 As we are assuming unattainment and $\feasS$ is closed, 
$\{y^k\}$ cannot be bounded. Passing to a subsequence if necessary, 
we may assume that $\norm{y^k} \to \infty$ and $y^k/\norm{y^k}$ converges 
to some $ y \in \Re^m$ and that $\dOpt \geq \inProd{b}{y^k} \geq \dOpt-1$ for every $k$.  Dividing $c-\stdMap^\T y^k \in \stdCone$
and $\dOpt \geq \inProd{b}{y^k} \geq \dOpt-1$ by $\norm{y^k}$ and taking limits, we conclude that $f \coloneqq - \stdMap^\T  y$ satisfies
\[
f \in \stdCone,\quad f \in \matRange \stdMap^\T, \quad  \inProd{b}{{y}} = 0
\]
and ${y}$ must be nonzero.
This shows that $(f,y)$ is a reducing direction for \eqref{eq:c_primal}, see Definition~\ref{def:red}. In fact, with more effort, we
can show that there must be at least one pair $(f,y)$ as above satisfying 
$f \not \in \lineality \stdCone$, see  Lemma~\ref{lem:fr}.
In other words, a necessary condition for 
unnatainment of \eqref{eq:c_dual} is the existence of $(f,y)$ as above with $f \not \in \lineality \stdCone$.
Informally speaking, $(f,y)$ acts as a ``recession direction'' for the problem \eqref{eq:c_dual}. 
This suggests that one possible way of fixing unattainment is by 
preventing $f$ from becoming a recession direction. This is accomplished, for instance, 
by substituting $\stdCone$ by $\closure (\stdCone+\spanVec\{f\})$, so 
that $f \in \lineality (\closure (\stdCone+\spanVec\{f\}))$.
However, $\closure (\stdCone+\spanVec\{f\})$ is equal to 
$(\stdCone^*\cap\{f\}^\perp)^*$, which corresponds to a single 
facial reduction step done at \eqref{eq:c_primal}.

In other words, from the point of view of \eqref{eq:c_dual}, facial reduction done at \eqref{eq:c_primal} removes recession directions that affect attainment. We remark that  Abrams~\cite{Abrams75} also proposed a regularization procedure that removes recession directions, in order to fix unattainment in convex programming.

\subsection{Obtaining feasible almost optimal solutions}\label{sec:almost}
The pair of problems \eqref{eq:p_min_d} and \eqref{eq:p_min} are strongly feasible 
and their common optimal value is $\dOpt$. However, an optimal solution to 
\eqref{eq:p_min_d} is unlikely to be feasible for \eqref{eq:c_dual}. In fact, it may 
happen that $\dOpt$ is not attained, in which case  \eqref{eq:c_dual} has no optimal solution at all.

Nevertheless, we will show how to construct feasible solutions that are almost optimal for \eqref{eq:c_dual} using 
the directions obtained when calling Algorithm~\ref{alg:fra} with \eqref{eq:fr_primal} as 
input, see Algorithm~\ref{alg:eps}.
%We will denote by 
%$y^*$ an optimal solution to \eqref{eq:p_min_d} and by $s^*$ the corresponding 
%slack $s^* = c - \stdMap ^\T y^*$.

\begin{algorithm}[h]
	\caption{Finding an $\epsilon$-optimal solution to \eqref{eq:c_dual} }\label{alg:eps}
	\KwIn{
		\begin{enumerate}
			\item Reducing directions for \eqref{eq:fr_primal} (Definition~\ref{def:red}): $(f_1,y_1), \ldots, (f_{\ell _2},y_{\ell_2})$.
			\item $\hat y$ such that $c - \stdMap^\T \hat y \in \reInt \minFaceD$,
			\item an optimal solution $y^*$ to \eqref{eq:p_min_d}.
			\item $\epsilon > 0$
		\end{enumerate}
	}
	\KwOut{ A feasible solution $y_\epsilon$ to \eqref{eq:c_dual} satisfying 
		$\inProd{b}{y_\epsilon} \geq \dOpt - \epsilon$.
	}
	\eIf{$\inProd{b}{\hat{y}} \geq \dOpt - \epsilon$}
	{ 
				\Return{} $\hat y$\label{alg:eps:ret}\\
	}
	{
		$\beta \leftarrow \frac{\dOpt-\inProd{b}{\hat{y}} - \epsilon}{\dOpt-\inProd{b}{\hat{y}}}$\\
		$w_{\ell _2+1} \leftarrow \beta y^* + (1-\beta)\hat y $ \label{alg:eps:start} \\
		$\stdFace _{\ell _2+1} \leftarrow(\minFaceD)^*\cap \{f_1\}^\perp \cap \cdots \cap \{f_{\ell _2}\}^\perp$ \\
		\For{$i=\ell _2$ \KwTo $1$}
		{
			$\stdFace _i \leftarrow(\minFaceD)^*\cap \{f_1\}^\perp \cap \cdots \cap \{f_{i-1}\}^\perp$\\
			Find $\alpha _i$ positive such that $
			c -\stdMap^\T (w_{i+1} + \alpha _i y_{i})  = c -\stdMap^\T w_{i+1} + \alpha _i f_i \in \reInt \stdFace _{i}^*
			$\label{alg:eps:search} \\ 
			$w_{i} \leftarrow w_{i+1} + \alpha _i y_{i}$
		}
		\Return{} $w_1$
	}
\end{algorithm}
Note that the inner loop in Algorithm~\ref{alg:eps} goes from $\ell _2$ to $1$. This is because we start from a 
relative interior solution to $\stdFace _{\ell _2+1}$ and we have to work our 
way until the bottom of the chain $\minFaceD$.  
The tricky part is ensuring 
that at each step there is indeed an $\alpha _i$ as in Line~\ref{alg:eps:search}.
If there is at least one $\alpha _i$, then any number larger than $\alpha _i$ will work. 
Therefore, it is enough to keep trying larger and larger numbers until the condition 
in Line~\ref{alg:eps:search} is met. We will now show that an appropriate $\alpha _i$
always exists and that Algorithm~\ref{alg:eps} is indeed correct.
For that, we need a few auxiliary results. First, suppose that $f \in \stdFace^*$, for $\stdFace$ a closed convex cone. We have 
by items~$(\ref{laux:conj})$ and $(\ref{laux:ts})$ of Lemma~\ref{lemma:aux} and \eqref{eq:facedual} that
\begin{equation}
(\stdFace \cap \{f\}^\perp)^*  = (\minFacePoint{f}{\stdFace^*}^{\Delta})^*  = \closure(\stdFace^* + \minFacePoint{f}{\stdFace^*}^{\Delta\perp}) = 
\closure(\stdFace^* + \tanSpace{f}{\stdFace^*}). \label{eq:lem_aux}
\end{equation}
\begin{lemma}\label{lem:ri}
Suppose that $s \in (c+\matRange \stdMap^\T)\cap \reInt \stdCone$ and let $\minFaceP$ be the minimal face of $\stdCone^*$ that contains the feasible region of \eqref{eq:c_primal}.
If $\minFaceP \neq \emptyset$, then $s \in \reInt ((\minFaceP)^*)$.
	
\end{lemma}
\begin{proof}
	It is a consequence of the proof of Theorem~\ref{theo:min_changes}. 
	Nevertheless, we will work out the details here.
	As in the proof of Theorem~\ref{theo:min_changes}, 	applying facial reduction to \eqref{eq:c_primal} (e.g., Algorithm~\ref{alg:fra}), we see that $\minFaceP$ can be written 
	as
	\[
	\minFaceP =  \stdCone^* \cap \{f_1\}^\perp \cap \cdots \cap \{f_\ell \}^\perp,
	\]
	where each $f_i$ satisfies  
	\[
	f_i \in \left( \stdCone^* \cap \{f_1\}^\perp \cap \cdots \{f_{i-1}\}^\perp \right)^*\cap \matRange \stdMap ^\T.
	\] 
	Now, let $\stdFace _i \coloneqq \left( \stdCone^* \cap \{f_1\}^\perp \cap \cdots \cap \{f_{i-1}\}^\perp \right)$. We observe 
	the following:
	 
	\begin{enumerate}
		\item $\stdFace _1 = \stdCone^*$ and $\stdFace _{\ell+1} = \minFaceP$,
		\item $\stdFace _i \leftarrow \stdFace _{i-1} \cap \{f_{i-1} \}^\perp$, for all $i > 1$.
	\end{enumerate}
	By hypothesis, we have $ s \in \reInt \stdFace _1^*$ and by \eqref{eq:lem_aux} we have
	\begin{align*}
	\stdFace _2^* = \closure (\stdFace _1^* + \tanSpace{f_1}{\stdFace_1^*})\\
	\reInt \stdFace _2^* = (\reInt \stdFace_1^*) + \tanSpace{f_1}{\stdFace_1^*}.
	\end{align*}
	Therefore, $s \in \reInt \stdFace _2^*$ as well. At the $i$-th step, we have:
	\begin{align*}
	\stdFace _i^* = \closure (\stdFace _{i-1}^* + \tanSpace{f_{i-1}}{\stdFace_{i-1}^*})\\
	\reInt \stdFace _i^* = (\reInt \stdFace_{i-1}^*) + \tanSpace{f_{i-1}}{\stdFace_{i-1}^*}.
	\end{align*}
	By induction, we conclude that $s \in \reInt \stdFace _i^*$ for every $i$.
	In particular, $s \in \reInt ((\stdFace_{\ell+1})^*) = \reInt ((\minFaceP)^*)$.
\end{proof}

\begin{theorem}
	Algorithm \ref{alg:eps} is correct, that is, the output $y_\epsilon$ is indeed a feasible solution 
	to \eqref{eq:c_dual} satisfying $\inProd{b}{y_\epsilon} \geq \dOpt - \epsilon$.
\end{theorem}\label{theo:alg_eps}
\begin{proof}
	\fbox{$y_\epsilon$ is $\epsilon$-optimal.} By construction, 
	$\inProd{b}{w_{\ell _2+1}} \geq \dOpt -\epsilon$. Moreover, all the 
	$y_i$ satisfy $\inProd{b}{y_i} = 0$. Therefore, $\inProd{b}{y_\epsilon} \geq \dOpt -\epsilon$.
	
	\fbox{$y_\epsilon$ is feasible for \eqref{eq:c_dual}.} If the algorithm stops before 
	Line~\ref{alg:eps:start},  $y_\epsilon$ is feasible because $\minFaceD \subseteq \stdCone$. So suppose that we have reached Line~\ref{alg:eps:start}.
	Since $\stdFace _1^* = \minFaceD$, 
	if Line~\ref{alg:eps:search} is correct, then  $y_\epsilon$ is feasible 
	for \eqref{eq:c_dual}. We now show that Line~\ref{alg:eps:search} is indeed correct.
	
	Let 
	\begin{align*}
	\hat{s} & \coloneqq c-\stdMap^\T \hat y,\\
	s_{i} & \coloneqq c-\stdMap^\T w_{i}, \text{ for } i = 1, \ldots, \ell_2+1.
	\end{align*}
	
	We have $\hat{s} \in \reInt \minFaceD$ and, by Lemma \ref{lem:ri}, 
	$\hat s \in \reInt((\minFace{\text{\ref{eq:fr_primal}}})^*)$ as well.
	Note that $s_{\ell _2+1}$ is a strict convex combination of 
	$c-\stdMap^\T y^*$ and $\hat s$, see Line~\ref{alg:eps:start}. These points belong to $(\minFace{\text{\ref{eq:fr_primal}}})^*$ and  $\reInt((\minFace{\text{\ref{eq:fr_primal}}})^*)$, respectively, so  $s_{\ell _2+1}$ must belong to   $\reInt((\minFace{\text{\ref{eq:fr_primal}}})^*)$ as well. In addition, $s_{\ell _2+1}$ is a feasible 
	slack for \eqref{eq:p_min_d}.

	Now suppose  that we have shown that $s_{i+1} \in   \reInt  \stdFace _{{i+1}}^*$, for some $i$. By \eqref{eq:lem_aux}, we have
	\begin{align*}
	\stdFace _{{i+1}}^* & = (\stdFace _{i} \cap \{f_i\}^\perp)^*  = \closure (\stdFace _{i}^* + \tanSpace{f_i}{\stdFace _i^*}) \\
	\reInt \stdFace _{{i+1}}^* & =  (\reInt \stdFace _{i}^*) + \tanSpace{f_i}{\stdFace _i^*}.
	\end{align*}
	Therefore, $s_{i+1} = u_i + v_i $ for some $u_i \in \reInt \stdFace _{i}^*$ and
	$v_i \in \tanSpace{f_i}{\stdFace _i^*}$. 
	We 
	can apply Lemma \ref{lemma:reint} to $u_i,v_i,f_i, \stdFace _i^*$ and conclude the 
	existence of positive $\alpha _i$ such that $s_{i+1} + \alpha f_i$ belongs to 
	$\reInt \stdFace _i^*$. Therefore,
	\[
	c - \stdMap^\T(w_{i+1} + \alpha y_i) \in \reInt \stdFace _i^*.
	\]
	In other words, $s_i = c-\stdMap^\T w_{i} \in \reInt  \stdFace _{{i}}^*$. 
	
	By induction, we conclude 
	that at each iteration it is possible to find  $\alpha _i$ as stated in Line~\ref{alg:eps:search}.
\end{proof}

\subsubsection{Computational aspects of Algorithm~\ref{alg:eps}}\label{sec:eps_cost}
Having proved the correctness of Algorithm~\ref{alg:eps} in Theorem~\ref{theo:alg_eps}, we discuss the computation of $\alpha _i$ in Line~\ref{alg:eps:search}, which is the most computationally expensive part of the algorithm. As remarked previously, the existence of 
$\alpha _i$ follows from Lemma~\ref{lemma:reint}. So, we will discuss the computation of $t$ as in Lemma~\ref{lemma:reint}. 

Let $u,v,d$ be as in Lemma~\ref{lemma:reint}, i.e., $u \in \reInt \stdCone$, $d \in \stdCone$ and $v \in \tanSpace{d}{\stdCone}$. 
As we remarked before Lemma~\ref{lem:ri}, if $t > 0 $ is such 
that $u+v +td \in \reInt\stdCone$, then any $\hat t \geq t$ will 
also work. So, the simplest algorithm for computing $t$ starts with 
some arbitrary positive value and keeps doubling it, until $u+v +td \in \reInt\stdCone$.
Whether this is computationally reasonable or not depends on how hard it is to decide 
membership in $\stdCone$ and $\reInt \stdCone$.
If $\stdCone=\PSDcone{n}$ or $\stdCone$ is as in 
Proposition~\ref{prop:refor}, the membership problem is not too expensive in contrast to the situation where $\stdCone$ is, say, a completely positive cone.
%Either way, being able to decide membership seems to be a computationally gentler assumption than the capability of solving conic linear programs.

%In the context of semidefinite programming, 
%In our context, where we assume the ability to solve strongly feasible conic linear programs over $\stdCone$, 
%Furthermore, although there are exceptions, deciding membership over a cone $\stdCone$ seems to be significantly less expensive than solving conic linear programs over $\stdCone$.
%In particular, We note that being to able decide membership over $\PSDcone{n}$ (or its faces) is, in principle, a weaker assumption than having access to $\intOracle$. [\textbf{Note}: I am considering adding a more detailed explanation on that]
%To explain that, let $e \in \reInt \stdCone$ be arbitrary.
%This is because $s \in \stdCone $ if and only if the following pair of strongly feasible primal dual problems have optimal value zero\[
%\inf \{\inProd{s}{x} \mid \inProd{x}{e} = 1,  x\in \stdCone^*\} = \sup \{t \mid s -te \in \stdCone \}
%\]
%In addition, the problem  
%\[
%\sup \{-t \mid u+v +td \in \stdCone \}
%\]
%looks quite similar to a minimum eigenvalue computation. 
We also recall 
that given $s \in \S^n$, its minimum eigenvalue $\lambda_{\min}(s)$ satisfies
\[
\lambda _{\min}(s) = \sup\, \{t \mid s -tI_n \in \PSDcone{n} \} = \sup\,\{-t \mid s +tI_n \in \PSDcone{n} \}, 
\]
where $I_n$ is the $n\times n$ identity matrix. Accordingly, the line search problem of finding $t$ with  $u+v +td \in \reInt\stdCone$ seems quite akin to a minimum eigenvalue computation. In the context of semidefinite programming, although one could use $\intOracle$ to solve the membership problem, it seems  more reasonable to solve it directly via minimum eigenvalue computations and/or factorizations.

Nevertheless, for the sake of completeness, we show that for arbitrary $\stdCone$, we can 
obtain $t$ by solving a pair of primal and dual strongly feasible problems.
First, we consider the following pair of problems:
\begin{multicols}{2}
	\noindent\begin{align}
	\underset{x}{\inf} & \quad \inProd{u+v}{x} \label{eq:alpha_primal}\tag{$P_d$}\\ 
	\mbox{subject to} & \quad \inProd{d}{x} = 1 \nonumber \\ 
	&\quad x \in \stdCone^* \nonumber
	\end{align}%
	\noindent
	\begin{align}%
	\underset{t}{\sup} & \quad -t \label{eq:alpha_dual} \tag{$D_d$}\\ 
	\mbox{subject to} & \quad u+v +td \in \stdCone. \nonumber
	\end{align}	
\end{multicols}
Lemma~\ref{lemma:reint} guarantees that \eqref{eq:alpha_dual} is strongly feasible, so we can apply Lemma~\ref{lemma:red_sdp} to the pair \eqref{eq:alpha_primal} and \eqref{eq:alpha_dual} by replacing 
$b,c$ by $-1$ and $u+v$ respectively and $\stdMap^*$ by the map 
that takes $t$ to $-td$. From item $(iii)$ of Lemma~\ref{lemma:red_sdp}, if we solve the pair \eqref{eq:red} and \eqref{eq:red_dual} we will obtain $t$ such that $u+v +td \in \reInt\stdCone$. If $t$ turns out to be negative, we can just set it to zero.
%
%Therefore, it is enough to solve 
%the pair \eqref{eq:red} and \eqref{eq:red_dual}, in order to obtain $t$. In particular, this can be done with $\intOracle$, when $\stdCone$ is $\PSDcone{n}$ or a face of $\PSDcone{n}$. This leads immediately to the following result.
%\begin{proposition}\label{prop:eps_cost}
%When $\stdCone=\PSDcone{n}$, Algorithm~\ref{alg:eps} can be implemented by invoking $\intOracle$ at most $\ell_2$ times, where $\ell_2$ is the number of reducing directions (see input 1. of Algorithm~\ref{alg:eps}). 
%In particular, since the number of reducing directions is bounded above  $n+1$, we also have that $n+1$ is an upper bound for the number of times $\intOracle$ is invoked.
% \end{proposition}
%We remark, however, that Proposition~\ref{prop:eps_cost} seems to be exceedingly conservative. 

%That said, if $\stdCone$ is a computationally difficult cone such as the cone of completely positive matrices, \eqref{eq:alpha_dual} can still be quite hard.
%On the other hand, the situation is less clear if $\stdCone$ is a computationally difficult cone such as the cone of completely positive matrices.

Moving on, we also notice the following curious feature of Algorithm~\ref{alg:eps}. Except for 
the problem of finding $\alpha _i$ in Line~\ref{alg:eps:search}, the complexity of Algorithm~\ref{alg:eps} does not depend on $\epsilon$.
Decreasing $\epsilon$, however, might lead to larger $\alpha _i$ in 
Algorithm~\ref{alg:eps}. 

Finally, we compare Algorithm~\ref{alg:eps} to an elementary method for computing an $\epsilon$-optimal solution. Namely, once $\dOpt$ is known and given some fixed $\epsilon > 0$, the naive approach is to directly apply a feasibility algorithm to the problem of finding  a point in the set
\begin{equation}\label{eq:naive}
\{y \in \Re^m \mid c - \stdMap^\T y \in \minFaceD, \inProd{b}{y} \geq \dOpt - \epsilon  \}.
\end{equation}
The set in \eqref{eq:naive} can be expressed as the feasible set of a conic linear program over 
$\minFaceD \times \Re_+$:
\begin{align}
\underset{y}{\sup} & \quad 0 \label{eq:dual_naive} \tag{Naive} \\ 
\mbox{subject to} & \quad c - \stdMap ^\T y \in \minFaceD , \nonumber\\
& \quad   \epsilon -\dOpt + \inProd{b}{y} \in \Re_+. \nonumber
\end{align}	
Since \eqref{eq:fr_dual} is strongly feasible, there  is  $y_{\epsilon}$ such that $c- \stdMap^\T y_{\epsilon} \in \reInt \minFaceD$ and $\inProd{b}{y_{\epsilon}} > \dOpt - \epsilon$,
thus showing that  \eqref{eq:dual_naive} is strongly feasible\footnote{A  way to see that this must be true is through the correctness of Algorithm~\ref{alg:eps}. Lines~\ref{alg:eps:ret} and \ref{alg:eps:search} ensure that Algorithm~\ref{alg:eps} returns an $\epsilon$-optimal $y_{\epsilon}$ for which 
	the slack $c - \stdMap^\T y_{\epsilon}$ is a relative interior point of the minimal face of \eqref{eq:c_dual}. }. 
Then, if we solve the auxiliary problems \eqref{eq:red}, \eqref{eq:red_dual} in Lemma~\ref{lemma:red_sdp} associated to \eqref{eq:dual_naive}, we obtain a feasible solution to \eqref{eq:dual_naive} by item $(iii)$ of Lemma~\ref{lemma:red_sdp}. A feasible solution to \eqref{eq:dual_naive} is precisely an $\epsilon$-optimal solution to \eqref{eq:c_dual}.

The drawback of this naive approach is that for every $\epsilon$ we need to, at the very least,  solve one extra conic linear program. However, the approach using Algorithm~\ref{alg:eps} only requires solving cone membership problems which might be significantly cheaper depending on $\stdCone$.
%Let $s_i \coloneqq c - \stdMap^\T w _{i+1}$. Recalling that 
%$ f = -\stdMap^\T y_i$ holds, we have that $	c -\stdMap^\T (w_{i+1} + \alpha _i y_{i}) \in \reInt \stdFace_i^*$ if and only if
%\[
%s_i + \alpha _i f_i \in \reInt \stdFace_i^*.
%\]
%As in the proof of Theorem~\ref{theo:alg_eps}, $s_i$ can be written 
%as  $u_i + v_i $ for some $u_i \in \reInt \stdFace _{i}^*$ and
%$v_i \in \minFacePoint{f_i}{\stdFace _i^*}^{\Delta \perp}$. Equivalent, we need to find $\alpha _i > 0$ such that 
%\[
%u_i + v_i + \alpha _i f_i \in \reInt \stdFace_i^*
%\]

%We conclude this subsection remarking that a curious feature of Algorithm~\ref{alg:eps} is that its cost is not affected by $\epsilon$. If we indeed have access to the necessary inputs to Algorithm~\ref{alg:eps}

\subsection{Distinguishing between weak and strong infeasibility}\label{sec:inf}
When \eqref{eq:c_dual} is infeasible, it can be either strongly or weakly infeasible. Strong infeasibility is relatively straightforward to analyze. Indeed, by Proposition~\ref{prop:infeas}, if we wish to show that \eqref{eq:c_dual} is strongly infeasible, it is enough to exhibit some $x \in \stdCone \cap \ker \stdMap$ such that $\inProd{c}{x} = -1$.
Therefore, in order to prove that \eqref{eq:c_dual} is strongly infeasible we need to 
solve an CLP feasibility problem. In particular, 
when $\stdCone = \PSDcone{n}$, this can be done in 
at most $n+1$ calls to  $\intOracle$, by Proposition~\ref{prop:fra_cost}.

When \eqref{eq:c_dual} is weakly infeasible, the situation is far more complicated. In order to prove that \eqref{eq:c_dual} is weakly infeasible, we have to prove that \eqref{eq:c_dual} is infeasible (which can also be done by Algorithm~\ref{alg:fra}) and that the feasibility problem associated to strong infeasibility is infeasible, i.e., we have to show that there is no solution to 
\[
\mathrm{find }\quad  x \in \stdCone^* \cap \{x \in \ker \stdMap \mid \inProd{c}{x} = -1\}.
\]

In this subsection, we will use the techniques of Section~\ref{sec:almost} to analyze weak infeasibility. This is not surprising because 
weak infeasibility and non-attainment of optimal solutions are closely related as we will see in a moment. In fact, let $\stdInt \in \reInt \stdCone$ and consider the following problem and its primal counterpart.
\begin{multicols}{2}
	\noindent\begin{align}
	\underset{x}{\inf} & \quad \inProd{c}{x} \tag{P-Feas} \label{eq:feas_p}\\ 
	\mbox{subject to} & \quad \stdMap x = 0 \nonumber \\ 
	& \quad \inProd{\stdInt}{x} = 1 \nonumber\\ 
	&\quad x \in \stdCone^* \nonumber
	\end{align}
	\noindent
	\begin{align}
	\underset{t,y}{\sup} & \quad t  \tag{D-Feas}\label{eq:feas_d} \\ 
	\mbox{subject to} & \quad c - t\stdInt - \stdMap ^\T y \in \stdCone \nonumber.
	\end{align}	
\end{multicols}

Before we proceed, we need two preliminary results.
\begin{proposition}\label{prop:feas_d_sf}
If 	$(c+ \matRange \stdMap^\T ) \cap \spanVec \stdCone \neq \emptyset $, then \eqref{eq:feas_d} is strongly feasible. 
If $(c+ \matRange \stdMap^\T ) \cap \spanVec \stdCone = \emptyset $, then 
\eqref{eq:c_dual} is strongly infeasible.
\end{proposition}
\begin{proof}
Suppose that $(c+ \matRange \stdMap^\T ) \cap \spanVec \stdCone \neq \emptyset $ and let $y$ and $s$ be such that 
\[
s = c- \stdMap^\T y \in  (c+ \matRange \stdMap^\T ) \cap \spanVec \stdCone.
\]
Then, since $\stdInt \in \reInt \stdCone$,  there exists $\alpha > 0$ such that $\stdInt + \alpha s \in \reInt \stdCone$. Since $\stdCone$ is a cone, 
we have $\stdInt/\alpha + s \in \reInt \stdCone$. Therefore, 
$(t,y) \coloneqq (-\alpha,y)$ is a solution for \eqref{eq:feas_d} for which the corresponding slack $c - t\stdInt - \stdMap ^\T y$ belongs to 
$\reInt \stdCone$, thus showing that \eqref{eq:feas_d} is strongly feasible.

Next, suppose that $(c+ \matRange \stdMap^\T ) \cap \spanVec \stdCone = \emptyset $. Because $c+ \matRange \stdMap^\T$ and $\spanVec \stdCone$ are polyhedral sets, this implies that
\[
\dist(c+ \matRange \stdMap^\T, \spanVec\stdCone) > 0,
\]
e.g., see Corollary~19.3.3 and Theorem~11.4 in \cite{rockafellar}.
In particular, we must have $\dist(c+ \matRange \stdMap^\T, \stdCone) > 0$ as well, thus showing that \eqref{eq:c_dual} is strongly infeasible.
\end{proof}

\begin{lemma}\label{lem:wi}
If $\dist(c + \matRange \stdMap^\T, \stdCone) = 0$.
Then, $(c + \matRange \stdMap^\T)\cap \spanVec \stdCone \neq \emptyset $ and 
\[
\dist((c + \matRange \stdMap^\T)\cap \spanVec \stdCone, \stdCone) = 0.
\]	
\end{lemma}
\begin{proof}
For simplicity of notation, let $\stdSpace \coloneqq \matRange \stdMap^\T$.
Since $\dist(\stdSpace+c, \stdCone) = 0$, we also have $\dist(\stdSpace+c, \spanVec\stdCone) = 0$. However, because $\stdSpace+c$ and $\spanVec \stdCone$ are polyhedral sets, this implies that  	$(\stdSpace+c)\cap \spanVec \stdCone \neq \emptyset $, see Corollary~19.3.3 and Theorem~11.4 in \cite{rockafellar}. So, let $\hat c \in (\stdSpace+c)\cap \spanVec \stdCone $. We have
\[
(\stdSpace+c)\cap \spanVec \stdCone = (\stdSpace \cap \spanVec\stdCone) + \hat c.
\]
For the sake of obtaining a contradiction, assume 
that $\dist((\stdSpace+c)\cap \spanVec \stdCone, \stdCone) > 0$. By item $(ii)$ of Proposition~\ref{prop:infeas}, there exists $x$ such that 
\[
\inProd{\hat c}{x} = -1,\quad x \in \stdCone^* \cap ((\stdSpace\cap \spanVec \stdCone)^\perp ).
\]
Therefore, $x$ satisfies $x = u + v$, where $u \in \stdSpace^\perp$ and 
$v \in  \stdCone ^\perp$. Recall that, since $\hat c \in \stdSpace+c$, there exists $l \in \stdSpace$ such that $\hat c = l +c$.
We have
\[
-1 = \inProd{\hat c}{x} = \inProd{l+c}{u+v} = \inProd{c}{u},
\]
because $l + c \in \spanVec \stdCone$, $v \in  \stdCone^\perp$ and $u \in \stdSpace^\perp$. Furthermore, $u \in \stdCone^*$, because 
$u = x-v$ and $ \stdCone ^\perp \subseteq \stdCone^*$.
Gathering all we have shown, we obtain
\[
\inProd{c}{u} = -1, \quad u \in \stdCone^* \cap \stdSpace^\perp.
\]
Again, by item $(ii)$ of Proposition~\ref{prop:infeas}, we conclude 
that $\dist(\stdSpace+c,\stdCone) > 0$, which contradicts our assumptions.

\end{proof}
%\begin{lemma}
%Let $\stdSpace \subseteq \ambSpace$, $c \in \ambSpace$ and $\stdCone$ be a closed convex cone such that $\lineality \stdCone$ is strictly contained 
%in $\stdCone$. Suppose additionally that $\stdSpace+c \subseteq \spanVec \stdCone$ and $\dist(\stdSpace+c, \stdCone) = 0$.	 Then,
%\[
%\dist((\stdSpace+c), \stdCone\cap (\lineality \stdCone^\perp)) = 0.
%\]
%\end{lemma}
%Next, we define the function 
%\[
%\lambda _{\min}(x) \coloneqq \sup \{t \mid x - t\stdInt \in \stdCone\}. 
%\]
%Renegar proved in Proposition~2.1 of \cite{RE16} that 
%$\lambda(\cdot)$ is concave and Lipschitz continuous over $\spanVec \stdCone$, under the assumption that $\stdCone$ is pointed.
%We summarize below the key facts

\begin{proposition}\label{prop:feas_d}
Denote by $\opt{\ref{eq:feas_d}}$ the optimal value of \eqref{eq:feas_d}.
Then,
\begin{enumerate}[$(i)$]
	\item $\opt{\ref{eq:feas_d}} > 0 $ if and only if \eqref{eq:c_dual} is strongly feasible.
	\item $\opt{\ref{eq:feas_d}} = 0 $ if and only if \eqref{eq:c_dual} is in weak status (i.e., either weakly infeasible or weakly feasible).
	\item $\opt{\ref{eq:feas_d}} = 0 $ and is not attained if and only if 
	\eqref{eq:c_dual} is weakly infeasible
	\item  $\opt{\ref{eq:feas_d}} < 0 $ if and only if \eqref{eq:c_dual} is strongly infeasible.
\end{enumerate}
\end{proposition}
\begin{proof}
\begin{enumerate}[$(i)$]
	\item First, suppose that \eqref{eq:c_dual} is strongly feasible and let $s,y$ be such that 
	\[
	s = c - \stdMap^\T y \in \reInt \stdCone.
	\]
	By hypothesis, we have $\stdInt \in \reInt \stdCone$. By item $(\ref{laux:ri2})$ of Lemma~\ref{lemma:aux}, there exists $\alpha > 1$ such 
	that \[
	\alpha s + (1-\alpha)\stdInt \in \reInt \stdCone.
	\]
	Therefore, 
	\[
	c -t\stdInt  - \stdMap^\T y  \in \stdCone,
	\]
	where $t = (\alpha-1)/\alpha$. This shows that $\opt{\rm\ref{eq:feas_d}} > 0 $.
	
	Conversely, if $\opt{\rm\ref{eq:feas_d}} > 0 $, there exists $(t,y)$ such 
	that  $c -t\stdInt - \stdMap^\T y \in \stdCone $ with $t > 0$. By  item $(\ref{laux:ri})$ of 
	Lemma~\ref{lemma:aux}, we have $c - \stdMap^\T y \in \reInt \stdCone$.
%	\item [($iv$)$\Rightarrow$] Here we show the ``$\Rightarrow$'' implication that is, we suppose that $\opt{\ref{eq:feas_d}} < 0 $ and
%	we  prove that \eqref{eq:c_dual} must be strongly infeasible. 
%	First, it must be the case that  $\opt{\ref{eq:feas_d}}$ is finite or $\opt{\ref{eq:feas_d}} = -\infty$. If $\opt{\ref{eq:feas_d}} = -\infty$, we have  that \eqref{eq:feas_d} is infeasible. By Proposition~\ref{prop:feas_d_sf}, it follows that \eqref{eq:c_dual} is strongly infeasible. 
%	
%	Next, if $\opt{\ref{eq:feas_d}}$ is negative and finite, then \eqref{eq:feas_d} has a feasible solution $(t,y)$. Recalling that $\stdInt \in \reInt \stdCone$, we have that
%	\[
%	c - (t-1)\stdInt -\stdMap^\T y \in \reInt \stdCone.
%	\] 
%	This shows that \eqref{eq:feas_d} is strongly feasible and by  Proposition~\ref{prop:slater}, 
%	we conclude that $\opt{\ref{eq:feas_p}}$ is negative as well. This implies the existence of a feasible solution $x$ to \eqref{eq:feas_p} satisfying
%	\[
%	\inProd{c}{x} < 0, \quad \stdMap x = 0, \quad x \in  \stdCone^*.
%	\]
%	By item $(ii)$ of Proposition~\ref{prop:infeas}, we conclude that \eqref{eq:c_dual} is strongly infeasible.

	\item Suppose that \eqref{eq:c_dual} is in weak status. If $\eqref{eq:c_dual}$ is weakly feasible, then there is $y$ such 
	that $(0,y)$ is feasible for \eqref{eq:feas_d}. Therefore, $\opt{\rm\ref{eq:feas_d}} \geq 0 $. By item $(i)$, we must have 
	 $\opt{\rm\ref{eq:feas_d}} = 0 $. 
	 
	 Next, we suppose that  \eqref{eq:c_dual} is weakly infeasible. 
	 From item $(i)$, we must have $\opt{\rm\ref{eq:feas_d}} \leq 0$.
	  By Lemma~\ref{lem:wi}, we have 
	\[
	(c+\matRange \stdMap^\T)\cap \spanVec \stdCone \neq \emptyset.
	\]
	so 	 Proposition~\ref{prop:feas_d_sf} implies that \eqref{eq:feas_d} must be strongly feasible. In particular, $\opt{\rm\ref{eq:feas_d}}$ is finite.
	By Proposition~\ref{prop:slater}, \eqref{eq:feas_p} has optimal
	value equal to $\opt{\rm \ref{eq:feas_d}}$ and 
	there exists a feasible solution $x$ to \eqref{eq:feas_p} satisfying
	\[
	\inProd{c}{x} = \opt{\rm \ref{eq:feas_d}}, \quad \stdMap x = 0, \quad x \in  \stdCone^*.
	\]
	If $\opt{\rm \ref{eq:feas_d}} < 0$, we would have that \eqref{eq:c_dual} is strongly infeasible by Proposition~\ref{prop:infeas}. Since this would contradict the weak infeasibility of \eqref{eq:c_dual}, we must have $\opt{\rm \ref{eq:feas_d}} = 0$.
	This concludes the first half of item $(ii)$. 

	Now suppose that $\opt{\rm \ref{eq:feas_d}} = 0$. Then, there are sequences $\{t_k\}$, 
	$\{y_k\}$ such that $t_{k} \to 0$ and $(t_k,y_k)$ is feasible for \eqref{eq:feas_d}.
	Then, since $	c -t_{k}\stdInt  - \stdMap^\T y_{k}  \in \stdCone$ holds for every $k$, we have
	\[
	\dist(c + \matRange \stdMap^\T, \stdCone) \leq \norm{(c  - \stdMap^\T y_{k})-(c -t_{k}\stdInt  - \stdMap^\T y_{k} )}	\leq \norm{t_k \stdInt}, \quad \forall k.
	\]
	Since $t_k \to 0$, this shows that  $\dist(c + \matRange \stdMap^\T, \stdCone) = 0$ and, therefore, \eqref{eq:c_dual} is in weak status.
	
	\item From item $(ii)$, we know that \eqref{eq:c_dual} is in weak status if and only if 
	$\opt{\rm \ref{eq:feas_d}} = 0$. In particular, if \eqref{eq:c_dual} is weakly infeasible then
	the optimal value of \eqref{eq:feas_d} cannot be attained because, otherwise, we would obtain a feasible solution to \eqref{eq:c_dual}. Conversely, if $\opt{\ref{eq:feas_d}} = 0$ and is not attained, then \eqref{eq:c_dual} is in weak status and cannot be weakly feasible, so it must be weakly infeasible.
	\item[($iv$)] 	Suppose that \eqref{eq:c_dual} is strongly infeasible.
	By items $(i)$, $(ii)$ and $(iii)$, $\opt{\ref{eq:feas_d}} \geq 0$ implies that 
	\eqref{eq:c_dual} is either strongly feasible or in weak status. Therefore, 
	it must be the case that $\opt{\ref{eq:feas_d}} < 0$.
	
	Conversely, if $\opt{\ref{eq:feas_d}} < 0$, items $(i)$, $(ii)$ and $(iii)$ imply that 
	\eqref{eq:c_dual} is neither strongly feasible nor in weak status. Therefore, it must be strongly infeasible.
	
%	Therefore, there for every $\epsilon > 0$, there exists  $x_\epsilon \in \stdCone$ and $y_{\epsilon}$	such that $s_\epsilon \coloneqq c- \stdMap^\T y_{\epsilon}$ satisfies
%	\[
%	s_{\epsilon} \in \spanVec \stdCone, \quad \dist(s_\epsilon, \stdCone) = \norm{x_\epsilon - s_\epsilon} \leq \epsilon. 
%	\]
%	Because $\stdInt \in \reInt \stdCone$, there exists $t_{\epsilon}<0$ such that $\stdInt - \alpha s_{\epsilon} \in \stdCone$. Therefore,
%	\[
%	-\frac{1}{\alpha}\stdInt + s_{\epsilon} \in \stdCone.
%	\]
%	Let $t_{\epsilon} \coloneqq \sup \{t \mid -t\stdInt + s_{\epsilon }\in \stdCone \}$. Because \eqref{eq:dual} is infeasible, we must 
%	have 
%	\[
%		\limsup _{\epsilon \to 0} t_{\epsilon} \leq 0.
%	\]
%	For the sake of obtaining a contradiction suppose that 
%	$_{\epsilon \to 0} t_{\epsilon} < 0$. Therefore, 
%	there exists a sequence $\epsilon _k \to 0$ such that $t_{\epsilon _k}$ converges to 
\end{enumerate}
\end{proof}
\begin{remark}
Item $(iv)$ of Proposition~\ref{prop:feas_d} includes the possibility that 
$\opt{\rm\ref{eq:feas_d}} = -\infty$, i.e., \eqref{eq:feas_d} might be infeasible. 
This happens, for example, when $\stdCone$ is a subspace and $c + \matRange \stdMap^\T$ does not intersect $\stdCone$. However, under the hypothesis that $\stdCone$ is full-dimensional (i.e., $\spanVec \stdCone = \ambSpace$), \eqref{eq:feas_d} must always be feasible, %because $\stdInt \in \interior \stdCone$, 
see Proposition~\ref{prop:feas_d_sf}.
\end{remark}
%Note that as long as we have a way of testing membership in the cones appearing 
%in Step \ref{alg:eps:search}, we do not need to solve any additional CLPs when 
%$\epsilon$ varies.

From Proposition~\ref{prop:feas_d} we see that if \eqref{eq:c_dual} is weakly infeasible, 
we can obtain almost feasible solutions to \eqref{eq:c_dual} by constructing almost optimal solution solutions to \eqref{eq:feas_d}, which can be done through the discussion in Section~\ref{sec:almost} and Algorithm~\ref{alg:eps}. For future 
reference, we register this fact as a proposition.

\begin{proposition}[From almost optimality to almost feasibility]\label{prop:inf_alm}
Suppose that $\epsilon > 0$ and that $(t_{\epsilon},y_{\epsilon})$ is a feasible 
solution to \eqref{eq:feas_d} satisfying $0 \geq t_{\epsilon} \geq -\epsilon$. Then,
\[
\dist(c - \stdMap^\T y _{\epsilon}, \stdCone) \leq \epsilon \norm{\stdInt}.
\]
\end{proposition}
\begin{proof}
Since $(t_{\epsilon},y_{\epsilon})$ is feasible for \eqref{eq:feas_d} we have
$
 c - t_\epsilon \stdInt - \stdMap^\T y_{\epsilon} \in \stdCone.
$
Therefore, 
\[
\dist(c- \stdMap^\T y _{\epsilon}, \stdCone) \leq \norm{c- \stdMap^\T y _{\epsilon} - (c- t_\epsilon \stdInt - \stdMap^\T y _{\epsilon}) } \leq \epsilon \norm{\stdInt}.
\]
\end{proof}
To conclude this section, we present an algorithm for handling 
infeasibility of \eqref{eq:c_dual}, see Algorithm~\ref{alg:inf}. The algorithm is able to distinguish between 
weak and strong infeasibility and, for weakly infeasible problems, it returns 
almost feasible solutions. During the algorithm's run, we will need 
to apply facial reduction to \eqref{eq:feas_p}. For convenience, 
denote by $\minFace{\rm\scriptsize\ref{eq:feas_p}}$ the minimal face of $\stdCone^*$ that contains the feasible region of \eqref{eq:feas_p}. Applying 
facial reduction to \eqref{eq:feas_p} leads to the following pair of problems: 
\begin{multicols}{2}
	\noindent\begin{align}
	\underset{x}{\inf} & \quad \inProd{c}{x} \tag{\^{P}-Feas} \label{eq:feas_p2}\\ 
	\mbox{subject to} & \quad \stdMap x = 0 \nonumber \\ 
	& \quad \inProd{\stdInt}{x} = 1 \nonumber\\ 
	&\quad x \in \minFace{\rm\ref{eq:feas_p}} \nonumber
	\end{align}
	\noindent
	\begin{align}
	\underset{t,y}{\sup} & \quad t  \tag{\^{D}-Feas}\label{eq:feas_d2} \\ 
	\mbox{subject to} & \quad c - t\stdInt - \stdMap ^\T y \in (\minFace{\rm\ref{eq:feas_p}})^* \nonumber.
	\end{align}	
\end{multicols}
\noindent The pair \eqref{eq:feas_p2} and \eqref{eq:feas_d2} satisfy the following property.
\begin{proposition}\label{prop:feas_d_dfr}
Suppose \eqref{eq:feas_d} is feasible and  $\minFace{\rm\scriptsize\ref{eq:feas_p}} \neq  \emptyset$, then the pair \eqref{eq:feas_p2}, \eqref{eq:feas_d2} are both strongly feasible and their common optimal value is equal to $\opt{\ref{eq:feas_d}}$. 
Moreover, if $\minFace{\rm\ref{eq:feas_p}} =  \emptyset$ then \eqref{eq:c_dual} is strongly feasible.
\end{proposition}
\begin{proof}
Under the assumption that \eqref{eq:feas_d} is feasible, \eqref{eq:feas_d} must be, in fact, strongly feasible, by Proposition~\ref{prop:feas_d_sf}. Therefore, the minimal face of $\stdCone$ that contains the feasible region of \eqref{eq:feas_d} is $\stdCone$ itself. That is,
\[
\minFace{\rm\scriptsize\ref{eq:feas_d}} = \stdCone.
\]
Applying Theorem~\ref{theo:double} to \eqref{eq:feas_d} we conclude 
that $\opt{\rm\scriptsize\ref{eq:feas_d}}$ is finite if and only if $\minFace{\rm\scriptsize\ref{eq:feas_p}} \neq  \emptyset$. We also obtain from item~$(i)$ of Theorem~\ref{theo:double} that, if indeed $\minFace{\rm\scriptsize\ref{eq:feas_p}} \neq  \emptyset$ holds, then 
 \eqref{eq:feas_p2}, \eqref{eq:feas_d2} are both strongly feasible and their common optimal value must coincide with $\opt{\rm\scriptsize\ref{eq:feas_d}}$. 
Alternatively, from item~$(ii)$ of Theorem~\ref{theo:double}, we conclude 
that $\minFace{\rm\scriptsize\ref{eq:feas_p}} = \emptyset$ if and only if $\opt{\rm\scriptsize\ref{eq:feas_d}} = +\infty$, in which case \eqref{eq:c_dual} must be strongly feasible by Proposition~\ref{prop:feas_d}.
\end{proof}
%We can now state Algorithm~\ref{alg:inf}.

\begin{algorithm}
	\caption{Determining the infeasibility status of \eqref{eq:c_dual}\label{alg:inf}}
	\KwIn{ $\stdCone, \stdMap, b, c, \epsilon$ (\eqref{eq:c_dual} is assumed to be infeasible)} 
	\KwOut{  \texttt{Weakly Infeasible} or  \texttt{Strongly Infeasible}.
	If  \texttt{Weakly Infeasible} then $y_{\epsilon}$ such 
that $\dist(c - \stdMap^\T y_{\epsilon},\stdCone) \leq \epsilon$ is also returned.} 
	\If{$(c + \matRange\stdMap^\T)\cap \spanVec \stdCone = \emptyset$  \label{alg:inf:c1}}{\Return{} \texttt{Strongly Infeasible}\tcc*[f]{Proposition~\ref{prop:feas_d_sf}}} \label{alg:inf:infe1} 
	Apply Algorithm~\ref{alg:fra} to \eqref{eq:feas_p} and  
	let 	$(f_1,y_1), \ldots, (f_{\ell _2},y_{\ell_2})$ be the obtained corresponding reducing directions. 
	
	\eIf{
		Algorithm~\ref{alg:fra} returned \texttt{Infeasible} (i.e., $\minFace{\rm\scriptsize\ref{eq:feas_p}} = \emptyset$) \label{alg:inf:c2}}{ \Return{} \texttt{Strongly Feasible} \label{alg:inf:infe2}\tcc*[f]{\footnotesize{Proposition~\ref{prop:feas_d_dfr}, impossible by assumption}}}
	{
		Solve the pair \eqref{eq:feas_p2} and \eqref{eq:feas_d2} and denote the optimal value by $\opt{}$ and the optimal solution of \eqref{eq:feas_d2} by $(t^*,y^*)$. \label{alg:inf:o}
		
		\uIf{$\theta < 0$ \label{alg:inf:c3} }{\Return{} \texttt{Strongly Infeasible}\tcc*[f]{{Proposition~\ref{prop:feas_d}}}} \label{alg:inf:infe3}
		\uElseIf{$\theta = 0$ \label{alg:inf:c4}}{
			Let $\hat t, \hat y$ be such that $c -\hat t \stdInt - \stdMap^\T \hat y \in \reInt \stdCone$.\label{alg:inf_ri}\footnote{By Proposition~\ref{prop:feas_d_sf}, $(\hat t, \hat y)$ must exist because if we have reached this line, \eqref{eq:c_dual} is not strongly infeasible. 
%				To compute $(\hat t, \hat y)$, we can start with any $\hat y$ satisfying $c- \stdMap^\T\hat y \in \spanVec \stdCone$ and any negative $\hat t$ and, then, we successively double $\hat t$ until the condition is satisfied. 
			} 	
			
			Apply Algorithm~\ref{alg:eps} to \eqref{eq:feas_d} using as 
			input $(f_1,y_1), \ldots, (f_{\ell _2},y_{\ell_2})$, $(\hat t, \hat y)$, $(t^*,y^*)$, $\epsilon/{\norm{\stdInt}}$. Let $y_{\epsilon}$ be the 
			output of Algorithm~\ref{alg:eps}.
			
			\Return{} $y_{\epsilon}$ and \texttt{Weakly Infeasible}
		
		} 
		
		\ElseIf{$\theta > 0$ }{\Return{} \texttt{Strongly Feasible} \tcc*[f]{\footnotesize{Proposition~\ref{prop:feas_d}, impossible by assumption}}} \label{alg:inf:inf4}
	}
\end{algorithm}
 
\begin{proposition}[Algorithm~\ref{alg:inf} is correct]\label{prop:inf_cor}
The following hold.
\begin{enumerate}[$(i)$]
	\item Assuming that \eqref{eq:c_dual} is infeasible, Algorithm~\ref{alg:inf} correctly identifies whether \eqref{eq:c_dual} is strongly or weakly infeasible.
	\item  When \eqref{eq:c_dual} is weakly infeasible, the output of Algorithm~\ref{alg:inf} is indeed an $\epsilon$-feasible solution.
	\item When $\stdCone = \PSDcone{n}$, Algorithm~\ref{alg:inf} is implementable with  $O(n)$ calls to $\intOracle$. % check this
\end{enumerate}
\end{proposition}
\begin{proof}
The correctness of Algorithm~\ref{alg:inf} follows from Proposition~\ref{prop:feas_d_dfr} and the correctness of 
Algorithm~\ref{alg:eps}. We will now explain some details.

By Propositions~\ref{prop:feas_d} and \ref{prop:feas_d_dfr}, to distinguish between weak and strong infeasibility it is enough to check 
the following three items: whether  \eqref{eq:feas_d} is feasible or not;  whether $\minFace{\ref{eq:feas_p}}$ is empty or not; whether the optimal value of the pair \eqref{eq:feas_p2} and \eqref{eq:feas_d2} is negative or zero. These three items are checked at Lines~\ref{alg:inf:c1}, \ref{alg:inf:c2}, \ref{alg:inf:c3} and \ref{alg:inf:c4} of Algorithm~\ref{alg:inf}.

At Line~\ref{alg:inf:c1}, if $(c + \matRange\stdMap^\T)\cap \spanVec \stdCone = \emptyset$, then the optimal value of \eqref{eq:feas_d} is $-\infty$ and \eqref{eq:c_dual} is strongly infeasible by Proposition~\ref{prop:feas_d_sf}.

However, if we progress until the check of Line~\ref{alg:inf:c2}, \eqref{eq:feas_d} must be strongly feasible, also by Proposition~\ref{prop:feas_d_sf}. By this point, facial reduction is applied to \eqref{eq:feas_p} and if $\minFace{\ref{eq:feas_p}} = \emptyset$, then Proposition~\ref{prop:feas_d_dfr} tells us that \eqref{eq:c_dual} is strongly feasible. As we are assuming that \eqref{eq:c_dual} is infeasible, this should not happen.

If the algorithm reaches Line~\ref{alg:inf:c3} then \eqref{eq:feas_d} is feasible, $\minFace{\ref{eq:feas_p}} \neq \emptyset$ and, therefore, 
Proposition~\ref{prop:feas_d_dfr} applies.
By Proposition~\ref{prop:feas_d}, if $\theta < 0$ it must be the case that \eqref{eq:c_dual} is strongly infeasible. If $\theta = 0$ then \eqref{eq:c_dual} is weakly infeasible and the correctness of Algorithm~\ref{alg:eps} shows that $(t_{\epsilon},y_{\epsilon})$ is indeed an $\epsilon/\norm{\stdInt}$ optimal solution to \eqref{eq:feas_d} which, by Proposition~\ref{prop:inf_alm} implies that $y_{\epsilon}$ must be an $\epsilon$-feasible solution to \eqref{eq:c_dual}.

Finally, suppose that $\stdCone= \PSDcone{n}$. The only lines where SDPs need to be solved are when Algorithm~\ref{alg:fra} is invoked and 
at Line~\ref{alg:inf:o}. Algorithm~\ref{alg:fra} and Algorithm~\ref{alg:eps} require at most $n+1$  calls to $\intOracle$ each so, we only need to check 
that we can indeed solve the SDP at Line~\ref{alg:inf:o} with $\intOracle$. We note that 
if we reach Line~\ref{alg:inf:o}, then \eqref{eq:feas_d} is feasible and 
$\minFace{\ref{eq:feas_p}} \neq  \emptyset$ which, by Proposition~\ref{prop:feas_d_dfr} implies that the pair \eqref{eq:feas_p2} and \eqref{eq:feas_d2} are both strongly feasible and can indeed be solved by $\intOracle$. Therefore, Algorithm~\ref{alg:inf} can be implemented with $O(n)$ calls to $\intOracle$.
\end{proof}
To conclude, we note that, when $\stdCone = \PSDcone{n}$,
 the problem at Line~\ref{alg:inf_ri} can be solved by invoking $\intOracle$ (which would not affect the $O(n)$ complexity of Algorithm~\ref{alg:inf}), but that is not necessary.
Let $\lambda _{\min}(\cdot)$ denote the minimum eigenvalue function and recall $\lambda_{\min}(U+V) \geq \lambda _{\min}(U) + \lambda_{\min} (V)$ always holds for $U,V \in \S^n$.  With that, if $\alpha$ is positive then
 \[
 \alpha > -\frac{\lambda_{\min}(c-\stdMap^\T \hat y)}{\lambda _{\min}(\stdInt)} \Rightarrow c-\stdMap^\T \hat y + \alpha \stdInt \in \PSDcone{n}.
 \]
So, with two minimum eigenvalue computations, we can solve the problem in 
Line~\ref{alg:inf_ri} of Algorithm~\ref{alg:inf}. As in Section~\ref{sec:eps_cost}, a strategy of starting with some negative $t$ and doubling it at each step would also work.

\section{Completely solving {\eqref{eq:c_dual}}
}
\label{sec:comp}
Using the techniques described in Sections~\ref{sec:fr} and \ref{sec:dfr}, we can now present  a general algorithm for completely solving \eqref{eq:c_dual}, in the sense of 
Definition~\ref{def:comp_solve}. In particular, when $\stdCone = \PSDcone{n}$, we can completely solve \eqref{eq:dual} through polynomially many calls to $\intOracle$. For ease of reference, we write down below again some of the auxiliary problems that are referenced in Algorithm~\ref{alg:comp} below.
\begin{multicols}{2}
	\noindent\begin{align}
	\underset{x}{\inf} & \quad \inProd{c}{x} \tag{\^{P}}\\ 
	\mbox{subject to} & \quad \stdMap x = b \nonumber \\ 
	&\quad x \in (\minFaceD)^* \nonumber
	\end{align}
	\noindent
	\begin{align}
	\underset{y}{\sup} & \quad \inProd{b}{y}  \tag{\^{D}} \\ 
	\mbox{subject to} & \quad c - \stdMap ^\T y \in \minFaceD. \nonumber
	\end{align}	
\end{multicols}
\begin{multicols}{2}
	\noindent\begin{align}
	\underset{x}{\inf} & \quad \inProd{c}{x} \tag{{P}*}\\ 
	\mbox{subject to} & \quad \stdMap x = b \nonumber \\ 
	&\quad x \in \minFace{\text{\ref{eq:fr_primal}}} \nonumber
	\end{align}%
	\noindent
	\begin{align}%
	\underset{y}{\sup} & \quad \inProd{b}{y}  \tag{D*}\\ 
	\mbox{subject to} & \quad c - \stdMap ^\T y \in (\minFace{\text{\ref{eq:fr_primal}}})^*.\nonumber
	\end{align}	
\end{multicols}
\noindent Here, we recall that $\minFaceD$ is the minimal face of $\stdCone$ that contains $(c+\matRange\stdMap^\T)\cap \stdCone$ and $\minFace{\text{\ref{eq:fr_primal}}}$ is the minimal face of $(\minFaceD)^*$ that contains the feasible region of \eqref{eq:fr_primal}.
We also recall that, by Theorem~\ref{theo:double}, $\dOpt$ is finite if and only if \eqref{eq:feas_d} is feasible and $\minFace{\text{\ref{eq:fr_primal}}} \neq \emptyset$. In this case, 
\eqref{eq:dfr_primal} and \eqref{eq:dfr_dual} are both strongly feasible 
and, when $\stdCone = \PSDcone{n}$, they can be solved by invoking $\intOracle$. By doing so, we are able to obtain the dual optimal value 
$\dOpt$. Checking whether $\dOpt$ is attained can be done by solving the following feasibility problem. 
	\begin{align}%
\mathrm{ find } & \quad  y  \tag{D-OPT} \label{eq:d_opt}\\ 
\mbox{subject to} & \quad c - \stdMap ^\T y \in \minFaceD, \quad \inProd{b}{y} = \dOpt.\nonumber
\end{align}	
Let $L:\Re^m\to \Re^m$ be an affine map (that is, $L-u$ is linear for some $u \in \Re^m$) such that 
\[
\matRange L = \{y\mid \inProd{b}{y} = \dOpt \}
\]
In particular,  $\inProd{b}{L(\hat y)} = \dOpt$ holds for every $\hat y$. With that we can put \eqref{eq:d_opt} in ``dual standard format'' as follows
	\begin{align}%
\mathrm{ find } & \quad  \hat y  \tag{D-OPT-STD} \label{eq:d_opt_std}\\ 
\mbox{subject to} & \quad c - \stdMap ^\T (L(\hat y)) \in \minFaceD. \nonumber
\end{align}	
We observe that \eqref{eq:d_opt} is feasible if and only if 
\eqref{eq:d_opt_std} is feasible\footnote{If $y$ is feasible for 
\eqref{eq:d_opt}, then $y \in \matRange L$, so there exists $\hat y$ such that $L(\hat y)=y$. Reciprocally, if $\hat y$ is feasible for 
\eqref{eq:d_opt_std} then $L\hat y$ is feasible for \eqref{eq:d_opt}.}.
Once \eqref{eq:d_opt_std} is solved (for example, with Algorithm~\ref{alg:fra}) and a solution $\hat y^*$ is obtained, a solution to \eqref{eq:d_opt} is obtained by letting $y^* = L\hat y^*$.
Of course, it might be the case \eqref{eq:d_opt} is not feasible in the first place. Nevertheless, we now have all pieces in place, see Algorithm~\ref{alg:comp}.

\begin{algorithm}
	\caption{Completely Solving \eqref{eq:c_dual}\label{alg:comp}}
	\KwIn{ $\stdCone, \stdMap, b, c, \epsilon$} 
	Apply Algorithm~\ref{alg:fra} to \eqref{eq:c_dual} and let 
	$d_1, \ldots, d_{\ell _1}$ be the corresponding reducing directions. \label{alg:comp:fra}\\
	\eIf{Algorithm~\ref{alg:fra} returned \texttt{Infeasible}}{
		Invoke Algorithm~\ref{alg:inf}, \Return{} its outputs.
		
	}
	{
		Let $\hat s$ be such that  $\hat s \in \reInt \minFaceD$ (see output 2. of Algorithm \ref{alg:fra}) and 
		$\hat y$ such that $c - \stdMap^\T \hat y = \hat s$.
		
		Apply Algorithm \ref{alg:fra} to \eqref{eq:fr_primal}, obtain reducing directions 
		$(f_1,y_1), \ldots, (f_{\ell _2}, y_{\ell _2})$. \label{alg:comp_fra3} 
		
		\eIf{ $\minFace{\text{\ref{eq:fr_primal}}} = \emptyset$ \label{alg:com:unb}}{
			\Return{}  $\hat s$ and \texttt{Feasible Unbounded}.
		}
		{
			Solve \eqref{eq:dfr_primal} and \eqref{eq:dfr_dual} and
			 obtain $\dOpt$ and optimal solutions $y^*,x^*$ to \eqref{eq:p_min_d} and \eqref{eq:p_min},
			respectively. \label{alg:comp_o}\\
			Apply Algorithm~\ref{alg:fra} to \eqref{eq:d_opt_std}. \label{alg:comp_fra2}
			
			\eIf{Algorithm~\ref{alg:fra} returned \texttt{Infeasible}}{ 
				Use Algorithm~\ref{alg:eps} with $f_1, \ldots, f_{\ell _1},\hat y, y^*, \epsilon$ as inputs and \Return{} $y_\epsilon$ (the output of Algorithm~\ref{alg:eps}), $\dOpt$ and \texttt{Feasible Unattained}.
			}
			{
				Let $(y,s)$ be the feasible solution returned by Algorithm~\ref{alg:fra}.	
				
				\Return{} $y$, $\dOpt$ and \texttt{Feasible Attained}. 
			}
		}
	}
\end{algorithm}

\begin{theorem}[Algorithm~\ref{alg:comp} is correct]\label{theo:comp_correct}
	Algorithm~\ref{alg:comp}  completely solves \eqref{eq:c_dual}. That is, it correctly determines whether 	\eqref{eq:c_dual} is feasible or not. If \eqref{eq:c_dual} is infeasible, Algorithm~\ref{alg:comp} distinguishes between weak and strong infeasibility and, in case of weak infeasibility, an $\epsilon$-feasible solution is returned.
	If \eqref{eq:c_dual} is feasible, Algorithm~\ref{alg:comp} computes the optimal value of
	\eqref{eq:c_dual}. If the optimal value is finite and attained, an optimal solution is returned. If the optimal value is finite but not attained, an $\epsilon$-optimal solution is returned. If the optimal value is $+\infty$, a feasible solution is returned.
	
	In addition, if $\stdCone = \PSDcone{n}$, then \eqref{eq:c_dual} can be implemented with $O(n)$ calls to the oracle $\intOracle$.
\end{theorem}
\begin{proof}
	To prove the result we gather everything we have done so far.	We consider the following cases.
	\begin{enumerate}
		\item \emph{\eqref{eq:c_dual} is infeasible}. The correctness of Facial Reduction and Algorithm \ref{alg:fra} implies that if \eqref{eq:c_dual} is infeasible, 
		then this will be correctly detected after Line~\ref{alg:comp:fra}.
		Furthermore, the correctness of Algorithm~\ref{alg:inf} (Proposition~\ref{prop:inf_cor}) ensures 
		that weak infeasibility and strong infeasibility will be correctly detected. And, in case of weak infeasibility, an $\epsilon$-feasible solution will be returned.
		
		\item \emph{\eqref{eq:c_dual} is feasible but unbounded}. 
		If the algorithm advances until Line~\ref{alg:com:unb}, it is because \eqref{eq:c_dual} is feasible and, in particular, $\minFaceD$ is not empty. In this case, we are under 
		the hypothesis of Theorem~\ref{theo:double}.
		By item $(ii)$ of Theorem \ref{theo:double}, we have $\dOpt = +\infty$ if and only if 
		$\minFace{\text{\ref{eq:fr_primal}}} = \emptyset$.
		\item  \emph{\eqref{eq:c_dual} is feasible, $\dOpt$ is finite but not attained.} In that case, when Algorithm~\ref{alg:fra} is invoked at Line~\ref{alg:comp_fra2},  it is correctly detected that \eqref{eq:d_opt_std} is infeasible and Algorithm \ref{alg:eps} correctly constructs 
		an $\epsilon$-optimal solution.
		\item  \emph{\eqref{eq:c_dual} is feasible, $\dOpt$ is finite and attained.} In that case, when Algorithm~\ref{alg:fra} is invoked at Line~\ref{alg:comp_fra2}, a feasible solution to \eqref{eq:d_opt_std} will be obtained, which corresponds to an optimal solution to \eqref{eq:c_dual}.
	\end{enumerate}
For the last part of the proof, suppose that $\stdCone = \PSDcone{n}$.
Algorithm~\ref{alg:comp} directly invokes facial reduction (Algorithm~\ref{alg:fra}) at most 3 times (Lines~\ref{alg:comp:fra}, \ref{alg:comp_fra3} and \ref{alg:comp_fra2}). It also directly invokes Algorithm~\ref{alg:eps} and Algorithm~\ref{alg:inf} at most one time, each. 
The only other time where SDPs need to be solved 
is at Line~\ref{alg:comp_o} where  we need to solve the SDPs \eqref{eq:dfr_primal} and \eqref{eq:dfr_dual}, which are both strongly feasible (item $(i)$ of  Theorem~\ref{theo:double}) and therefore can be solved by a single call to $\intOracle$. 
By Propositions~\ref{prop:fra_cost}, \ref{prop:inf_cor} and the discussion in Section~\ref{sec:eps_cost}  we have that Algorithms~\ref{alg:fra}, \ref{alg:eps} and \ref{alg:inf} can also be implemented with $O(n)$ calls to $\intOracle$, so the same must be true of Algorithm~\ref{alg:comp}.
\end{proof}
See Appendix~\ref{app:ex} for an example where Algorithm~\ref{alg:comp} is applied to an instance that has a finite nonzero duality gap and unattainment at both primal and dual sides.

\section{Comparison with other approaches}\label{sec:further}
One of the features of this work is the interior point oracle $\intOracle$ and its application to the complete solvability of semidefinite programs.
Related to this task, the usage of reducing directions to the  construction of almost optimal solutions and almost feasible solutions for general conic linear programs seems to be new, although some of those ideas were present in our previous works on semidefinite programming~\cite{lourenco_muramatsu_tsuchiya} and second order cone programming~\cite{LMT16}. 
Technical results on facial reduction and double facial reduction such as Proposition~\ref{prop:feas_changes}, Theorems~\ref{theo:min_changes} and \ref{theo:double} seem to be novel as well.
Nevertheless, the idea of completely solving a problem in some sense is not necessarily new and, in this section, we compare our approach with other proposals in the literature that had similar goals.

In section 5.10 of \cite{KTR00}, de Klerk, Terlaky and Roos have  described 
 a possible sequence of steps to  solve \eqref{eq:dual}.
Their tool of choice is a self-dual embedding strategy of the original pair \eqref{eq:primal} 
and \eqref{eq:dual}. As we mentioned before, in the absence of both primal and dual strong feasibility, 
the embedded problem might fail to reveal the optimal value of the original problem or detect infeasibility/nonattainment. 
To account for that, they go for a second step, where they consider an embedded problem using Ramana's dual. The Ramana's dual $(P_R)$
is a substitute for \eqref{eq:primal} and they consider the pair formed by $(P_R)$ and its 
dual $(D_\text{cor})$, which is a ``corrected'' version of $\eqref{eq:dual}$. The pair $(P_R,D_\text{cor})$ can 
then be solved by their embedding strategy to find $\dOpt$. As the embedded problem is both primal and dual 
strongly feasible, it is possible to invoke  $\intOracle$ to solve it. However, if the 
solution given by  $\intOracle$ is not of maximum rank at both steps, their strategy might not work. 
We should mention that they do show in detail how to build an interior 
point method suitable for their approach.  Our analysis, on the other hand, is completely agnostic to the inner 
workings of the interior point oracle and no assumption is made on the optimal solutions 
returned by  $\intOracle$. 

%Nevertheless, missing from their analysis is how one can 
%recover a solution to \eqref{eq:dual} from a solution $(D_\text{cor})$ or to check that this is not possible. 
%Moreover, it is not clear from their approach how to obtain points close to optimality in case of unattainment, or 
%points close to feasibility in case of weak infeasibility. 
As our approach does not rely on Ramana's dual, our analysis
is easily generalizable to other classes of cones. Indeed, Algorithm~\ref{alg:comp} is valid and correct for any closed convex cone $\stdCone$. We remark that although there is a strong 
connection between Ramana's dual and facial reduction \cite{ramana_strong_1997,pataki_strong_2013}, no 
similar construction is known for any other class of cones. For example, following Pataki's approach 
in \cite{pataki_strong_2013}, one could formulate an alternative dual system for a second order cone programming problem.
Such a system would have many of the properties that Ramana's dual has, but it is not clear whether that system 
can be expressed via second order cone constraints.

Permenter, Friberg and Andersen present in \cite{PFA17} a very elegant approach for general conic linear programming based on self-dual embeddings and they are able to  achieve most of the goals included in Definition~\ref{def:comp_solve}. They showed that the relative interior of the set of solutions to a certain self-dual embedding of the pair \eqref{eq:c_primal} and \eqref{eq:c_dual} will reveal reducing directions for \eqref{eq:c_primal} and \eqref{eq:c_dual} under certain circumstances, 
see Corollary 3.3 in \cite{PFA17}. They used this property to present an algorithm for solving \eqref{eq:c_primal} and \eqref{eq:c_dual} while identifying several pathologies, see Algorithms~1 and 2 in \cite{PFA17}. 

We remark that even if \eqref{eq:c_primal} and/or \eqref{eq:c_dual} are not strongly feasible, it could still be the case that the duality gap is zero and both problems are attained. In this case, certain self-dual embeddings might recover optimal solutions to the pair \eqref{eq:c_primal}, \eqref{eq:c_dual} even in the absence of strong feasibility. Indeed, a crucial advantage of the approach in \cite{PFA17} is that a facial reducing step is performed only if it strictly necessary in order to recover zero duality gap and attainment, see item 1.~of Theorem~4.1 therein. As such, the approach in \cite{PFA17} regularizes a problem only if needed.

However, the main drawback in \cite{PFA17} seems to be the fact that it requires a \emph{relative interior} solution to their self-dual embedding, which is a stronger requirement than our assumption 
of having access to $\intOracle$, since $\intOracle$ is allowed to return any optimal solution.
While relative interior optimal solutions might be obtainable via interior point algorithms, this is not necessarily true for other methods.
Nevertheless, although our algorithm is more general, 
the approach in \cite{PFA17} seems to be more likely to lead to a practical implementation than ours, especially in conjunction with interior point methods. Indeed, the numerical experiments 
in Section~5 of \cite{PFA17} suggest that even when reducing directions are computed 
inexactly there are cases where they are still useful for analyzing the problem, although sometimes these approximate directions can also lead to 
incorrect conclusions.
  
Finally, we remark on  the difference between the double reformulation proposed by Pataki \cite{P18} and our double facial reduction of Section~\ref{sec:dfr}.
In Definition~1 of \cite{P18}, Pataki defined that a reformulation of \eqref{eq:primal} and \eqref{eq:dual} corresponds to the SDP primal and dual pair obtained by applying certain elementary operations to \eqref{eq:primal} and \eqref{eq:dual}. These elementary operations preserve the properties of the original problem such as duality gaps and whether the optimal value is attained or not. In simplified terms, Pataki showed in Theorem~4 of \cite{P18} that \eqref{eq:dual} can be ``doubly  reformulated'' as
	\begin{align}
	\underset{y}{\sup} & \quad \inProd{b'}{y} \label{eq:double_dual} \tag{SDP-Ref} \\ 
	\mbox{subject to} & \quad c' - \sum _{i=1}^m \stdMap _i' y_i \in \PSDcone{n} , \nonumber
	\end{align}
	where $\stdMap_i \in \S^n$ for all $i$ and
	in such a way that $c'$ belongs to  $\reInt \minFaceD$. Furthermore, 
	for some $\ell$, $(c',\stdMap_1',\ldots, \stdMap _\ell')$ can be used to obtain the minimal face associated to the so-called ``homogeneous dual'' of \eqref{eq:double_dual}.
	So, this double reformulation, in a sense, reveal both the minimal face of \eqref{eq:double_dual} and the minimal face associated to a homogenized version  of 
	the corresponding dual problem of \eqref{eq:double_dual}.
	
	As far as we could see, the homogeneous dual of \eqref{eq:double_dual} is related but it is quite different from either \eqref{eq:fr_primal} or \eqref{eq:dfr_primal} even when $\stdCone=\PSDcone{n}$. 
	We note that the homogeneous dual in Theorem~4 of \cite{P18} is  computed with respect the cone $\PSDcone{n}$ and not with respect the minimal face associated to 
	\eqref{eq:double_dual}. Therefore, the closest analogous of the double reformulation in our setting would be if we applied the second facial reduction to  \eqref{eq:c_primal} instead of applying to \eqref{eq:fr_primal}.
	In conclusion, it seems to us that double facial reduction and Pataki's double reformulation serve different purposes.

\section{Concluding remarks}\label{sec:conc}
In this paper, we have discussed how to use facial reduction and double facial reduction to completely solving (Definition~\ref{def:comp_solve}) a general conic linear program, under the assumption that certain auxiliary problems can be solved, see Algorithm~\ref{alg:comp} and Theorem~\ref{theo:comp_correct}.
When specialized to the particular case of semidefinite programming, these results imply that an arbitrary semidefinite program over $n\times n$ matrices can be completely solved by invoking at most $O(n)$ times an oracle that only return solutions to primal and dual strongly feasible SDPs. 
We also provided technical results on facial reduction and double facial reduction that might be of independent interest, see Sections~\ref{sec:fr} and \ref{sec:dfr}.

For limitations, drawbacks and comparison to other approaches, see Sections~\ref{sec:lim} and \ref{sec:further}. In particular, as discussed in Section~\ref{sec:lim}, in our analysis we assumed that the oracle $\intOracle$ returns an exact solution. An interesting topic of future research would be to consider the effects of impreciseness in the solutions returned by $\intOracle$.

\section*{Acknowledgements}
We thank the anonymous referee for the several helpful suggestions and comments.
We also thank Prof.~Henry Wolkowicz for bringing to our attention the paper by Abrams~\cite{Abrams75} and Prof.~G\'abor Pataki for helpful feedback.
The first author would also like to acknowledge the support and the interesting discussions with Prof.~Mituhiro Fukuda 
at the Tokyo Institute of Technology. 
Finally, this paper is dedicated to Professor Masao Iri with deep appreciation to all what he did and left for the authors, directly and indirectly.

\bibliographystyle{abbrvurl}
\bibliography{bib}

\appendix
\section{The proof of Proposition~\ref{prop:refor}}\label{app:refor}
In this appendix we discuss and prove Proposition~\ref{prop:refor}, which we restate below.
\begin{proposition*}
Let $\stdCone$ be as in \eqref{eq:cone_ref} with some orthogonal matrix $R$, some $r \leq n$, and
some linear subspace $\stdSpace \subseteq \S^n$ such that $\stdSpace\subseteq (\PSDcone{r,n})^\perp$.
Suppose that  \eqref{eq:c_primal} and \eqref{eq:c_dual} are strongly feasible. 	Then, \eqref{eq:c_primal} and \eqref{eq:c_dual} are solvable 
with a single call to $\intOracle$.
\end{proposition*}
\begin{proof}
	In what follows, let $\stdMap_1,\ldots, \stdMap  _m \in \S^n$  be such that 
	\[
	c - \stdMap^\T y = c - \sum _{i=1}^m \stdMap_i y_i, \text{ and } \quad \stdMap x = b \Leftrightarrow \inProd{\stdMap _i }{x} = b_i, i = 1,\ldots, m.
	\]
	
	Since $\stdCone$ is as in \eqref{eq:cone_ref}, we have
	\begin{align}
	R\stdCone R^\Tr = \PSDcone{r,n} \oplus \stdSpace,&\quad \reInt (R\stdCone R^\Tr) = (\reInt \PSDcone{r,n}) \oplus \stdSpace,  \label{eq:ref_cone}\\
	(R\stdCone R^\Tr)^* = \PSDcone{r,n} \oplus (\stdSpace^\perp \cap (\PSDcone{r,n})^\perp),& \quad \reInt (R\stdCone R^\Tr)^*= (\reInt \PSDcone{r,n}) \oplus (\stdSpace^\perp \cap (\PSDcone{r,n})^\perp). \label{eq:ref_cone_d}
	\end{align}
	Furthermore, 
	\begin{equation}\label{eq:face_ri}
	\reInt \PSDcone{r,n} =  \left\{ \begin{pmatrix}U & 0 \\ 0 & 0 \end{pmatrix}  \in \S^n \relmiddle{\vert} U \in \PDcone{r} \right\},
	\end{equation}
	where $\PDcone{r}$ denotes the set of $r\times r$ real symmetric positive definite matrices.
		Let $\S^{r,n}$ denote the span of $\PSDcone{r,n}$.
	The proof is divided in four cases and we will show that each case leads back 
	to the previous one.
		
	\noindent\fbox{\textbf{Case 1.} $\stdSpace = (\S^{r,n})^\perp$, i.e., $\stdCone = (\S^{r,n}_+)^*$. $R$ is the identity matrix.}
	
	Let $\pi_{r}:\S^{n} \to \S^r$ be the orthogonal projection
	that maps $x \in \S^n$ on its upper left $r\times r$ block.
	We can reformulate \eqref{eq:c_primal} and \eqref{eq:c_dual} as follows	
	\begin{multicols}{2}
		\noindent\begin{align}
		\underset{ \hat x}{\inf} & \quad \inProd{\pi_r( c)}{\hat x} \label{eq:ref_p} \tag{P}\\ 
		\mbox{subject to} & \quad   \inProd{\pi_r ( \stdMap _i)}{ \hat x } =  b_i, \quad i = 1,\ldots, m\nonumber \\ 
		&\quad  \hat x \in \PSDcone{r} \nonumber
		\end{align}
		\noindent
		\begin{align}
		\underset{ y}{\sup} & \quad \inProd{ b}{ y}  \label{eq:ref_d} \tag{D} \\ 
		\mbox{subject to} & \quad \pi_r( c) - \sum _{i=1}^m \pi_r ( \stdMap _i)  y_i \in \PSDcone{r}. \nonumber
		\end{align}	
	\end{multicols}
	We note that $s \in \stdCone$, if and only if the upper-left $r\times r$ block of $s$ is positive semidefinite. Therefore, \eqref{eq:ref_d} and \eqref{eq:c_dual} are equivalent in the sense that they share the same feasible solutions and have the same optimal value. From \eqref{eq:ref_cone}, we see that \eqref{eq:c_dual} is strongly feasible if and only if \eqref{eq:ref_d} is strongly feasible.
	
	Similarly, $x \in \stdCone^*$ if and only the upper-left $r \times r$ block of $x$ is positive semidefinite and the other entries are zero. Therefore,
	\[
	\inProd{A}{ x } =  \inProd{\pi_{r}(A)}{\pi _r(x)},\quad \forall A \in \S^n.
	\]
	Accordingly, if $x$ is a feasible solution to \eqref{eq:c_primal} such that $x \in \reInt \stdCone^*$, then $\pi_r(x) \in \reInt \PSDcone{r}$ and $\pi _r(x)$ is feasible for \eqref{eq:ref_p}. Conversely, 
	if $\hat x$ is feasible for \eqref{eq:ref_p}, then $\pi_{r}^*(\hat x)$ is feasible for \eqref{eq:ref_p}.
	Furthermore, $\hat x \in \PSDcone{r,n} \Leftrightarrow\pi_{r}^*(\hat x) \in \stdCone^*$, where $\pi_{r}^*$ is the adjoint of $\pi_r$.
	We conclude that  \eqref{eq:c_primal} is strongly feasible if and only if \eqref{eq:ref_p} is strongly feasible. Also, both problems have the same optimal values.

	Therefore, if \eqref{eq:c_dual} and \eqref{eq:c_primal} are both strongly feasible, the same is true 
	for the pair \eqref{eq:ref_p} and \eqref{eq:ref_d}. Since \eqref{eq:ref_p} and \eqref{eq:ref_d}
	are \emph{bona fide} SDPs, they can be solved with $\intOracle$. The preceding discussion shows how to 
	recover optimal solutions to \eqref{eq:c_dual} and \eqref{eq:c_primal} from the solutions to 
	\eqref{eq:ref_d} and \eqref{eq:ref_p}.
	
	\noindent\fbox{\textbf{Case 2.} $\stdSpace = \{0\}$, i.e., $\stdCone = \S^{r,n}_+$. $R$ is the identity matrix.}
	
	First, we recall that the role of primal and dual problems in conic linear programming 
	is interchangeable in the following sense. Given $\stdMap,b,c, \stdCone$, we can 
	find $\hat \stdMap, \hat b, \hat c$ such that \eqref{eq:c_dual} is ``equivalent'' to
	\begin{align}\label{eq:ref2_d}
	\underset{ s}{\inf} & \quad \inProd{ \hat c}{s} \tag{P-D}\\
	\mbox{subject to} & \quad   \inProd{  \hat \stdMap _i}{ s } =  \hat b_i, \quad i = 1,\ldots, m\nonumber \\ 
	&\quad  s \in \stdCone, \nonumber
	\end{align}
	which is a problem in ``primal format'', where the feasible solutions correspond to the feasible slacks of \eqref{eq:c_dual}.
	Only linear algebra is needed to find $\hat \stdMap,\hat b, \hat c$.
	For example, we can take $\hat \stdMap, \hat b$ to be such that 
	\[
	s \in c+ \matRange \stdMap^* \Leftrightarrow \hat A s = \hat b.
	\]
	Next, let $\hat c$ be such that 
	$\stdMap \hat c = b$. Here, $\hat c$ is not required to be feasible for \eqref{eq:c_primal}, so no SDPs need to be solved in order to obtain $\hat c$. Nevertheless, if $y$ is a feasible solution to \eqref{eq:c_dual} and 
	$s = c- \stdMap^\T y$, we have
	\begin{equation}\label{eq:case2}
	-\inProd{b}{y} = \inProd{\hat c}{-c+c-\stdMap^\T y} = -\inProd{\hat c}{c} + \inProd{\hat c}{s}.
	\end{equation}
	From \eqref{eq:case2}, we draw two conclusions. The first is that, provided that the linear system $Ax = b$ has a solution\footnote{Which it does since we are assuming \eqref{eq:c_primal} is (strongly) feasible.}, then $\inProd{b}{y} = \inProd{b}{y'}$ if $y,y'$ are associated to the same slack of \eqref{eq:c_dual}. Therefore, without ambiguity we can associate an  objective value to an slack $s$ of \eqref{eq:c_dual}.
	The second conclusion is that the optimal values of \eqref{eq:ref2_d} and \eqref{eq:c_dual} differ by 	the constant $\inProd{\hat c}{c}$. 
	
	We can now explain in which sense \eqref{eq:c_dual} and \eqref{eq:ref2_d} are equivalent. Every feasible slack of \eqref{eq:c_dual} is a feasible solution to 
	\eqref{eq:ref2_d} (and vice-versa), and their objective values differ by $\inProd{\hat c}{c}$.
	In particular, given an optimal solution $s^*$ to \eqref{eq:ref2_d}, any $y^*$ satisfying 
	$s^* = c- \stdMap^\T y^*$ will be optimal to \eqref{eq:c_dual}.
	Furthermore, \eqref{eq:c_dual} is strongly feasible if and only if \eqref{eq:ref2_d} is strongly feasible.
	
	The dual counterpart of \eqref{eq:ref2_d} is equivalent to \eqref{eq:c_primal} in an analogous fashion and the similarly cumbersome details are omitted.
	
	In {\textbf{Case~2}}, $\stdCone^*$ corresponds to the cone $\stdCone$ in \textbf{Case~1}.
	The overall conclusion is that  in order to return to \textbf{Case~1}, it is enough to reformulate \eqref{eq:c_dual} as a problem in primal format.

	\noindent\fbox{\textbf{Case 3.} $\stdSpace \subseteq (\S^{r,n})^\perp$ is arbitrary. $R$ is the identity matrix.}
	
	Let $E_1,\ldots, E_{\ell}$ be a basis for $\stdSpace$.
	In view of \eqref{eq:ref_cone}, \eqref{eq:ref_cone_d}, we have that \eqref{eq:c_dual} and \eqref{eq:c_primal} are 
	equivalent to the following pair of problems.%	
	\begin{multicols}{2}%
		\noindent\begin{align}
		\underset{  x}{\inf} & \quad \inProd{ c}{x} \label{eq:ref3_p} \tag{P}\\ 
		\mbox{subject to} & \quad   \inProd{  \stdMap _i}{  x } =  b_i, \quad i = 1,\ldots, m\nonumber \\ 
		& \quad   \inProd{  E_j}{ x } =  0, \quad j = 1,\ldots, \ell\nonumber \\ 
		&\quad  x \in (\PSDcone{r,n})^* \nonumber
		\end{align}
		\noindent
		\begin{align}
		\underset{y,t}{\sup} & \quad \inProd{ b}{ y}  \label{eq:ref3_d} \tag{D} \\ 
		\mbox{subject to} & \quad  c - \sum _{i=1}^m  \stdMap _i  y_i - \sum _{j=1}^\ell E_jt_j  \in \PSDcone{r,n} \nonumber
		\end{align}	
	\end{multicols}
	In particular, the feasible solutions and the optimal value of \eqref{eq:ref3_p} and \eqref{eq:c_primal} are the same. Furthermore, $y$ is feasible for \eqref{eq:c_dual} if and only if 
	there exists $t$ such that $(y,t)$ is feasible for \eqref{eq:ref3_d}.
	
	The pair of problems \eqref{eq:ref3_p} and \eqref{eq:ref3_d} are in the format described 
	in \textbf{Case~2}.
	
	\noindent\fbox{\textbf{Case 4.} $\stdSpace \subseteq (\S^{r,n})^\perp$ is arbitrary, $R$ is an arbitrary orthogonal matrix.}
	
	We observe that 
	\[
	c - \sum _{i=1}^m \stdMap _i y_i \in R^\Tr (\PSDcone{r,n}\oplus \stdSpace)R\quad \Leftrightarrow\quad R cR^\Tr - \sum _{i=1}^m R\stdMap _iR^\Tr y_i \in \PSDcone{r,n}\oplus \stdSpace.
	\]
	Therefore, replacing the $c$, $\stdMap_1,\ldots, \stdMap _m$ by $R cR^\Tr$, $R\stdMap _1R^\Tr, \ldots, R\stdMap _mR^\Tr$ in 
	\eqref{eq:c_dual}, \eqref{eq:c_primal} leads to an equivalent problem that falls under \textbf{Case~3}.
\end{proof}

\section{An example}\label{app:ex}
In this appendix, we show an example of the application of Algorithm~\ref{alg:comp} to a problem 
that has a nonzero duality gap and  such that  both primal and dual optimal values are not attained. 

In the following, we denote the $r\times r$ zero matrix and identity matrix by $\zM{r}$ and $I_r$, respectively.
Similar to Section~\ref{sec:oracle}, when it is clear from context, we use $0$ to  denote a zero matrix of appropriate size.
Furthermore, when an entry of a matrix is omitted, it is assumed to be zero.

%As a convention, we use ``$*$" to denote {\em wild card} in representation of symmetric matrices, i.e., 
%if $*$ appear as a diagonal entry or a diagonal block, $*$ may be any real symmetric matrix, and 
%if $*$ appear as an off diagonal entry or an off diagonal block, $*$ may be any real matrix, but symmetry
%between $ij$ entry/block and $ji$ entry/block is assumed to hold. 

Let us consider the following problem.
\begin{align}
\underset{y \in \Re^8}{\sup} & \quad -y_4-2y_6-2y_7 \label{eq:nasty_d_alt} \tag{D} \\ 
\mbox{subject to} & \quad \begin{pmatrix}
y_1   &        &      &           &              &              &      & y_3-1 \\
& y_1   &      &           &              &              &      & y_5-1 \\
&        & y_2 & y_3      &              &              &      &        \\
&        & y_3 & y_4-y_5 &              &              &      &        \\
&        &      &           & y_4         & -0.5y_8+0.5 & y_6 &        \\
&        &      &           & -0.5y_8+0.5 & y_8         & y_7 &        \\
&        &      &           & y_6         & y_7         &  0    &        \\
y_3-1 & y_5-1 &      &           &              &              &      &   0    
\end{pmatrix}  \in \PSDcone{8}. \nonumber
\end{align}
%Here, omitted entries correspond to zero. 
We also select $c, \stdMap_1, \ldots, \stdMap _8 \in \S^8$ such that $(y_1,\ldots, y_8)$ is feasible 
for \eqref{eq:nasty_d_alt} if and only if 
\begin{equation}\label{eq:nasty_slack}
c - \sum _{i=1}^8 \stdMap _i y_i \in \PSDcone{n}.
\end{equation}
The corresponding primal of \eqref{eq:nasty_d_alt} is
\begin{align}
\underset{x}{\inf} & \quad -2x_{18} -2x_{28} + x_{56} \label{eq:nasty_p}\tag{P}\\ 
\mbox{subject to} & \quad -x_{11}-x_{22}  = 0 \nonumber \\  
& \quad -x_{33} = 0 \nonumber  \\
& \quad -2x_{18} -2x_{34}  = 0 \nonumber  \\
& \quad -x_{44} -x_{55}  = -1 \nonumber  \\
& \quad -2x_{28} +x_{44}  = 0 \nonumber  \\
& \quad -2x_{57}  = -2 \nonumber  \\
& \quad -2x_{67}  = -2 \nonumber  \\
& \quad x_{56} - x_{66}  = 0 \nonumber  \\
&\quad x \in \PSDcone{8}  \nonumber. 
\end{align}

\subsection{The optimal values of \eqref{eq:nasty_p} and \eqref{eq:nasty_d_alt} and their unattainment }
We check that $\dOpt = -1$, $\pOpt = 0$ and that neither \eqref{eq:nasty_d_alt} nor its primal 
are attained.

Let $y$ be a feasible solution for \eqref{eq:nasty_d_alt}. 
At the dual side \eqref{eq:nasty_d_alt}, the $(7,7)$ and $(8,8)$ entries
are $0$, therefore any feasible $y$ must satisfy \[y_3 = y_5 = 1, \quad y_6 = y_7 = 0.\]
 In addition,  we have \[y_4 - y_5 = y_4 -1 \geq 0,\] 
 since the $(4,4)$ entry must be nonnegative.  Similarly, we have $y_8 \geq 0$.
 It follows that $\dOpt \leq -1$. However, for every $\epsilon > 0$, 
\[y_\epsilon = (0,1/\epsilon,1, 1+\epsilon,1,0,0,0 )\] 
is a feasible solution that has value 
equal to $-1-\epsilon$. This shows that $\dOpt = -1$. 

We will now show that $\dOpt$ is not attained. For the sake of obtaining a contradiction, 
suppose that we had $y_4 =1$ for some feasible $y$. Since $y_5 =1$,  this would imply that the $(4,4)$ entry is zero. This forces $y_3$ to be zero and, consequently, the $(1,8)$ entry must $-1$, which contradicts positive semidefiniteness. 

Next, we move on the primal problem \eqref{eq:nasty_p}.
Let $x$ be a feasible solution to \eqref{eq:nasty_p}. 
The  first and second equality constraints force \[x_{11} = x_{22} = x_{33} = 0,\] which implies that 
$x_{18} = x_{28} = 0$. Then, the fifth constraint implies that 
$x_{44} = 0$. Therefore, $x_{55} = 1$ by the fourth constraint. The last equality constraint forces $x_{56} = x_{66}$, 
which implies that $\pOpt \geq 0$. 
Putting everything together we have that $x$ is feasible for \eqref{eq:nasty_p} if and only if 
$x$ is positive semidefinite and has the following shape
\begin{equation}\label{eq:nasty_shape_p}
\begin{pmatrix}
\zM{4}    &           &            &             &    \\
        & 1            & x_{56}       & x_{57} &  x_{58}       \\
        & x_{56}       & x_{56}       & 1    &   x_{68}      \\
       & x_{57}       & 1            & x_{77}    &  x_{78}      \\
        & x_{58}       &  x_{68}      & x_{78}     &   x_{88}   
\end{pmatrix}.
\end{equation}
Then, because $x_{67} = 1$, we can never 
assign zero to $x_{66}=x_{56}$. However, if $x_{66}=x_{56}$ is small but positive, we can 
construct feasible points by taking 
\[
x_{57} = x_{58} = x_{68} = x_{78} = 0
\]
and taking $x_{77}$ and $x_{88}$ to be very large. 
This shows that $\pOpt = 0$ but is not attained.

\subsection{Applying Algorithm~\ref{alg:comp}}
We run Algorithm~\ref{alg:comp} with \eqref{eq:nasty_d_alt} as input.
The first step is to apply facial reduction (Algorithm~\ref{alg:fra}) to \eqref{eq:nasty_d_alt}. 

A possible 
reducing direction is 
\begin{equation*}
d_1 = \begin{pmatrix}
\zM{6}   & 0  \\
0   &  I_2    
\end{pmatrix}.
\end{equation*}
Let $\stdFace = \PSDcone{8}\cap \{d_1\}^\perp$, then
\[
\stdFace = \PSDcone{6,8} = \left\{ \begin{pmatrix} U & 0 \\ 0 & 0 \end{pmatrix} \in \S^8 \relmiddle{\vert}  U\in \PSDcone{6}\right\}.
\]
Recalling \eqref{eq:face_ri}, the relative interior of $\stdFace$ correspond to the matrices in $\stdFace$ for which $U$ is positive definite.
Therefore, we see that $\stdFace$ is the minimal face $\stdFace_{\min}^D$ of \eqref{eq:nasty_d_alt} by constructing the following feasible solution 
\begin{equation}\label{eq:hat_y}
\hat y = (\hat y_1, \hat y_2, \hat y_3, \hat y_4, \hat y_5, \hat y_6, \hat y_7, \hat y_8) = (1,2,1,2,1,0,0,1).
\end{equation}
Indeed, $\hat s=c - \stdMap^* \hat y$ is a matrix with nonzero entries
\[
\hat s_{11}=\hat s_{22}=\hat s_{44}= \hat s_{66}=1,\ \hat s_{33}=\hat s_{55}=2,\ \hat s_{34}=\hat s_{43}=1,
\]
which is a relative interior point of $\stdFace$.
Let (\^D) be the problem obtained by replacing $\PSDcone{8}$ by $\stdFace_{\min}^D =\stdFace$ in 
\eqref{eq:nasty_d_alt} (see also Section~\ref{sec:dfr_opt}).
Let (\^P) be the ``dual'' problem of (\^D), which is obtained 
by replacing $\PSDcone{8}$ by $(\stdFace_{\min}^D)^*$ in \eqref{eq:nasty_p}. Then $(\stdFace_{\min}^D)^*$ satisfies
\begin{equation}\label{newB4}
(\stdFace_{\min}^D)^* = (\PSDcone{6,8})^* = \left\{ \begin{pmatrix} U & V \\ V^\Tr & W \end{pmatrix} \in \S^8 \relmiddle{\vert} U\in \PSDcone{6}\right\}.
\end{equation}
This has the effect of relaxing the constraint ``$x \in \PSDcone{8}$'' 
to merely requiring that the upper left $6 \times 6$ block of $x$ be positive 
semidefinite.  The relative interior of $(\stdFace_{\min}^D)^*$ correspond to the matrices in 
\eqref{newB4} for which $U$ is positive definite.

At this stage, (\^D) is strongly feasible and shares the same optimal value 
with \eqref{eq:nasty_d_alt}. Because of this, the duality gap between 
(\^P) and (\^D) is zero and (\^P) has an optimal solution of value $-1$.
Indeed, let $x$ be the symmetric matrix having 
\[
x_{56}=x_{66} = -1, x_{57} = x_{67} = x_{55} = 1
\]
and let the other entries be zero. Then $x \in (\stdFace_{\min}^D)^*$ and 
is an optimal solution to (\^P).  Though it possesses an optimal solution, (\^ P) may not be strongly feasible.
We also remark that (\^D) does not have an optimal solution, since the feasible regions of (D) and (\^ D) are the same
and (D) does not have an optimal solution as we saw previously.

The next step in Algorithm~\ref{alg:comp} is applying facial reduction to 
(\^P).
We observe that (\^P) is not strongly feasible.  
The first two equality constraint in (P) and \eqref{newB4} yields that any feasible solution to (\^P) must have the upper left $3\times 3$ block equal to zero. 

Let $y^1\in \Re^8$ be such that $y^1_1 = y^1_2 = 1$ and the other entries are zero.
A reducing direction to (\^P) is given by the following matrix 
\begin{equation}
f_1 = \begin{pmatrix} I_3 & \zM{} \\
\zM{} & \zM{5}
\end{pmatrix} = -\sum _{i=1}^8 \stdMap_i y ^1_i = - \stdMap_1 - \stdMap _2, \label{eq:dir_y}
\end{equation}
where $\stdMap _i$ is as in \eqref{eq:nasty_slack}.
Let $\hat \stdFace \coloneqq (\stdFace_{\min}^D)^* \cap \{f_1\}^\perp$. $\hat \stdFace $
is a face of $(\stdFace_{\min}^D)^*$ that can be described as follows. 
\begin{equation}\label{eq:face_nasty_p}
\hat \stdFace
=\left\{\begin{pmatrix} \zM{3} & 0 & V_1 \\ 
0 & U & V_2 \\
V_1^\Tr & V_2^{\Tr} & W \end{pmatrix} \in \S^8 \relmiddle{\vert}
U \in \PSDcone{3}\right\}.
\end{equation}
The relative interior of $\hat \stdFace $ consists of the matrices in \eqref{eq:face_nasty_p} for which $U \in \PSDcone{3}$ is positive definite.
We see that  $\hat \stdFace$ is the minimal face of (\^ P) by letting $x \in \hat \stdFace$ such that 
\[
x_{45} = x_{46} = 0,\quad
x_{57} = x_{67} = 1,\quad  x_{56} = x_{66} = 0.25, \quad x_{44} = x_{55} = 0.5, \quad x_{28}=0.25,
\]

Next, let (P*) denote the problem obtained by replacing $\PSDcone{8}$ by 
$\hat \stdFace$ in \eqref{eq:nasty_p}. Let (D*) be the corresponding dual 
problem, which is obtained by replacing $\PSDcone{8}$ by $\hat \stdFace^*$ in 
\eqref{eq:nasty_d_alt}. We have  
\begin{equation}\label{eq:face_nasty_p_dual}
\hat \stdFace^* =\left \{\begin{pmatrix} W & V_1 & V_2 \\ 
V_1^{\Tr} & U & 0 \\
V_2^{\Tr} & 0 & \zM{2} \end{pmatrix} 
 \in \S^8 \relmiddle{\vert}
U \in \PSDcone{3}\right\}.
\end{equation}

By Theorem~\ref{theo:double}, we have 
$\dOpt = \opt{\text{P*}} = \opt{\text{D*}} = -1$.
And, indeed, the following solution $y^*$ is optimal to 
(D*) with $\inProd{b}{y^*} =\theta_D=-1.$
\begin{equation}\label{eq:opt_y}
y^*= (y_1,y_2,y_3,y_4,y_5,y_6,y_7,y_8) = (0,0,1,1,1,0,0,1).
\end{equation}
The nonzero entries of the slack matrix $s^*=c - \stdMap^* y^*$ are
\[
s^*_{55}=s^*_{66}= 1, s^*_{34}=s^*_{43}=1.
\]
We note that $s^* \in \hat \stdFace^*$ but $s^* \not \in \stdFace_{\min}^D$ and $s^* \not\in \PSDcone{8}$.
Thus, $y^*$ is not a feasible solution to (\^ D) nor (D).

At this stage, (P*) and (D*) are both strongly feasible and they can 
be solved by the interior point oracle, as in Step~\ref{alg:comp_o} of Algorithm~\ref{alg:comp}.
Knowing that the common optimal value of (P*) and (D*) is $-1$, Algorithm~\ref{alg:comp}
then checks whether there is an optimal solution to \eqref{eq:nasty_d_alt} by solving 
\eqref{eq:d_opt_std}. 

Since we already know that \eqref{eq:nasty_d_alt} is not attained, 
we skip this step and go to construction of an almost optimal solution to 
\eqref{eq:nasty_d_alt}. 
Let $\hat y, (f_1, y^1), y^*$ be as in \eqref{eq:hat_y}, \eqref{eq:dir_y}, \eqref{eq:opt_y}, respectively.
We feed $(f_1,y^1), \hat y, y^*$ and some $\epsilon > 0$ to 
Algorithm~\ref{alg:eps}.

Suppose that, say, $\epsilon = 0.1$. Recall that $\inProd{b}{\hat y}= -2$.  Let 
\[
\beta = \frac{\dOpt-\inProd{b}{\hat{y}} - \epsilon}{\dOpt-\inProd{b}{\hat{y}}}=0.9.
\]
The  computation done in Algorithm~\ref{alg:eps} boils 
down to finding $\alpha > 0$ such that
\[
c - \sum _{i=1}^8 \stdMap_i(\beta y^*_i + (1-\beta)\hat y + \alpha y^1) = \beta s^* + (1-\beta) \hat s + \alpha f_1 \in \reInt \minFaceD.
\]
Since $\inProd{b}{y^1} = 0$, moving in the direction of $y^1$ does not change the objective value,
and since $f_1$ is positive semidefinite, it does not violate the cone constraint, no matter how large $\alpha$ is.
Taking $\alpha = 10$ is enough for the purpose, so, we see that 
\[
\tilde y= \beta y^*_i + (1-\beta)\hat y + 10 y^1
\]
is an $\epsilon$-optimal solution to \eqref{eq:nasty_d_alt}.
Indeed, $\tilde s = c - \stdMap^* \tilde y$ is a positive semidefinite matrix whose nonzero entries are
\[
\tilde s_{11}=\tilde s_{22}=10.1, \ \tilde s_{33}=10.2, \ \tilde s_{55}=0.1,\ \tilde s_{66}=1,
\ \tilde s_{34}=\tilde s_{43}=1.
\]

\end{document}